\documentclass[sn-mathphys]{sn-jnl}

\usepackage[utf8]{inputenc} 
\usepackage{graphics}
\usepackage{float}
\usepackage{amsfonts}
\usepackage{amssymb,amscd,latexsym}
\usepackage{hyperref}
\usepackage{amsmath}
\usepackage{epsfig}
\usepackage{color}
\usepackage{caption}

\jyear{2021}%

\theoremstyle{thmstyleone}%
\newtheorem{thm}{Theorem}
\newtheorem{prop}[thm]{Proposition}%
\newtheorem{cor}[thm]{Corollary}%
\newtheorem{lema}[thm]{Lemma}%

\theoremstyle{thmstyletwo}%
\newtheorem{ex}{Example}%
\newtheorem{rmk}{Remark}%

\theoremstyle{thmstylethree}%
\newtheorem{defin}{Definition}%

\def\N{\mathbb{N}}

\def\P{\mathbb{P}}
\def\K{\mathbb{K}}

\begin{document}

\title[Counting square free monomial cremona maps]{Counting square free monomial Cremona maps}


\author*[1]{\fnm{Bárbara} \sur{Costa}}\email{barbara.costasilva@ufrpe.br}

\author[2]{\fnm{Thiago} \sur{Dias}}\email{thiago.diasoliveira@ufrpe.br}

\author[3]{\fnm{Rodrigo} \sur{Gondim}}\email{rodrigo.gondim@ufrpe.br}

\author[4]{\fnm{Ricardo} \sur{Machado}}\email{ricardo.machadojunior@ufrpe.br}

\affil*[1]{\orgdiv{Departamento de Matemática}, \orgname{Universidade Federal Rural de Pernambuco}, \orgaddress{\street{Rua Dom Manuel de Medeiros}, \city{Recife}, \postcode{52171-900}, \state{Pernambuco}, \country{Brasil}}}

\affil[2]{\orgdiv{Departamento de Matemática}, \orgname{Universidade Federal Rural de Pernambuco}, \orgaddress{\street{Rua Dom Manuel de Medeiros}, \city{Recife}, \postcode{52171-900}, \state{Pernambuco}, \country{Brasil}}}

\affil[3]{\orgdiv{Departamento de Matemática}, \orgname{Universidade Federal Rural de Pernambuco}, \orgaddress{\street{Rua Dom Manuel de Medeiros}, \city{Recife}, \postcode{52171-900}, \state{Pernambuco}, \country{Brasil}}}

\affil[4]{\orgdiv{Departamento de Matemática}, \orgname{Universidade Federal Rural de Pernambuco}, \orgaddress{\street{Rua Dom Manuel de Medeiros}, \city{Recife}, \postcode{52171-900}, \state{Pernambuco}, \country{Brasil}}}


\abstract{We give a complete list of square-free monomial Cremona maps of $P^{n-1}$, with $n\leq 6$, up to equivalence classes. We also give an algorithm to count them. Using this algorithm, we obtain in addition a complete list for $n=6$.}




\maketitle

\section{Introduction}\label{sec1}

Cremona transformations are birational automorphisms of the projective space. They were first systematically studied by L. Cremona in the 19th century, followed by Cayley and Noether. The study of the Cremona group remains a major classical topic in Algebraic Geometry. The interest in monomial Cremona transformations has gained impetus more recently (see \cite{V,GP,SV1,SV2,SV3}). 
Following the philosophy introduced in \cite{SV1,SV2} and shared by \cite{SV3,CS} we focus on the so-called ``birational combinatorics'', that is, the theory of characteristic-free rational
maps of the projective space defined by monomials, along with natural criteria for such maps to be birational. The central point of view is that the criteria must reflect the monomial data, as otherwise, one falls back in the general theory of birational maps in projective spaces. \\

Our main tool from combinatorics is graph theory and the theory of clutters naturally associated with a finite set of square-free monomials of the same degree. 
The determinantal principle of birationality proved in \cite{SV2} and stated here as Proposition \ref{prop:DPB}, is the fundamental result linking the combinatorics and the algebraic setup. We deal with two instances of monomial Cremona maps. 
The first one is quadratic Cremona monomial maps in an arbitrary number of variables. 
Here we recover \cite[Prop. 5.1]{SV2} characterizing monomial quadratic Cremona maps in a combinatorial language -- see Theorem \ref{prop:degreetwo}. The second instance is the monomial square free cubic transformations in $n \leq 6$ variables. Here we give a complete list up to equivalence classes under the action of $S^n$, the permutation group of $n$ symbols. We give an algorithm in SageMath whose input is the number of variables and the degree of the monomials that define the square-free Cremona maps. The output is the complete list of such Cremona Maps. \\

We now describe the contents of the paper in more detail. The second section highlights the combinatoric setup, the log matrix associated with a finite set of monomials having the same degree, and the associated clutter in the square-free case. We recall the determinantal principle of birationality in Proposition \ref{prop:DPB} and the 
duality principle in Proposition \ref{prop:duality}, both found in \cite{SV2}. We also present the so-called counting Lemmas \ref{processo}, \ref{isomorfas}, \ref{dual}, \ref{cone} and \ref{sequencia de grau}  which are the main tools to enumerate square free Cremona monomial maps up to equivalence under the natural action of $S_n$, the permutation group. We present also two extremal constructions of monomial Cremona transformations, Proposition \ref{prop:DLP} and Corollary \ref{cor:PRP}. These constructions allow us to recover the classification of quadratic monomial Cremona transformations in an arbitrary number of variables (see \cite[Prop. 5.1]{SV2} and Theorem \ref{prop:degreetwo}). This classification in our context has a natural combinatorial proof. 
We recall that the classification of quadro-quadric Cremona transformations in general was treated in \cite{PR,PR2}.\\

In the third section we recover the classification of square free monomial Cremona transformations in $\P^3$ and $\P^4$ (see \cite{SV2}), here, Proposition \ref{cor:p3} and Theorem \ref{cor:p4}. We prove the main result of the paper, Theorem \ref{thm:main}, that counts the number of square free monomial Cremona transformations in $\P^5$ up to equivalence classes. By the duality principle, Proposition \ref{prop:duality}, the 
hard part of the enumeration is the cubic square free monomial Cremona maps in $6$ variables, described in Proposition \ref{prop:d3n6t1}, Proposition \ref{prop:d3n6t2} and Proposition \ref{prop:d3n6t3}. \\

The fourth section is about an algorithm in SageMath that outputs the number of monomial square-free Cremona maps of degree $d$ in $n$ variables. We run the algorithm for $d \leq 3$ and $n \leq 7$ which give us the number of monomial square free Cremona maps in $\P^{n-1}$ for $n \leq 7$. 

\section{Combinatorics}

\subsection{The combinatoric setup}

Let $\K$ be a field and $n \geq 2$  let $ \K[x_1,\ldots,x_n]$ be the polynomial ring. For $v = (a_1,\ldots,a_n) \in \N^n$ we denote by 
$\underline{x}^{v}=x_1^{a_1}\ldots x_n^{a_n}$ the associated monomial and by $d=\mid v\mid=a_1+\ldots+a_n$ its degree. The vector $v$ is called the {\em log vector} of the monomial.

\begin{defin}
\sloppy{ Let $F=\{f_1,\ldots,f_n\}$ be an ordered set of monomials $f_j \in \K[x_1,\ldots,x_n]$, to which can associate the log vectors
$$v_j=(v_{1j}, \ldots,v_{nj})$$
where $\underline{x}^{v_j}=f_j$. The {\em log matrix} associated to $F$ is the matrix $A_F=(v_{ij})_{n \times n}$, whose columns are the (transpose of the) log vectors of $f_j$. 
If all the monomials have the same degree $d \geq 2$, which is our interest, then the log matrix is called $d$-stochastic.}
\end{defin}

The monomial $f_j$ is called square free if for all $x_i$, $i=1,\ldots,n$, we have $x_i^2 \nmid f_j$. The set $F$ is called square free if all its monomials are square free.\\

\sloppy{Let $F=\{f_1,\ldots,f_n\}$ be a set of monomials of same degree $d$ with $f_i \in \K[x_1,\ldots,x_n]$ for $i=1,\ldots,n$. $F$ defines a rational map: 
$$\varphi_F: \P^{n-1} \dashrightarrow \P^{n-1}$$
given by $\varphi_F(\underline{x})=(f_1(\underline{x}):\ldots,f_n(\underline{x}))$, where $\underline{x}=x_1, x_2, \ldots, x_n$.}
\\

The following definition has an algebro-geometric flavor, including an algebraic notion of birationality. It is advantageous in this context. For more details see \cite{SV1, SV2, CS}. 
An ordered set $F$ of $n$ monomials of the same degree $d$ is a Cremona set if the map $\varphi_F$ is a Cremona transformation.

\begin{defin} \label{defin:cremonaset} Let $\K[\underline{x}]=\K[x_1,\ldots,x_n]$. 
Let $\K[\underline{x}_d]$ be the Veronese algebra generated by all monomials of degree $d$.
Let $F$ be a set of monomials of same degree $d$. $F$ is a Cremona set if the extension $\K[f_1,\ldots,f_n] \subset \K[\underline{x}_d]$ 
becomes an equality at the level of fields of fractions.  
\end{defin}

We are interested in discussing whether the rational map $\varphi_F$ is birational. Hence we assume the following restrictions on the set $F$. 

\begin{defin} \label{defin:restricoescanonocas} 
We say that a set $F = \{f_1,\ldots,f_n\}$ of monomials $f_i \in \K[x_1,\ldots,x_n]$ of same degree satisfies the canonical restrictions if:
\begin{enumerate}
 \item For each $j=1, \ldots ,n$ there is a $k$ such that $x_j \mid f_k$;
 \item For each $j=1,\ldots,n$ there is a $k$ such that $x_j \nmid f_k$.
\end{enumerate}
\end{defin}

\begin{defin}
Let $\sigma \in S_n$ be a permutation of $n$ letters. For each monomial 
$f = \underline{x}^{v}$, with $v \in \N^n$, we denote $f_{\sigma}=\underline{x}^{\sigma(v)}$. 
For a finite set of monomials $F$ and $\sigma \in S_n$ we denote $F_{\sigma}=\{f_{\sigma}\mid f \in F\}$.  
 Two Cremona sets $F,F' \subset \K[x_1,\ldots,x_n]$ are said to be equivalent if there is $\sigma \in S_n$ such that $F'=F_{\sigma}$.
\end{defin}

\begin{rmk}\rm
Notice that we tacitly suppose that the order of the elements in $F$ is irrelevant. An ordered set of polynomials gives a Cremona map, but the property of being a Cremona transformation is invariant under permutations. One can see this by the algebraic definition of Cremona set, Definition \ref{defin:cremonaset}. An isomorphism between Cremona monomial maps is a relabel of the set of variables and a reorder of the forms.  
\end{rmk}

We make systematic use of the following Determinantal Principle of Birationality (DPB for short) due to Simis and Villarreal, see \cite[Prop. 1.2]{SV1}.

\begin{prop}\cite{SV1}\label{prop:DPB} {\bf(Determinantal Principle of Birationality (DPB))} Let $F$ be a finite set of monomials
of the same degree $d$ and let $A_F$ be its log matrix. Then $F$ is a Cremona set if and only if $\det A_F = \pm d$.
\end{prop}

We can associate to each set of square free monomials, $F \subset \K[x_1, \ldots, x_n]$, a combinatoric structure 
called clutter, also known as the Sperner family, see \cite[Chapter~6]{V} for more details. 

\begin{defin}
A clutter $S$ is a pair $S=(V,E)$ consisting of a finite set, the vertex set $V$, and a set of subsets of $V$, the edge set $E$. The cardinality of $S$ is given by $|E|$. The edge set is characterized by the property that that is no edge contained in another one.
We say that a clutter $S$ is a $d$-clutter if all the edges have the same cardinality $d$. 
\end{defin}

Let $F=\{f_1,\ldots,f_n\}$ be a set of square free monomials of same degree $d$ with $f_i \in \K[x_1,\ldots,x_n]$ and let $A=(v_{ij})$ be its log matrix. We define the clutter $S_F=(V,E)$ 
in the following way, $V = \{x_1, . . . , x_n\}$ is the vertex set, and $E=\{e_1,\ldots,e_n\}$ where 
$e_i=\{x_j, \ j \in \{1,\ldots,n\}\mid v_{ij}=1\}$. Notice that all the edges have the same cardinality, $\mid e\mid =d$, hence $S$ is a $d$-clutter. 
There is a bijective correspondence between $d$-clutters and sets of square-free monomials of degree $d$. In the present work, we deal only with $d$-clutters, here called, for short, clutter instead of to say $d$-clutter. 
Furthermore, all the clutter considered has the same number of vertex and edges. 

\begin{ex}\rm
 A simple graph $G=(V,E)$ is a $2$-clutter. Note that we can use a simple graph representing a square-free Cremona monomial map of degree two. If the set $F$ of monomials of degree two also contains some squares, they can be represented as loops in the graph. Hence a set of monomials of degree two can always be represented as a graph.  
\end{ex}

\begin{defin}

\begin{enumerate}
\item A subclutter $S'$ of a clutter $S=(V,E)$ is a clutter $S'=(V',E')$ where $V' \subset V$ and $E' \subset E$. Deleting an vertex $v$ of a clutter $S=(V,E)$ we obtain a subclutter $S \setminus v=(V',E')$ where $V'=V \setminus v$ and $E' = \{e \in E\mid v \not \in e\}$. 
\item  A clutter $C=(V,E)$  is called a cone 
 if there is a vertex $v \in V$ such that $v \in e$ for all $e \in E$.  The base of $C$ is the clutter $B_C=(V',E')$ with $V'=V\setminus v$ and $E'=\{e \setminus v, \forall e \in E\}$.
 \item A subclutter $C$ of a clutter $S$ is called maximal cone of $S$ if $C$ is a cone with maximal cardinality.

\end{enumerate}
 
 \end{defin}

 We recall a combinatoric notion of duality (see  \cite{SV2}, \cite{BARBARAARON} and \cite{DORIAARON}). 
 
 \begin{defin}
 Let $F$ be a set of square free monomials of the same degree $d$ with log-matrix $A_F = (v_{ij})_{n \times n}$, its dual complement is the set 
 $F^{\vee}$ of monomials whose log-matrix is $A_{F^{\vee}} = (1 - v_{ij})_{n \times n}$. From the clutter point of view, if $S_F=(V,E)$, then 
 the dual complement clutter (dual clutter for short), is the clutter $S^{\vee}:=S_{F^{\vee}}=(V,E^{\vee})$ where $E^{\vee}=\{V\setminus e\mid e \in E\}$. 
 \end{defin}

 The following basic principle is instrumental in the classification. For the proof, see \cite[Proposition 5.4]{SV2}.

\begin{prop}\cite{SV2} \label{prop:duality}{\bf(Duality Principle)} Let $F$ be a set of square free monomials in $n$ variables, of the same
degree $d$, satisfying the canonical restrictions. Then $F$ is a Cremona set if and only if $F^{\vee}$ is a Cremona set.
 \end{prop}

\subsection{Two extremal principles and its consequences}

Consider a set $F=\{f_1,\ldots,f_n\}$ of monomials $f_i \in \K[x_1,\ldots,x_n]$, of same degree $d$ satisfying the canonical restrictions. The log matrix of $F$, $A_F$ is a $d$-stochastic matrix since 
the sum along every column is equal to $d$. The incidence degree of $x_i$ in $F$, denoted by $a_i$, is the number of monomials $f$ in $F$ such that $x_i\mid f$. By double counting in the log matrix, we have the following incidence equation:
\begin{equation}\label{eq:incidence}
a_1+\ldots+a_n=nd,
\end{equation}
here $1 \leq a_i \leq n-1$ for all $i=1,\ldots,n-1$.\\

Let $F=\{f_1,\ldots,f_n\}$ be a set of square free monomials of same degree $d$. Suppose that there is a variable $x_i$ such that the incidence degree in $F$ is 1. This extremal condition means that the vertex $x_i$ belongs to only one facet of the associated clutter $S(V,E)$ of $F$.  
This observation  suggests the following definition  inspired by a similar notion on Graph theory.

\begin{defin}
 Let $S=(V,E)$ be a clutter. A leaf in $S$ is a vertex $v \in V$ with incidence degree $1$. Let $e$ be the only edge containing $v$. 
 Then deleting the leaf $v$ we obtain a subclutter $S\setminus v = (V\setminus v,E \setminus e)$.
\end{defin}

The following result allows us to delete leaves of the clutters associated with a square-free Cremona set to obtain another square-free Cremona set on a smaller ambient space.  Similarly, we can attach leaves to square-free Cremona sets. 

\begin{prop} \label{prop:DLP} {\bf (Deleting Leaves principle (DLP))}
 Let $F=\{f_1,\ldots,f_n\}$ be a set of monomials of same degree $d$, with $f_i \in \K[x_1,\ldots,x_n]$ for $i=1,\ldots,n$, satisfying the canonical restrictions. 
  Suppose that $x_n\mid f_n$, $x_n^2\nmid f_n$, and $x_n \nmid f_i$ for $i = 1,\ldots,n-1$. 
 Let $F' = \{f_1,\ldots,f_{n-1}\}$ be considered as a set of monomials of degree $d$ in $\K[x_1,\ldots,x_{n-1}]$.
 \begin{enumerate}
  \item If $F'$ does not satisfy the canonical restrictions, then $F$ is not a Cremona set.
  \item If $F'$ satisfies the canonical restrictions, then $F$ is a Cremona set if and only if $F'$ is a Cremona set.
 \end{enumerate}
\end{prop}

\begin{proof} Let $A_F$ be the log matrix of $F$. Computing the determinant $\det A_F$ by Laplace's rule on the last row it is easy to see that:
$$\det A_F = \pm \det A_{F'}.$$
\begin{enumerate}
 \item If $F'$ does not satisfy the canonical restrictions, then $\det A_{F'}=0$. Indeed, if there is a $x_j$ for some $j=1,\ldots,n-1$ such that $x_j \nmid f_i$ 
 for all $i=1,\ldots,n-1$, then the $j$-th row of $A_{F'}$ is null and $\det A_{F'}=0$. On the other side, if  there is a $x_j$ for some $j=1,\ldots,n-1$ 
 such that $x_j \mid f_i$ 
 for all $i=1,\ldots,n-1$, then the $j$-th row of $A_{F'}$ has all entries equal to $1$. Since $A_{F'}$ is $d$-stochastic in the columns, replace the first row
 for the sum of all the rows except the $j$-th give us a row with all entries $d-1$, hence $\det A_{F'}=0$.  
  \item If $F'$ satisfies the canonical restrictions, then by the Determinantal Principle of Birationality, Proposition \ref{prop:DPB}, $F$ is a Cremona set if and only if $\det A_F=d$, since 
  $\det A_F = \pm \det A_{F'}$, the result follows. 
\end{enumerate}
\end{proof}

In the opposite direction, suppose that $F$ has a variable whose incidence degree is $n-1$. Geometrically it is the maximal possible cone on the clutter since $F$ satisfies the canonical restrictions. We want to focus on the base of this cone.

\begin{defin}
 Let $S=(V,E)$ be a clutter with $|E|=n$ we say that $v \in V$ is a root if it has incidence degree $n-1$. 
 We define a new clutter $S/v=(\tilde{V},\tilde{E})$ pucking the root $v$, where $\tilde{V}=V \setminus v$ and $\tilde{E}=\{e \setminus v\mid e \in E, v \in e\}$. 
 It is easy to see that $S/v=(S^{\vee}\setminus v)^{\vee}$. 
 \end{defin}

\begin{cor} \label{cor:PRP} {\bf (Plucking Roots Principle (PRP))}
 Let $F=\{f_1,\ldots,f_n\}$ be a set of square free monomials of same degree $d$, with $f_i \in \K[x_1,\ldots,x_n]$ for $i=1,\ldots,n$, satisfying the canonical restrictions. 
  Suppose that $f_i=x_ng_i$, for $i = 1,\ldots,n-1$ and $x_n \nmid f_n$. 
 Let $\tilde{F} = \{g_1,\ldots,g_{n-1}\}$ be considered as a set of square free monomials of degree $d-1$ in $\K[x_1,\ldots,x_{n-1}]$. Then $F$ is a Cremona set if and only if $\tilde{F}$ is a Cremona set.
  In particular, if  $F$ is a Cremona set, then $\tilde{F}$ satisfies the canonical restrictions.

\end{cor}

\begin{proof}
 Let $S_F$ the clutter associated to $F$. Since $S_F/v=(S_F^{\vee}\setminus v)^{\vee}=(S_{F^\vee}\setminus v)^{\vee}$ is the clutter associated to $\tilde{F}$, 
 hence $(\tilde{F})^{\vee}=(F^{\vee})'$. By Duality Principle, Proposition \ref{prop:duality}, $F$ is a Cremona set if and only if $F^{\vee}$ is a Cremona set.
 Since $F^{\vee}$ has a leaf, by DLP, Proposition \ref{prop:DLP}, $F^{\vee}$   is a Cremona set if and only if $(F^{\vee})'$  is a Cremona set. 
 Since $(F^{\vee})'=(\tilde{F})^{\vee}$, the result follows by Duality Principle.
 \end{proof}

The Corollary  \ref{cor:PRP} is a particular case of the proposition 1.6 of \cite{rasi} about the generalized Jonquières transformations. 

To classify the Cremona sets of degree two, we use the following Lemma. A proof can be found in \cite[Lemma 4.1]{SV2}.

\begin{lema}\cite{SV2} \label{lema:cohesive}
Let $F = \{f_1, ... , f_n\} \subset \K[\underline{x}] = \K[x_1,\ldots,x_n]$ be forms of fixed degree $d \geq 2$.
Suppose one has a partition $\underline{x} = \underline{y} \cup \underline{z}$ of the variables such that $F = G \cup H$, where the forms in
the set $G$ (respectively, $H$) involve only the $\underline{y}$-variables (respectively, $\underline{z}$-variables). If neither $G$ nor
$H$ is empty then $F$ is not a Cremona set.
\end{lema}

\begin{defin}
 A set $F$ of monomials satisfying the canonical restrictions is said to be cohesive if the forms can not be disconnected as in the hypothesis of Lemma \ref{lema:cohesive}. 
\end{defin}

The next result is elementary, but it concentrates all the essential information about the leafless case.

\begin{lema}\label{lema:cycle}
 Let $C_n=(V,E)$ be an $n$-cycle with $n\geq 3$, $V=\{1,\ldots,n\}$ and\\ $E=\{\{1,2\},\{2,3\},\ldots,\{n-1,n\},\{n,1\}\}$. Let $M_n := M_{n \times n}$
 be the incidence matrix of $C_n$. Then $\det M_n = 1-(-1)^n$. 
\end{lema}

\begin{proof} Computing the determinant by Laplace's rule on the first row:\\
$$
 \begin{vmatrix}
  1&0&\ldots&0&1\\
  1&1&\ldots&0&0\\
  0&1&\ldots&0&0\\
  0&0&\ldots&0&0\\
  \vdots & \vdots &\ldots&\vdots &\vdots \\
  0&0&\ldots&1&0\\
  0&0&\ldots&1&1\\
 \end{vmatrix}_{n}
  = 
\begin{vmatrix}  
  1&0&\ldots&0&0\\
  1&1&\ldots&0&0\\
  0&1&\ldots&0&0\\
   \vdots & \vdots &\ldots&\vdots &\vdots \\
  0&0&\ldots&1&0\\
  0&0&\ldots&1&1\\
 \end{vmatrix}_{n-1}
  + 
 (-1)^{n-1}\begin{vmatrix}
  1&1&0&\ldots&0\\
  0&1&1&\ldots&0\\
  0&0&1&\ldots&0\\
  \vdots & \vdots & \vdots &\ldots&\vdots \\
  0&0&0&\ldots&1\\
  0&0&0&\ldots&1\\
 \end{vmatrix}_{n-1}
 $$

 Since both determinants on the right are triangular, the result follows.
 \end{proof}

 We now are in a position to give a direct and purely combinatoric proof of the structure result of monomial Cremona sets in degree two, see \cite[Prop. 5.1]{SV2}. 
 
\begin{thm}\label{prop:degreetwo} Let $F \subset \K[x_1,\ldots,x_n]$ be a cohesive set of monomials of degree two 
satisfying the canonical restrictions. Let $A_F$ denote the log matrix and $G_F$ the graph. The following conditions are equivalent:
\begin{enumerate}
  \item $\det A_F \neq 0$; 
 \item $F$ is a Cremona set;
 \item Either \begin{enumerate}
                \item[(i)] $G_F$ has no loops and a unique cycle of odd degree; 
                \item[(ii)] $G_F$ is a tree with exactly one loop.
              \end{enumerate}

\end{enumerate}
\end{thm}
 
\begin{proof}

Since $F$ is a cohesive set, the associated graph is connected. Therefore, by DLP, Proposition \ref{prop:DLP}, we can delete all the leaves of $G_F$ to construct another cohesive set $F'$ on 
$m\leq n$ variables whose graph is connected and leafless. 
The incidence Equation \ref{eq:incidence} applied to $F'$ give us $a'_1=\ldots=a'_m=2$. Hence $G_{F'}$ is a disjoint union of cycles and loops. 
Since $G_{F'}$ is connected, there are only two possibilities. Either
\begin{enumerate}
 \item[(i)] $G_{F'}$ is a single loop; $F'=\{x^2\}$ and $\det A_F=2$;
 \item[(ii)] or $G_{F'}$ is a cycle. By Lemma \ref{lema:cycle}, odd cycles has determinant $2$ and even cycles have determinant $0$. 
\end{enumerate}
The result follows by attaching the petals on $F'$ to reach $F$. 
\end{proof}

\subsection{Counting Lemmas}

Our main goal is to count the equivalence classes of monomial Cremona transformations under the natural action of $S_n$. To do this, we study the action of the group 
$S_n$ on the set of all square free Cremona sets of monomials of fixed degree.\\

\begin{defin}
Let $G$ be a group and $X$ a non empty set. We denote by
$$\begin{array}{cccc}
\ast: & G \times X & \rightarrow & X \\
 & (g,x) & \mapsto & g*x \end{array}$$
an action of $G$ on $X$. The action induces a natural equivalence relation among the elements of $X$; $x\equiv y$ if and only if, $x=g\ast y$ for some $g \in G$. We denote the orbit of a element $x \in X$ by
$$\mathcal{O}_x=\{x' \in X\mid x'=g \ast x \mbox{ for some}\  g \in G\}.$$
The set of the orbits is the quotient $X/G$. The stabilizer of $x$ is $G_x=\{g \in G \mid g*x=x\}<G$. If $G$ acts on finite set $X$, then this action induces an action on $X^k$ the set of $k$-subsets of $X$.\\
 
\end{defin}

\begin{defin}\label{def_action}
Let $\mathcal{M}_{n,d}$ be the set of square free monomials in $\K[x_1, \ldots, x_n]$ and let $\mathcal{M}^k_{n,d}$ be the set of $k$-subsets of $\mathcal{M}_{n,d}$. There is a natural action of $S_n$ on both the sets $\mathcal{M}_{n,d}$ and $\mathcal{M}^k_{n,d}$.

$$\begin{array}{cccc}
\ast: & S_n \times \mathcal{M}_{n,d}  & \rightarrow & \mathcal{M}_{n,d} \\ &(\sigma,m)=(\sigma, x_1^{a_1}\ldots x_n^{a_n}) & \mapsto & \sigma \ast m = x_{\sigma(1)}^{a_1}\ldots x_{\sigma(n)}^{a_n}
\end{array}$$
$$\begin{array}{cccc}
 \ast: & S_n \times \mathcal{M}^k_{n,d}  & \rightarrow & \mathcal{M}^k_{n,d}\\
& (\sigma,\{m_1, \ldots, m_k\}) & \mapsto & \{\sigma \ast m_1, \ldots, \sigma \ast m_k\}
\end{array}.$$

Let us consider the subset $\mathcal{C}_{n,d} \subset \mathcal{M}^n_{n,d}$ of Cremona sets representing square free monomial Cremona maps. 
To count equivalence classes of square free  monomial Cremona maps under the action of $S_n$ is to determine the orbits of $\mathcal{C}_{n,d}/S_n \subset \mathcal{M}^n_{n,d}/ S_n$.
 
\end{defin}

The next result allows to construct $\mathcal{M}^k_{n,d}/S_n$ iteratively. Notice that $S_n$ acts transitively on $M^1_{n,d}$.

\begin{lema}\label{processo}
 If $\mathcal{M}^i_{n,d}/S_n=\{\mathcal{O}_{F_1}, \ldots, \mathcal{O}_{F_r}\}$ then $$\mathcal{M}^{i+1}_{n,d}/S_n=\{\mathcal{O}_{\{F_j,\beta*f\}}; \, F_j \mbox{ is a representative of an orbit on}\ \mathcal{M}^i_{n,d}/S_n,$$ $$\mbox{for some} \, f \in \mathcal{M}_{n,d} \setminus F_j\ \mbox{ and }\ \, \forall \beta \in S_n \,  \}$$
 \end{lema}

\begin{proof} Define $A=\{\mathcal{O}_{\{F_j,\beta*f\}}; \, F_j \mbox{ is a representative of some orbit in} \mathcal{M}^i_{n,d}/S_n,$ $\mbox{for some } \, f \in \mathcal{M}_{n,d} \setminus F_j \mbox{ and } \, \forall \beta \in S_n\}$. 
 Since  $A \subset \mathcal{M}^{i+1}_{n,d}/S_n$ by construction we will show that $\mathcal{M}^{i+1}_{n,d}/S_n \subset  A$.

In fact, consider $\mathcal{O}_H \in \mathcal{M}^{i+1}_{n,d}/S_n$, 
say that $H=\{g_1, \ldots, g_i, g_{i+1}\}$. Since $\overline{H}=H \setminus \{g_{i+1}\}$ has $i$ elements 
there is a $j$ such that $\overline{H} \in \mathcal{O}_{F_j}$, that is, there is $\gamma \in S_n$ such that 
$\gamma*F_j=\overline{H}$.

Since $S_n$ acts transitively in $\mathcal{M}_{n,d}$, for $f \in \mathcal{M}_{n,d} \setminus F_j$, 
there is $\beta \in S_n$ such that $\beta*f=\gamma^{-1}*g_{i+1}$.

Therefore $\gamma*\{F_j, \beta*f\}=H$, that is, $\mathcal{O}_H=\mathcal{O}_{\{F_j, \beta*f\}} \in A$.
\end{proof}

This process determines $\mathcal{M}^i_{n,d}/S_n$ inductively, but each orbit can appear many times. We now 
answer partially the natural question of whether $\mathcal{O}_{\{F_j, \beta_1*f\}}=\mathcal{O}_{\{F_k, \beta_2*g\}}$, 
with $F_j, F_k$ representatives of distinct orbits in $\mathcal{M}^i_{n,d}/S_n$.

\begin{lema}\label{isomorfas}
Let $F \in \mathcal{M}^i_{n,d}$ and $f \in \mathcal{M}_{n,d}$. If $\gamma \in G_F$ and $\beta^{-1} \gamma \delta \in G_f$ then 
$\mathcal{O}_{\{F, \delta*f\}}=\mathcal{O}_{\{F, \beta*f\}}$.
In particular, if $\gamma \in G_F$ then $\mathcal{O}_{\{F, \gamma*f\}}=\mathcal{O}_{\{F, f\}}$.
\end{lema}
\begin{proof} Verify that $\gamma*\{F,\delta*f\}=\{F, \beta*f\}$.
\end{proof}

By using that every permutation is a bijection, we can prove the two following Lemmas. 

\begin{lema}\label{dual}
If $F,H \in \mathcal{M}^i_{n,d}$ and $\mathcal{O}_{F}=\mathcal{O}_{H}$ then 
$\mathcal{O}_{F^{\vee}}=\mathcal{O}_{H^{\vee}}$.
\end{lema}

\begin{lema}\label{cone}
If $F,H \in \mathcal{M}^i_{n,d}$ and $\mathcal{O}_{F}=\mathcal{O}_{H}$, then for all maximal cone 
$\mathcal{C}$ of $F$ there is a just one maximal cone $\mathcal{C}'$ of $H$ such that 
$\mathcal{O}_{\mathcal{C}}=\mathcal{O}_{\mathcal{C}'}$ and {\it vice versa}.
\end{lema}


Let $F \subset \K[x_1, \ldots, x_n]$ be a set of square free monomials of degree $d$.
The sequence of incidence degrees of $F$ is the sequence of incidence degree of $x_1, \ldots, x_n$ 
in non-increasing order. If the incidence degree of $x_i$ in $F$ is $a$ and $\alpha \in S_n$ such that $\alpha(i)=j$ 
then the incidence degree of $x_j$ in $\alpha * F$ is $a$. We have proved the following. 

\begin{lema}\label{sequencia de grau}
If $\mathcal{O}_{F}=\mathcal{O}_{H}$ then $F$ and $H$ have the same sequence of incidence degrees.
\end{lema}

\begin{rmk}\rm 
The converse is not true as one can check with easy examples. Indeed, let $F=\{x_1x_2,x_2x_3,x_1x_3,x_1x_4,x_4x_5,x_5x_6\}$ and 
let $H=\{x_1x_2,x_2x_3,x_3x_4,x_4x_5,x_1x_5,x_1x_6\}$ be sets of square free monomials of degree $2$. They have the same sequence of incidence degree, but $\mathcal{O}_{F} \neq \mathcal{O}_{H}$. 
It is easy to see that the graphs $G_F$ and $G_H$ associated with $F$ and $H$ respectively are nonequivalent. Indeed,  $G_H$ has just one cycle, and it has three elements.  On the other side, $G_H$ has just one cycle. It has five elements, therefore $\mathcal{O}_{F} \neq \mathcal{O}_{H}$. 
\end{rmk}

\section{Counting square-free monomial Cremona maps}

\subsection{Counting Cremona sets in $\mathbb{P}^{3}$}
In this section we re-obtain some results of \cite{SV2} in \S 3.1 and \S 3.2. The \S 3.3 contains the new results that are Propositions $3.6$, $3.7$ and $3.8.$\\


%

\begin{prop}\label{cor:p3}

There are three nonisomorphic square-free monomial Cremona transformations in $\P^3$. 
\end{prop}

\begin{proof}
 We are in the case $n=4$. Let $d$ be the common degree. If $d=1$ we have only the identity. If $d=3$, the unique Cremona Transformation is the standard inversion.

  If $d=2$, by Theorem \ref{prop:degreetwo} there are only one possible Cremona set whose graph is:

  \hspace{2.0in}
  \begin{figure}[!htbp]

\begin{center}
\label{aaa}
\includegraphics[height=0.8in]{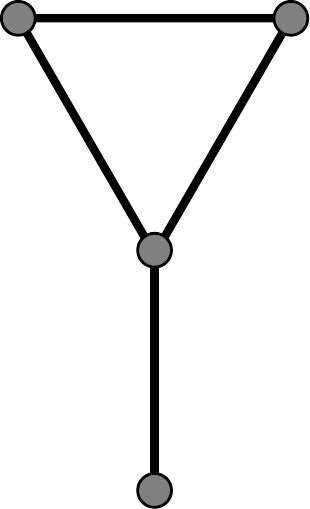}\\
\caption{The unique graph of a Cremona map in $\P^3$ with $d=2$}
\end{center}
\end{figure}

\end{proof}

\subsection{Counting Cremona sets in  $\mathbb{P}^{4}$}
\begin{thm}\label{cor:p4}
There are ten nonisomorphic square-free monomial Cremona transformations in $\P^4$. 
\end{thm}

\begin{proof}

If the degree $d=1$ or $d=4$, we have only one Cremona monomial map. \\

By Duality Principle, Proposition \ref{prop:duality}, the number of square-free monomial Cremona maps of degree $d=2$ and $d=3$ are the same. Furthermore, 
by Theorem \ref{prop:degreetwo}, the possible square free Cremona sets of degree two have the following graphs:

\medskip

\begin{figure}[h!]
\centering
\hspace{-0.6cm}
\begin{minipage}[h]{0.31 \linewidth}
\centering  
\includegraphics[height=0.8in]{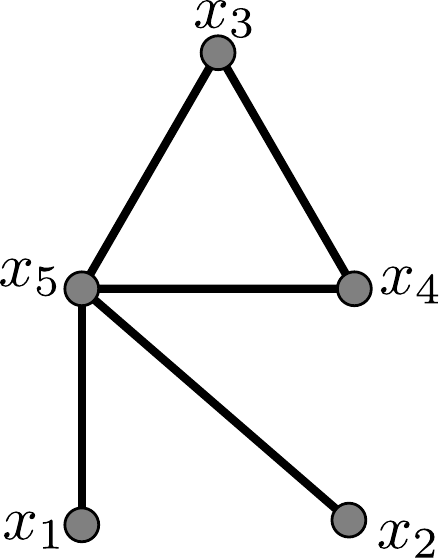} \\
\caption{\\Graph $G1$}\label{fig_2}
\end{minipage}
\hspace{-1.2cm}
\begin{minipage}[h]{0.31 \linewidth}
\centering
\includegraphics[height=0.8in]{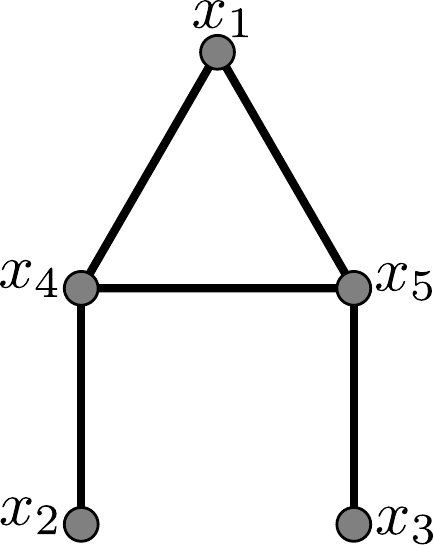}  \\
\caption{\\Graph $G_2$}\label{fig_3}
\end{minipage}
\hspace{-1.3cm}
\begin{minipage}[h]{0.31\linewidth}
\centering
\includegraphics[height=0.8in]{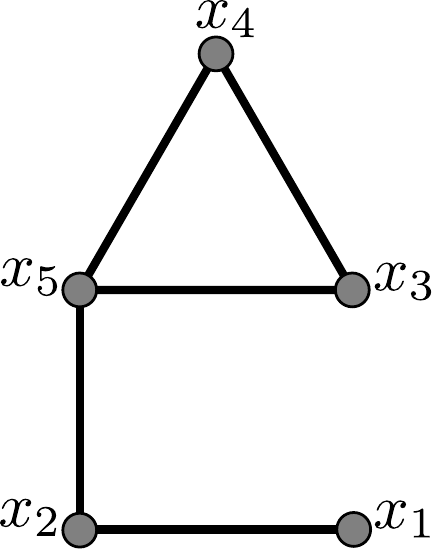}  \\
\caption{\\Graph $G_3$}\label{fig_4}
\end{minipage}
\hspace{-1.2cm}
\begin{minipage}[h]{0.31\linewidth}
\centering
\includegraphics[height=0.8in]{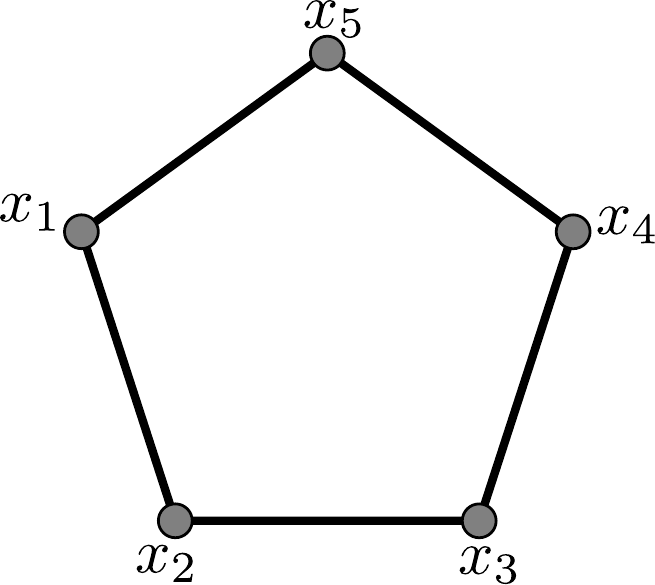}  \\
\caption{\\Graph $G_4$}\label{fig_5}
\end{minipage}

\end{figure}

%
\medskip

The dual complement of such Cremona sets are the square free Cremona sets of degree three, the associated clutter are the clutter duals of the preceding graphs:
 
\begin{figure}[h!] 
\centering
\hspace{-0.6cm}
\begin{minipage}[h]{0.31 \linewidth}

\centering
\includegraphics[height=.9in]{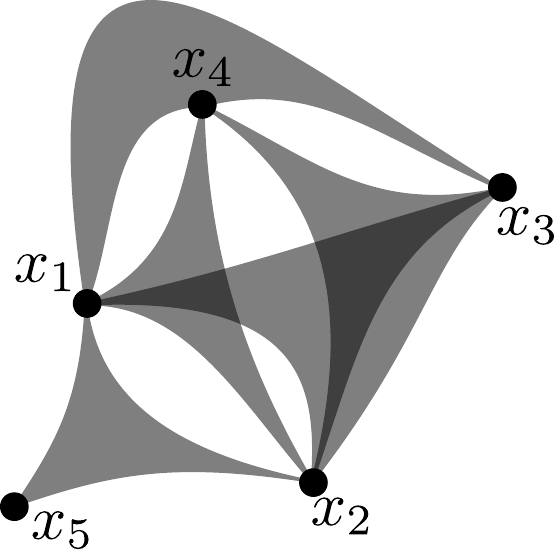}
\caption{Clutter \\$S_1=G_1^{\vee}$}
\end{minipage}
\hspace{-1.0cm}
\begin{minipage}[h]{0.31 \linewidth}
\centering
\includegraphics[height=.9in]{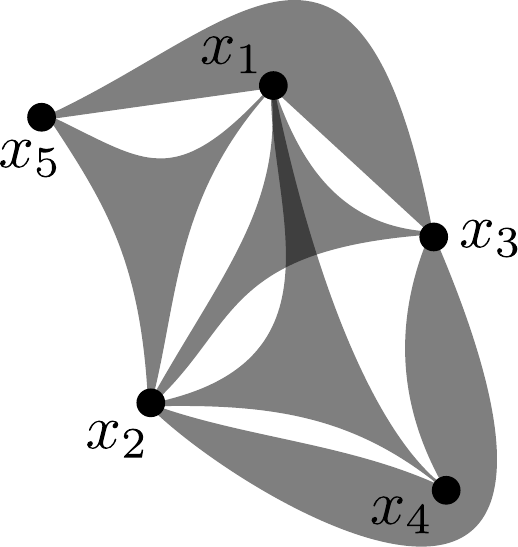}
\caption{Clutter \\$S_2 = G_2^{\vee}$}
\end{minipage}
\hspace{-1.1cm}
\begin{minipage}[h]{0.31\linewidth}
\centering
\includegraphics[height=.9in]{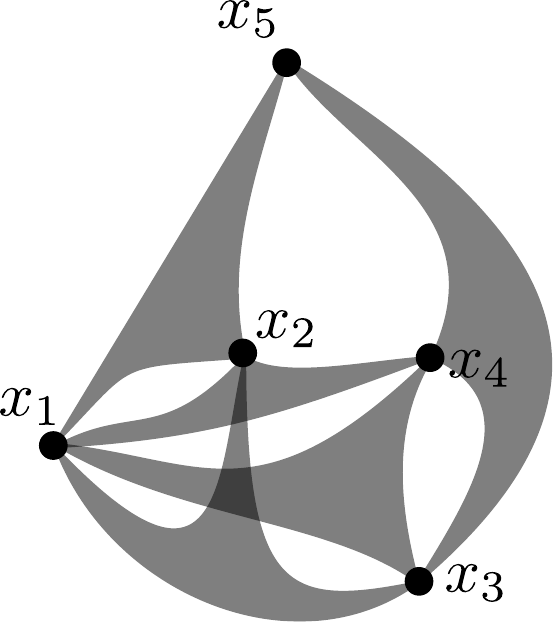}
\caption{Clutter \\$S_3 = G_3^{\vee}$}
\end{minipage}
\hspace{-1.2cm}
\begin{minipage}[h]{0.31\linewidth}
\centering
\includegraphics[height=.9in]{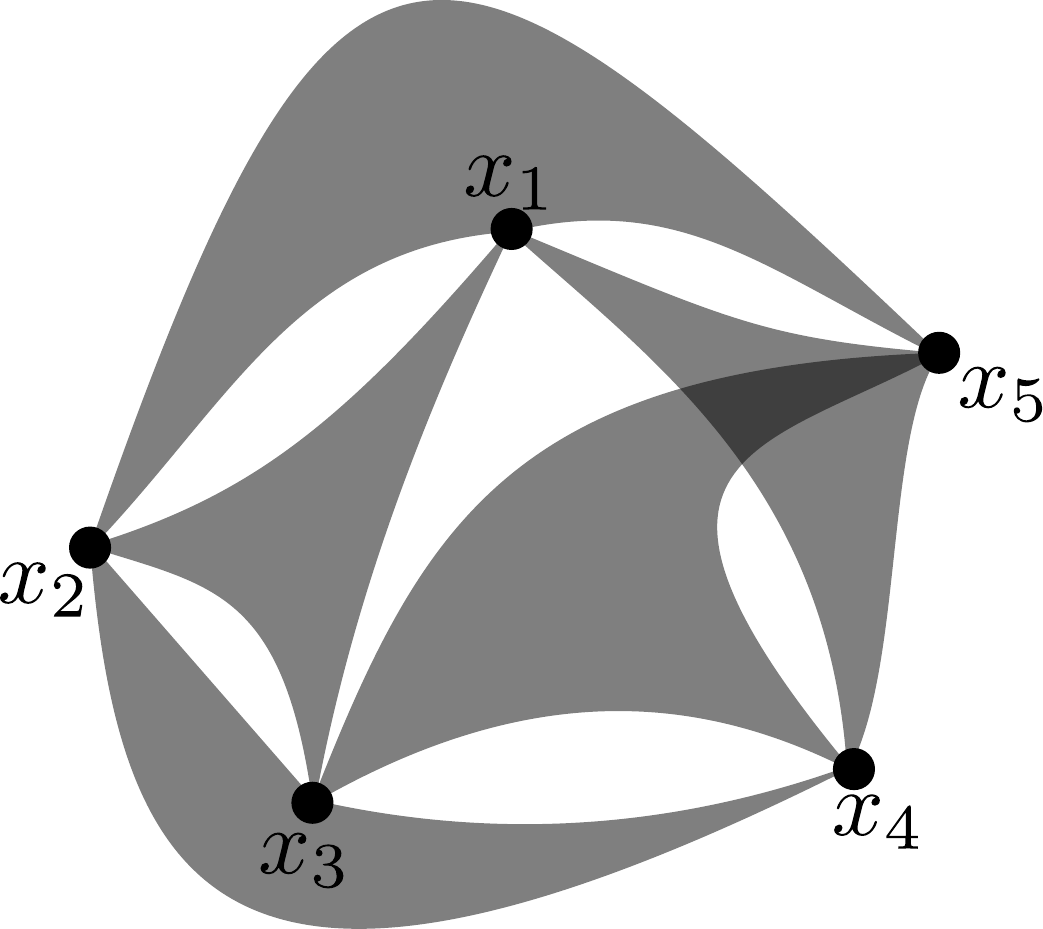}
\caption{Clutter \\$S_4 = G_4^{\vee}$}
\end{minipage}
\end{figure}
\medskip
 \end{proof}

\subsection{Counting Cremona sets in $\mathbb{P}^{5}$}

Our main result is the following Theorem:

\begin{thm}\label{thm:main}
 There exist fifty-eight equivalence classes of square free monomial Cremona transformations in $\P^5$.
\end{thm}

\begin{proof}
 Let $d$ be the degree of the Cremona set, we have $1 \leq d \leq 5$. We have only one Cremona set for either $d=1$ or $d=5$, and they are dual. 
 For $d=2$, or dually $d=4$, we have eight possibilities, see Proposition \ref{prop:d2n6}. In total sixteen. \\
 For $d=3$ there are forty nonequivalent monomial Cremona sets according to Propositions \ref{prop:d3n6t1}, \ref{prop:d3n6t2} and \ref{prop:d3n6t3}.
\end{proof}

We give a complete description of these Cremona sets by drawing the associated clutters.
 
 \begin{prop}\label{prop:d2n6}
 There are eight equivalence classes of square free monomial Cremona sets of degree $2$ in $\K[x_1,\ldots,x_6]$. 
 \end{prop}

 \begin{proof} According to Theorem \ref{prop:degreetwo} we have the following possibilities for the associated graph of such a Cremona set. 

 \medskip

\begin{figure}[!htbp]
\captionsetup{justification=centering}
\centering
\hspace{0.2cm}
\begin{minipage}[h]{0.25 \linewidth}
\hspace{-.2cm}
\centering
\includegraphics[height=.7in]{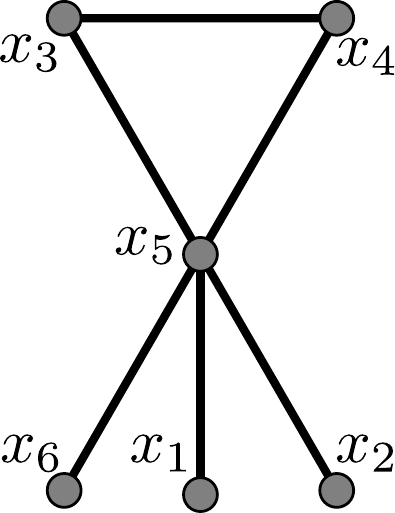}\\
\caption{\\Graph $G_5$}
\end{minipage}
\centering
\hspace{-0.2cm}
\begin{minipage}[h]{0.25 \linewidth}
\centering
\includegraphics[height=.7in]{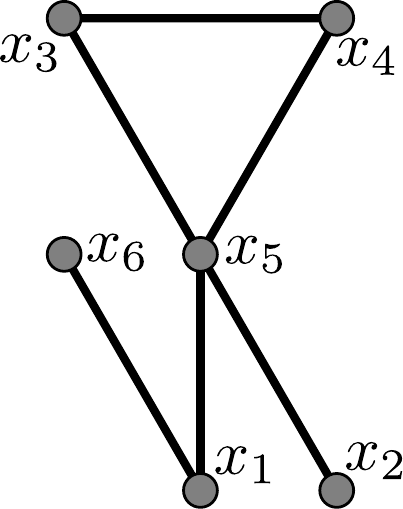}\\
\caption{\\Graph $G_6$}
\end{minipage}
\hspace{-0.1cm}
\begin{minipage}[h]{0.25\linewidth}
\centering
\includegraphics[height=.7in]{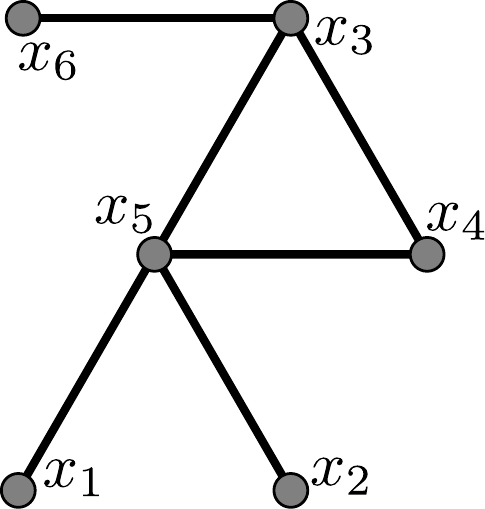}\\
\caption{\\Graph $G_7$}
\end{minipage}
\hspace{-0.5cm}
\begin{minipage}[h]{0.25 \linewidth}
\centering
\includegraphics[height=.7in]{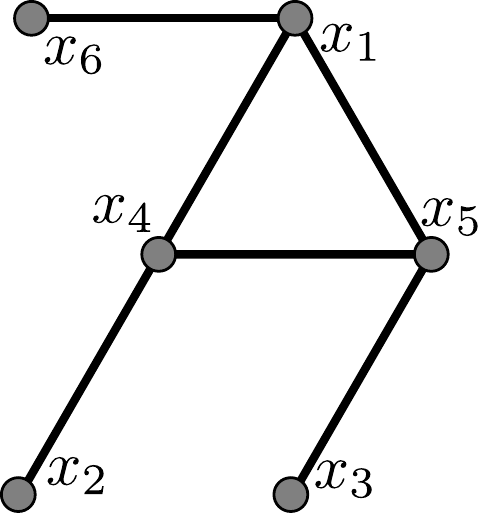}\\
\caption{\\Graph $G_8$}
\end{minipage}
\end{figure}


\begin{figure}[!h]
\captionsetup{justification=centering}
\centering
\hspace{-.2cm}
\begin{minipage}[h]{0.25 \linewidth}
\hspace{-.1cm}
\centering
\includegraphics[height=.7in]{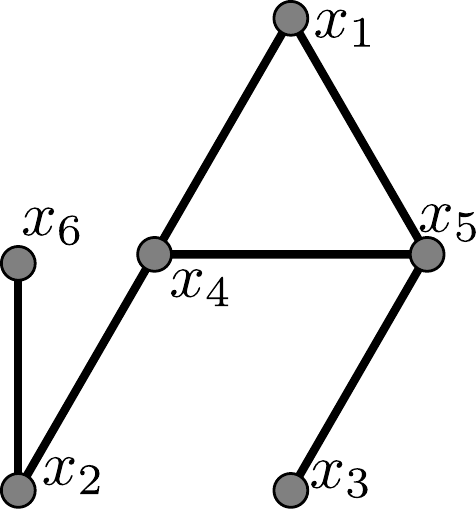}\\
\caption{\\Graph $G_9$}
\end{minipage}
\centering
\hspace{-.5cm}
\begin{minipage}[h]{0.25 \linewidth}
\centering
\includegraphics[height=.7in]{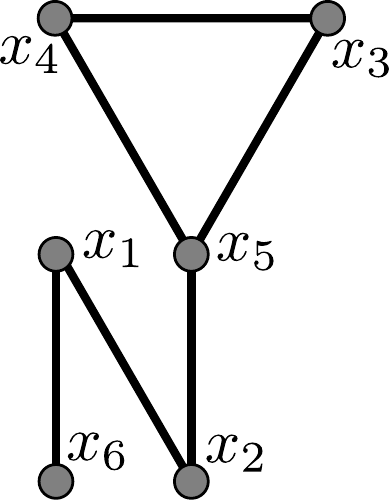}\\
\caption{\\Graph $G_{10}$}
\end{minipage}
\centering
\hspace{-.5cm}
\begin{minipage}[h]{0.25 \linewidth}
\centering
\includegraphics[height=.7in]{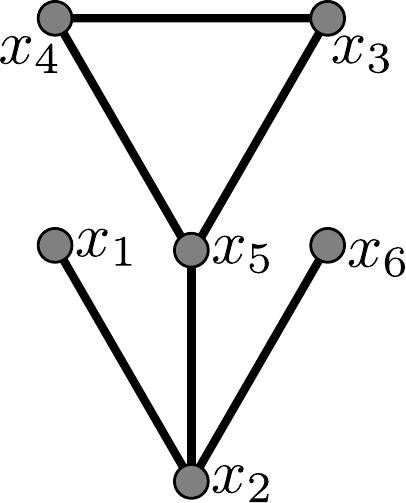}\\
\caption{\\Graph $G_{11}$}
\end{minipage}
\centering
\hspace{-0.5cm}
\begin{minipage}[h]{0.25 \linewidth}
\centering
\includegraphics[height=.7in]{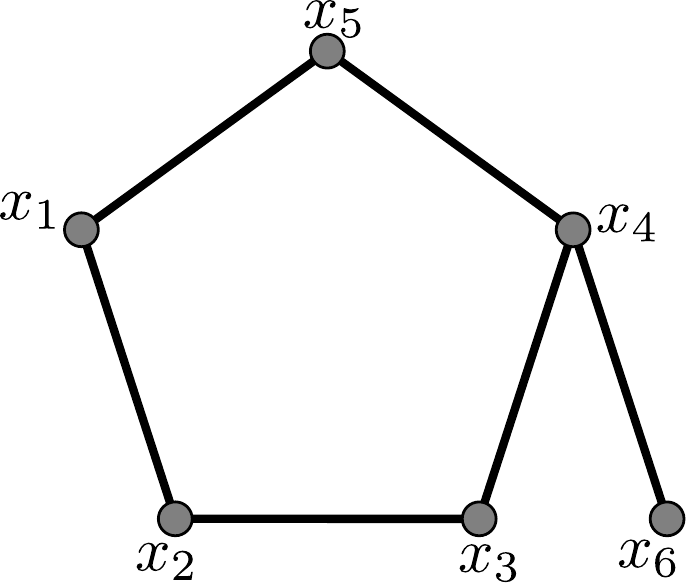}\\
\caption{\\Graph $G_{12}$}
\end{minipage}
\end{figure}
 
 \end{proof}

Let us consider square free monomial Cremona transformations of $\mathbb{P}^5$ as a set of $n=6$ square free monomials of degree $d=3$. The corresponding log-matrix is a $6\times 6$ $3$-stochastic matrix 
whose determinant is $\pm 3$ by the DPB, Proposition \ref{prop:DPB}.

\begin{rmk}\label{rmk:3types}\rm 
From now on we deal with the following setup: $A_F$ is a $6 \times 6$ $3$-stochastic matrix with eighteen entries $0$ and eighteen entries $1$ and
whose determinant is $\pm 3$. Furthermore, since the dimension $\dim R = 6$ and the degree $d=3$ are not coprime, $A_F$ can not
be doubly stochastic, see \cite[Proposition 5.6]{SV2}.
So there is a row of the matrix $A_F$ with at least four entries $1$ and it has three possible types:
\begin{enumerate}
\item $A_F$ does not have a row with five entries $1$ but has a row with only one entry $1$. The clutter has one leaf but has no root;

\item $A_F$ has a row with five entries $1$. The associated clutter has a root;

\item $A_F$ does not have a row with five entries $1$ neither a row with only one entry $1$. The clutter is leafless and has no root.
\end{enumerate}

\end{rmk}

\begin{lema}\label{mdc}
Let $F=\{f_1,\ldots,f_6\} \subset \K[x_1,\ldots,x_6]$ be cubic monomials defining a Cremona transformation of $\mathbb{P}^5$.
Then for each choice of $4$ monomials of $F$ there are $2$ of them whose $\operatorname{gdc}$ is of degree $2$.
 \end{lema}
 \begin{proof} Suppose, on teh contrary, that there are four monomials $f_1,f_2,f_3,f_4$ such that
 $\deg(\operatorname{gcd}(f_i,f_j)) \leq 1$. On the log matrix it imposes the existence of a $6 \times 4$ sub-matrix whose
 all $6 \times 2$ sub-matrices have at most one line with two entries $1$. Let us consider, up to a permutation, $f_1=x_1x_2x_3$.
It is easy to see that, up to a permutation, the log-matrix of these four vectors must be of the form:
$$\left[ \begin{array}{cccc}
1 & 1 & 0 & 0\\
1 & 0 & 1 & 0\\
1 & 0 & 0 & 1\\
0 & 1 & 1 & 0\\
0 & 1 & 0 & 1\\
0 & 0 & 1 & 1\\
\end{array} \right].$$
On the other side, by applying elementary operations in the rows, we can prove that  the log-matrix of any set
$F=\{x_1x_2x_3, x_1x_4x_5,x_2x_4x_6,x_3x_5x_6,f_5,f_6\}$
has even determinant. This contradicts our hypothesis that the set of monomials defines a Cremona transformation.
\end{proof}

\begin{prop} \label{prop:d3n6t1}
There are $10$ equivalence classes of square free Cremona monomials of degree $3$ in $\K[x_1, \ldots, x_6]$
whose log-matrix are of type $1$.
\end{prop}

\begin{proof} Let $F \subset \K[x_1, \ldots, x_6]$ be such a set, that is, $F=\{F',mx_6\}$, for which $\{m\} \cup F' \subset K[x_1, \ldots, x_5]$. By DLP, Proposition \ref{prop:DLP}, $F'$ is a Cremona set.

We describe square-free monomial Cremona transformations of degree $3$ in $\mathbb{P}^4$ in corollary \ref{cor:p4}. 
There are $4$ equivalence classes of them. The associated clutters have the following representation:

\begin{figure}[h!] 
\centering
\hspace{-0.3cm}
\begin{minipage}[h]{0.26 \linewidth}
\centering
\includegraphics[height=0.7in]{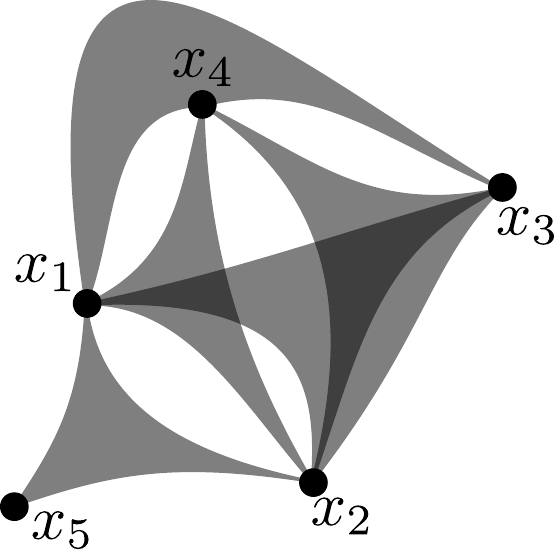}\\
\caption{\\ Clutter  $F'_1$}
\end{minipage}
\hspace{-0.5cm}
\begin{minipage}[h]{0.26 \linewidth}
\centering
\includegraphics[height=0.7in]{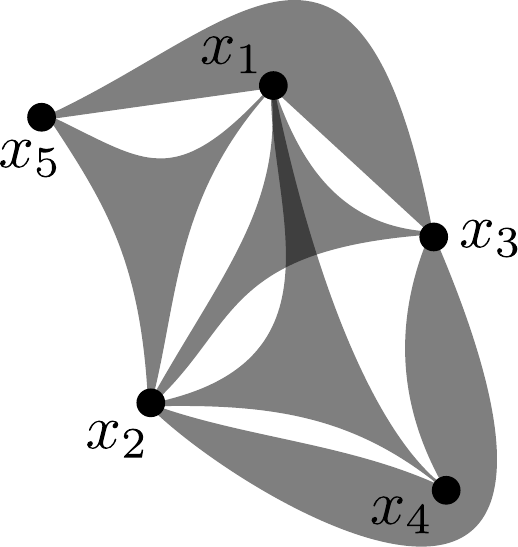}
\caption{\\ Clutter  $F'_2$}
\end{minipage}
\hspace{-0.4cm}
\begin{minipage}[h]{0.26\linewidth}
\centering
\includegraphics[height=0.7in]{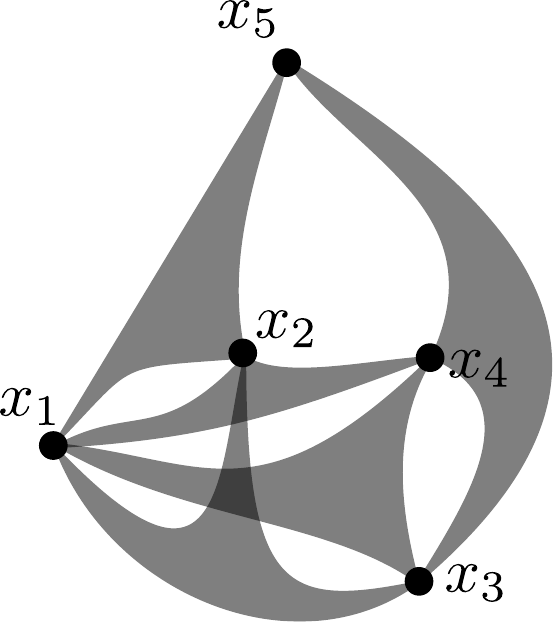}
\caption{\\ Clutter  $F'_3$}
\end{minipage}
\hspace{-0.2cm}
\begin{minipage}[h]{0.26 \linewidth}
\centering
\includegraphics[height=0.7in]{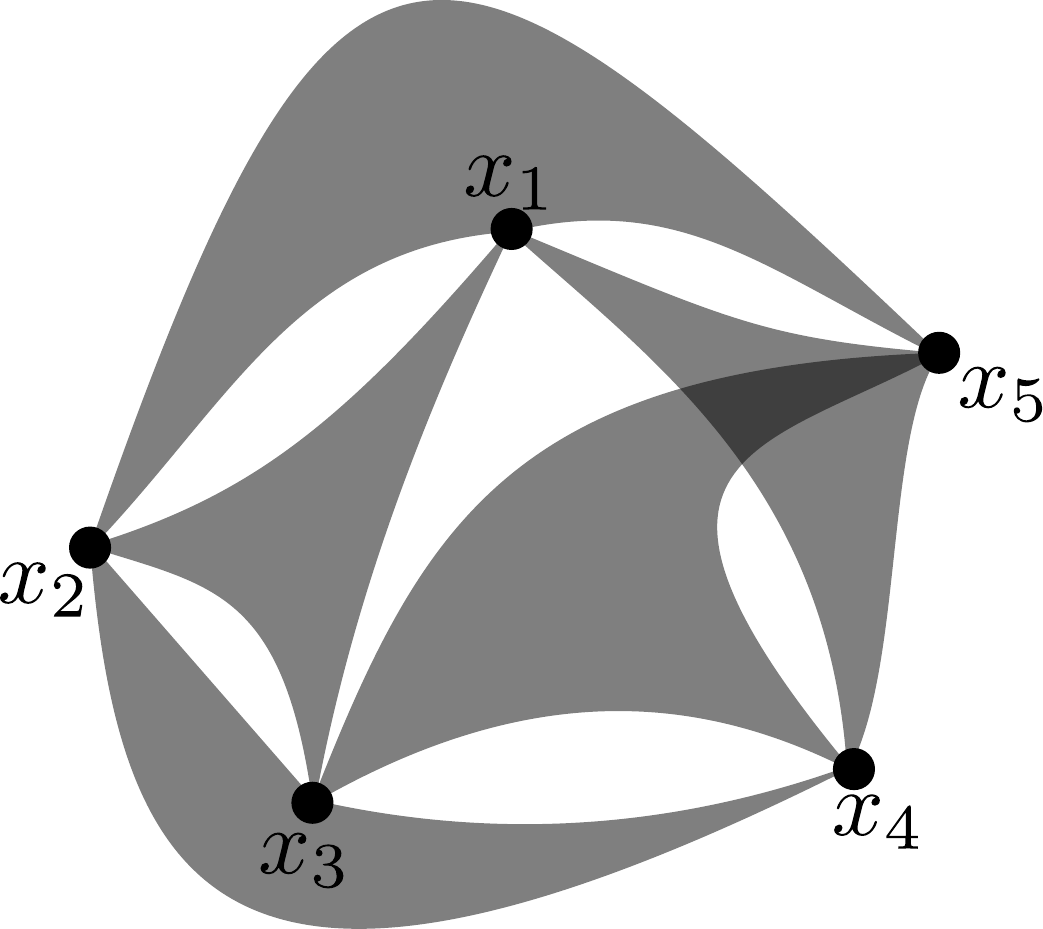}
\caption{\\ Clutter $F'_4$}
\end{minipage}
\end{figure} 


\medskip

Therefore, $F$ is of the form $\{F'_i,m_jx_6\}$, where $i=1, 2, 3, 4$ and
$m_j \in \K[x_1, \ldots, x_5]$ is a square free monomial of degree $2$. Let us study the possibilities for the last monomial. According to Lemma \ref{isomorfas} there are some orbits that coincide.

First of all notice that the incidence degree of $x_1$ and $x_2$ in both, $F'_1$ and $F'_2$ is $4$, therefore
$F=\{F'_i,m_jx_6\}$, with $i=1,2$, satisfying our hypothesis imposes $m_j \in \K[x_3,x_4,x_5]$. The stabilizer of $F'_1$ has generators $\beta=(1,2)$ and $\gamma=(3,4)$ and the stabilizer of $F'_2$ is generated by $\beta= (1,2)(4,5)$.
By Lemma \ref{isomorfas} we have $$\mathcal{O}_{\{F'_1,x_3x_5x_6\}}=\mathcal{O}_{\{F'_1,\gamma*(x_3x_5x_6)\}}. $$
By choosing one representative for each orbit we have two possibilities for $m_1$: $x_3x_4$ or $x_3x_5$, with associated clutters:

\medskip

\begin{figure}[h!] 
\captionsetup{justification=centering}
\centering
\hspace{-1cm}
\begin{minipage}[h]{0.26 \linewidth}
\centering

\includegraphics[height=.8in]{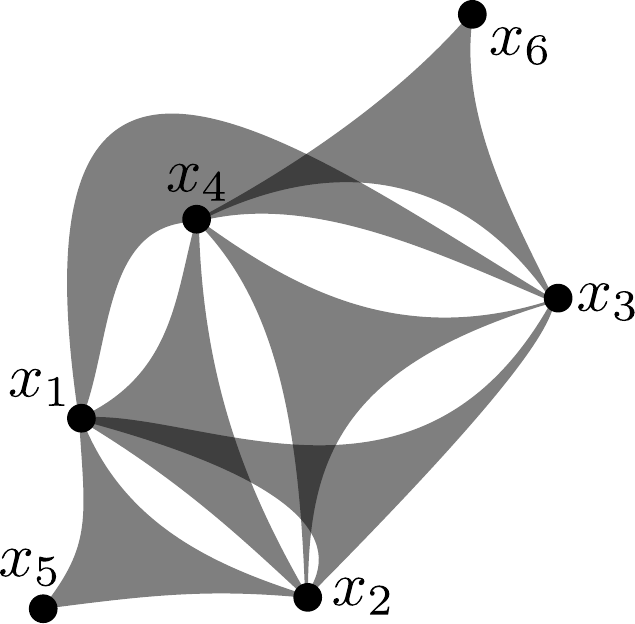}\\
\caption{Clutter  $G_1$  \\ $F'_1,x_3x_4x_6$}
\end{minipage}
\centering
\hspace{1cm}
\begin{minipage}[h]{0.26 \linewidth}
\centering
\includegraphics[height=.8in]{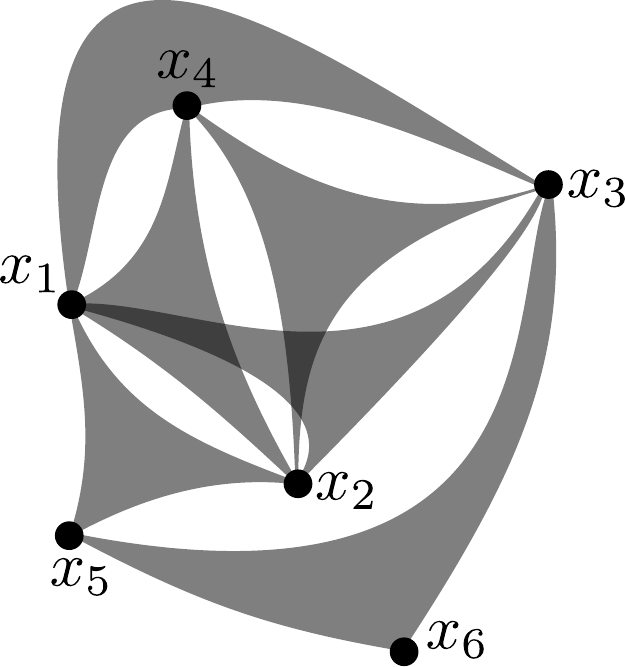}\\
\caption{Clutter $G_2$ \\ $F'_1,x_3x_5x_6$}
\end{minipage}

\end{figure} 
\medskip


By Lemma \ref{isomorfas}: 
$$\mathcal{O}_{\{F'_2,x_3x_4x_6\}}=\mathcal{O}_{\{F'_2,\beta*(x_3x_4x_6)\}}. $$
We have also two possibilities for $m_2$: $x_3x_4$ or $x_4x_5$. The associated clutters are:

\medskip

\begin{figure}[!htbp]
\captionsetup{justification=centering}
\centering
\hspace{-1cm}
\begin{minipage}[h]{0.26 \linewidth}
\centering
\includegraphics[height=.8in]{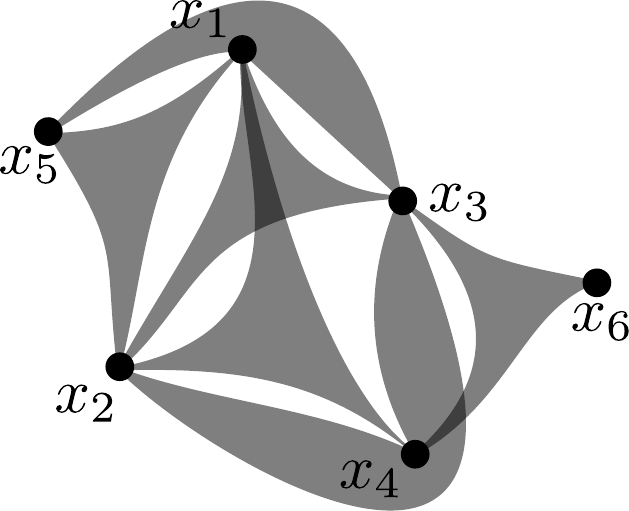}\\
\caption{Clutter $G_3$ \\ $F'_2,x_3x_4x_6$}
\end{minipage}
\hspace{2cm}
\begin{minipage}[h]{0.26 \linewidth}
\centering
\includegraphics[height=.8in]{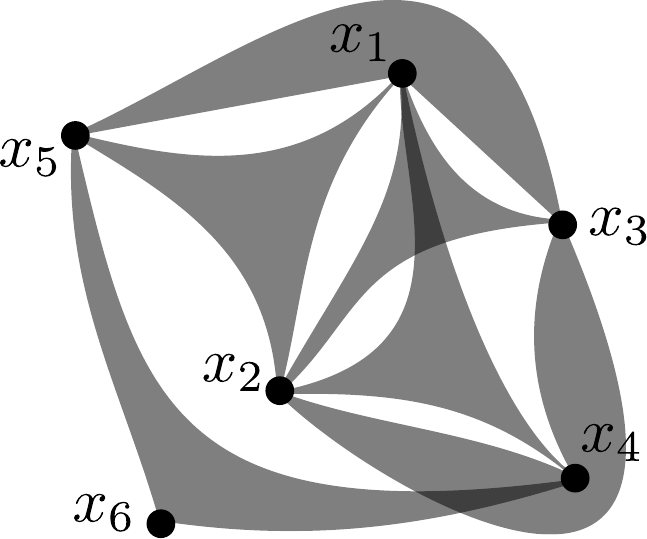}\\
\caption{Clutter $G_4$ \\ $F'_2,x_4x_5x_6$}
\end{minipage}
\end{figure} 

\medskip

In the same way, the incidence degree of $x_1$ in $F'_3$ is $4$, so $m_3 \in \K[x_2,x_3,x_4,x_5]$.

The stabilizer of $F'_3$ is generated by $\beta= (3,4)$.  By the Lemma \ref{isomorfas} we have
$$\mathcal{O}_{\{F'_3,x_2x_3x_6\}}=\mathcal{O}_{\{F'_3,\beta*(x_2x_3x_6)\}}\ \mbox{and}\ \mathcal{O}_{\{F'_3,x_3x_5x_6\}}=\mathcal{O}_{\{F'_3,\beta*(x_3x_5x_6)\}}. $$
Taking one representative for each orbit, the last monomial can be: $x_2x_3x_6$, $x_2x_5x_6$, $x_3x_4x_6$ or $x_3x_5x_6$. The associated clutters are:

\begin{figure}[h!] 
\centering
\hspace{-0.3cm}
\begin{minipage}[h]{0.26 \linewidth}
\centering
\includegraphics[height=.7in]{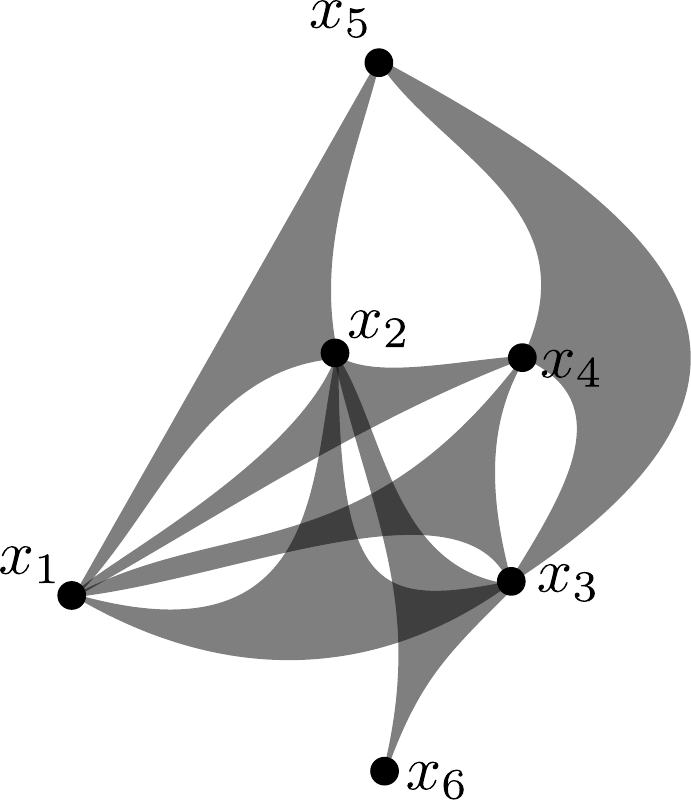}\\
\caption{\\ Clutter  $F'_1$}
\end{minipage}
\hspace{-0.5cm}
\begin{minipage}[h]{0.26 \linewidth}
\centering
\includegraphics[height=.7in]{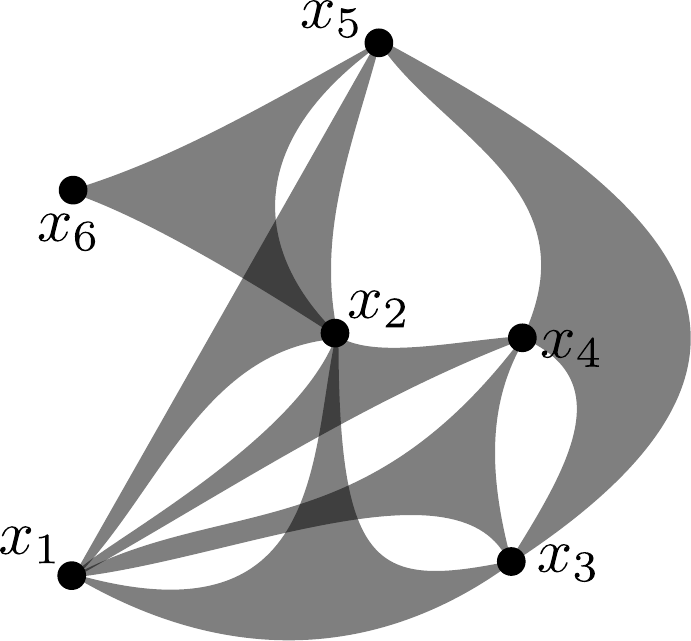}\\
\caption{\\ Clutter  $F'_2$}
\end{minipage}
\hspace{-0.4cm}
\begin{minipage}[h]{0.26\linewidth}
\centering
\includegraphics[height=.7in]{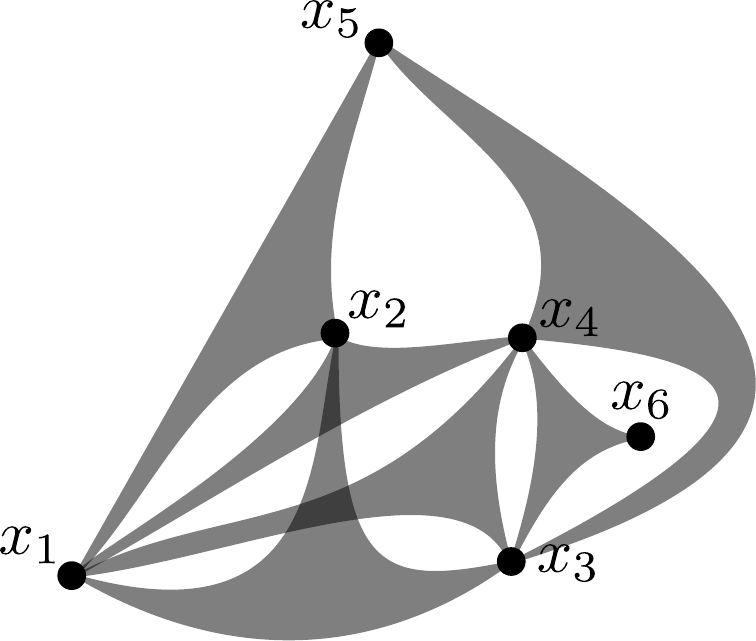}
\caption{\\ Clutter  $F'_3$}
\end{minipage}
\hspace{-0.5cm}
\begin{minipage}[h]{0.26 \linewidth}
\centering
\includegraphics[height=.7in]{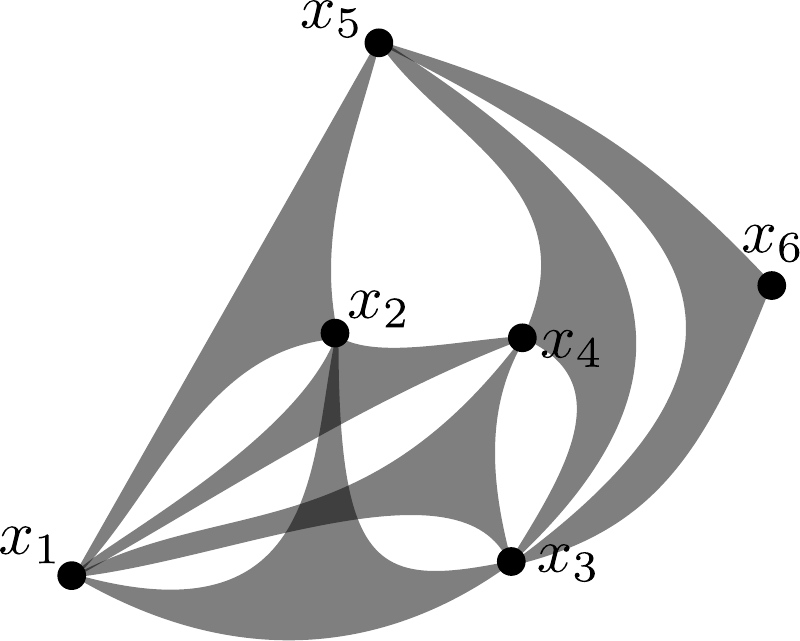}
\caption{\\ Clutter $F'_4$}
\end{minipage}
\end{figure}

\medskip
The stabilizer of $F'_4$ is generated by $\beta = (1,2,3,4,5)$.  By the Lemma \ref{isomorfas}
$$\mathcal{O}_{\{F'_4,x_1x_2x_6\}}=\mathcal{O}_{\{F'_4,\beta^{i}*(x_1x_2x_6)\}} \mbox{ and} \quad \quad \qquad$$ $$\mathcal{O}_{\{F'_4,x_1x_3x_6\}}=\mathcal{O}_{\{F'_4,\beta^i*(x_1x_3x_6)\}} \text{ for } i=1,2,3. $$
The possibilities for the last monomial are $x_1x_2x_6$ or $x_1x_3x_6$. The associated clutters are:

\medskip

\begin{figure}[h!] 
\captionsetup{justification=centering}
\centering
\hspace{-1cm}
\begin{minipage}[h]{0.28 \linewidth}
\centering
\includegraphics[height=0.7in]{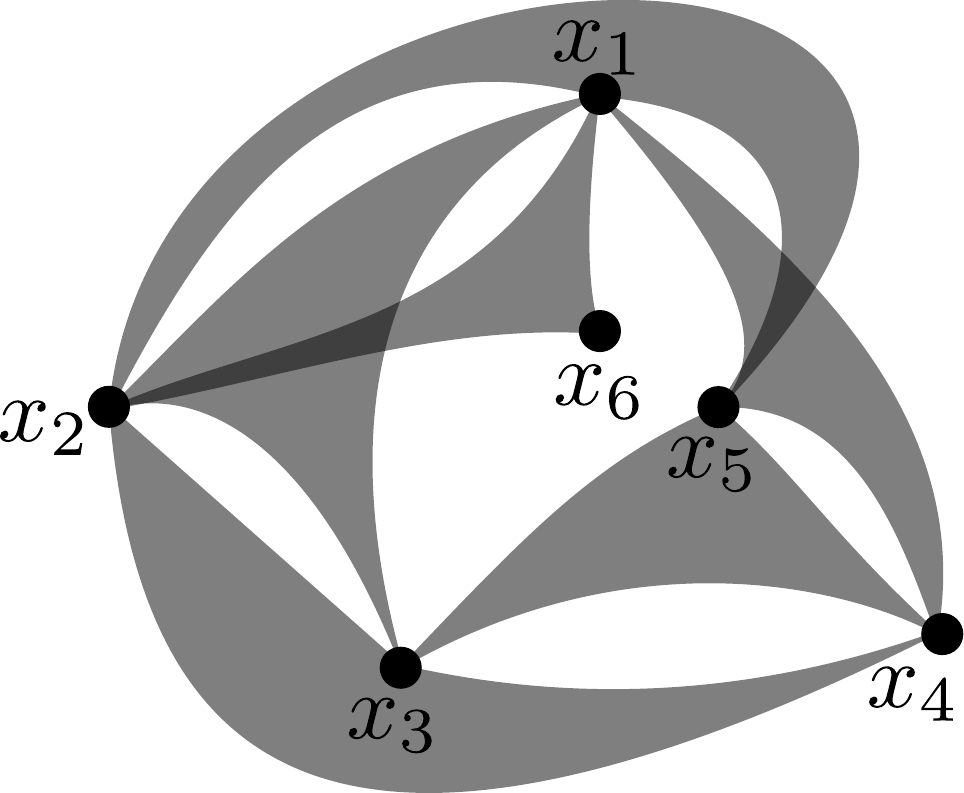}\\
\caption{Clutter $G_9$ \\ $F'_4,x_1x_2x_6$}
\end{minipage}
\hspace{2cm}
\begin{minipage}[h]{0.29 \linewidth}
\centering
\includegraphics[height=0.7in]{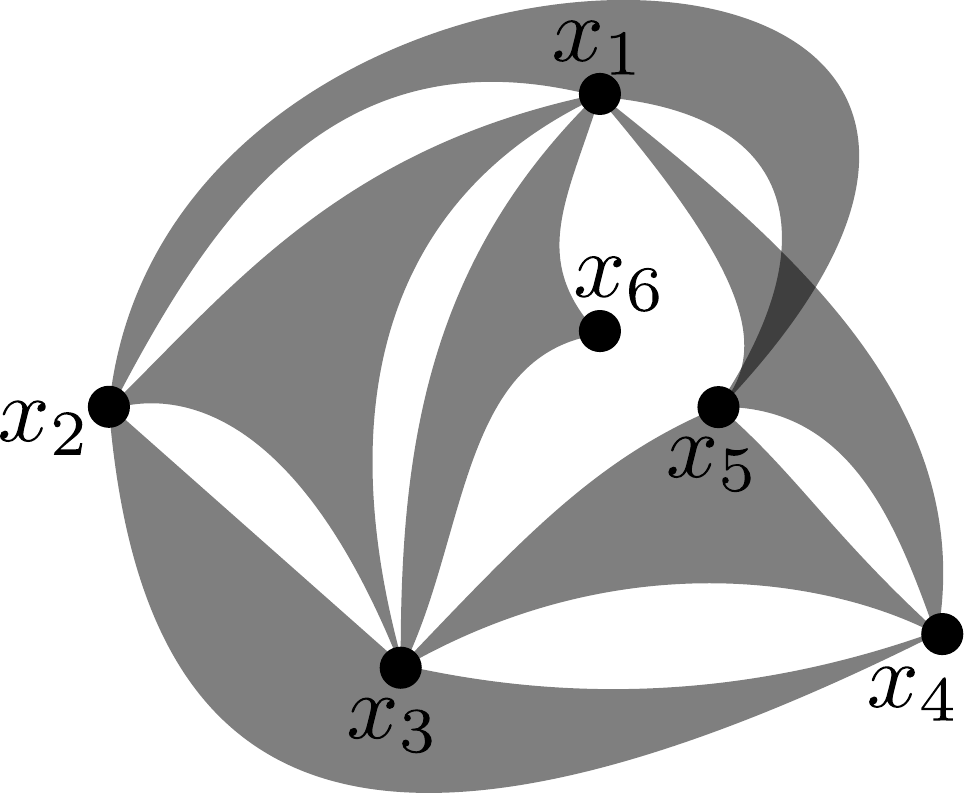}\\
\caption{Clutter $G_{10}$ \\ $F'_4,x_1x_3x_6$}
\end{minipage}
\hspace{-1cm}
\end{figure}

\medskip

Using Lemma \ref{cone} and Lemma \ref{sequencia de grau} it is easy to see that these ten sets represent nonisomorphic monomial Cremona transformations.\\

In fact, $G_1$ can not be equivalent to other by the Lemma \ref{sequencia de grau}.
Notice also that for $i=2,\ldots,10$ the clusters $G_i$ have, up to isomorphism, four distinct types of maximal cones, whose bases can be represented by the following graphs.

\medskip

\begin{figure}[!h]
\captionsetup{justification=centering}
\begin{center}
\begin{minipage}[h]{0.15 \linewidth}

\begin{center}
\hspace{2cm}
\includegraphics[height=.7in]{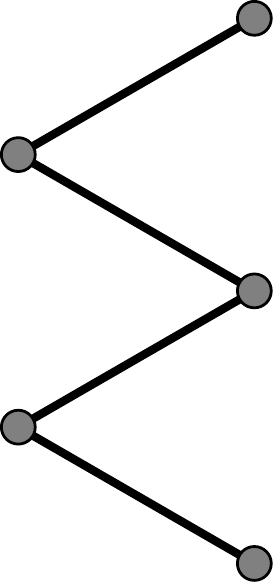}\\
\caption{\\Base $C_1$}
\end{center}
\end{minipage}
\hspace{0.4in}
\begin{minipage}[h]{0.15 \linewidth}
\begin{center}
\includegraphics[height=.7in]{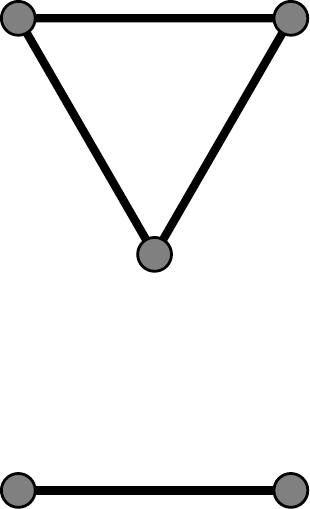}\\
\caption{\\Base $C_2$}
\end{center}
\end{minipage}
\hspace{0.4in}
\begin{minipage}[h]{0.15 \linewidth}
\begin{center}
\includegraphics[height=.7in]{figuras2/fig5.pdf}\\
\caption{\\Base $C_3$}
\end{center}
\end{minipage}
\hspace{0.4in}
\begin{minipage}[h]{0.15 \linewidth}
\begin{center}
\includegraphics[height=.7in]{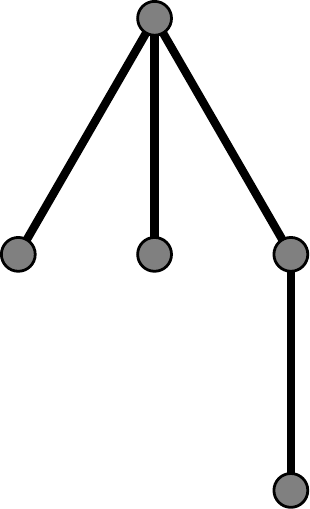}\\
\caption{\\Base $C_4$}
\end{center}
\end{minipage}
\end{center}
\end{figure}
\medskip

The following matrix shows that the ten Cremona sets obtained are nonequivalent, by
having distinct incidence sequence or maximal cones.  
\vspace{0.3cm}
\begin{center}
\begin{tabular}{|c|c|c|c|c|}
  \hline
$(4,4,4,3,2,1)$ & $G_2$ & $G_3$ & $G_5$ & $G_7$    \\
\hline

 MAXIMAL  &  $C_2$ e             &  $C_1$ e              &  $C_1, \, C_3$  & $C_3$ e  \\
CONE         & $C_3 (\times 2)$ & $C_3 (\times 2)$ & e $C_4$            & $C_4 (\times2)$ \\
   \hline
\end{tabular}
\end{center}
\vspace{0.3cm}
 \begin{center}
\begin{tabular}{|c|c|c|c|c|c|}
  \hline
$(4,4,3,3,3,1)$ & $G_4$ & $G_6$ & $G_8$ & $G_9$ & $G_{10}$    \\
   \hline
 MAXIMAL  & $C_3 (\times 2)$  &  $C_3$ e $C_4$   &  $C_1$  e $C_3$  & $C_4 (\times 2)$ & $C_1 (\times 2)$  \\
CONE         &  &  &   &  & \\
\hline
\end{tabular}
\end{center}
\vspace{0.3cm}
\end{proof}

\begin{prop} \label{prop:d3n6t2}
There are $20$ equivalence classes of square free Cremona sets of degree $3$ in $\K[x_1, \ldots, x_6]$
of type $2$.
\end{prop}

\begin{proof} Let $F $ be such a set, that is, $F=\{x_1F',m\}$, where $F' \subset \K[x_2, \ldots, x_6]$ is a set of five square-free monomials of degree $2$ 
and $m \in \K[x_2, \ldots, x_6]$ is a degree $3$ square free monomial. By PRP, Corollary \ref{cor:PRP}, $F'$ is a Cremona set. 
There are four equivalence classes of such Cremona sets (c.f. \cite{SV2}) and also Corollary \ref{cor:p4}. The associated graphs can be represented as:
\begin{figure}[h!] 
\captionsetup{justification=centering}
\centering
\hspace{-0.3cm}
\begin{minipage}[h]{0.26 \linewidth}
\centering
\includegraphics[height=0.7in]{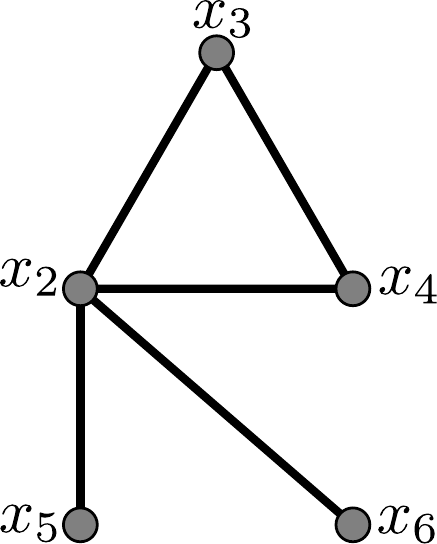}\\
\caption{\\ Graph $F'_1$}
\end{minipage}
\hspace{-0.5cm}
\begin{minipage}[h]{0.26 \linewidth}
\centering
\includegraphics[height=0.7in]{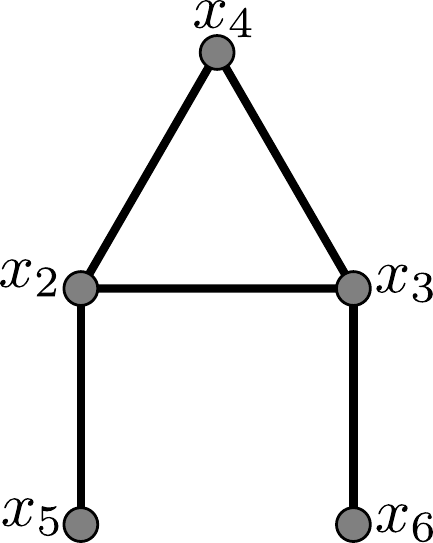}\\
\caption{\\ Graph $F'_2$}
\end{minipage}
\hspace{-0.4cm}
\begin{minipage}[h]{0.26\linewidth}
\centering
\includegraphics[height=.7in]{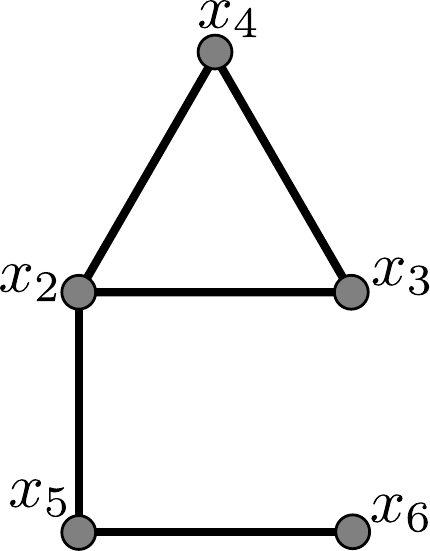}\\
\caption{\\ Graph $F'_3$}
\end{minipage}
\hspace{-0.2cm}
\begin{minipage}[h]{0.26 \linewidth}
\centering
\includegraphics[height=0.7in]{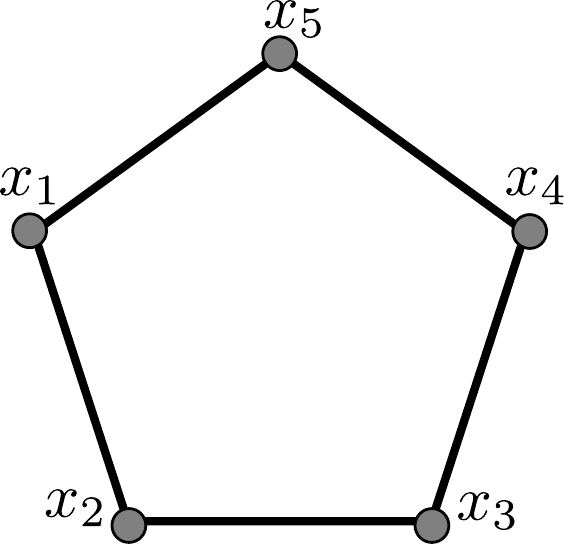}\\
\caption{\\ Graph $F'_4$}
\end{minipage}
\end{figure}


So, $F$ is of the form $\{x_1F'_i,m_i\}$, where $i=1, 2, 3, 4$ and $m \in \K[x_2, \ldots, x_6]$
is a cubic monomial. Let us analyze the possibilities for $m_i$ according to the permutations that stabilize
$x_1F_i$, see Lemma \ref{isomorfas}, in order to exclude transformations in the same orbit.
The symmetries of the graph associated with $F_i$ are helpful.

The stabilizer of $F'_1$ is generated by $\beta=(3,4)$ and $\gamma=(5,6)$.
By the Lemma \ref{isomorfas} we have
$$\mathcal{O}_{\{x_1F'_1,x_2x_3x_5\}}=\mathcal{O}_{\{x_1F'_1,\beta*(x_2x_3x_5)\}}=\mathcal{O}_{\{x_1F'_1,\gamma*(x_2x_3x_5)\}}=\mathcal{O}_{\{x_1F'_1,\beta \gamma*(x_2x_3x_5)\}},$$
$$\mathcal{O}_{\{x_1F'_1,x_3x_4x_5\}}=\mathcal{O}_{\{x_1F'_1,\gamma*(x_3x_4x_5)\}}\ \mbox{ and} \ 
\mathcal{O}_{\{x_1F'_1,x_3x_5x_6\}}=\mathcal{O}_{\{x_1F'_1,\beta*(x_3x_5x_6)\}}. $$
So the possibilities for $m_1$ are $x_2x_3x_4$, $x_2x_3x_5$, $x_2x_5x_6$, $x_3x_4x_5$ and $x_3x_5x_6$. The associated clutters are the following ones:

\medskip

\begin{figure}[h!]
\captionsetup{justification=centering} 
\centering
\hspace{-0.3cm}
\begin{minipage}[h]{0.33 \linewidth}
\centering
\includegraphics[height=.7in]{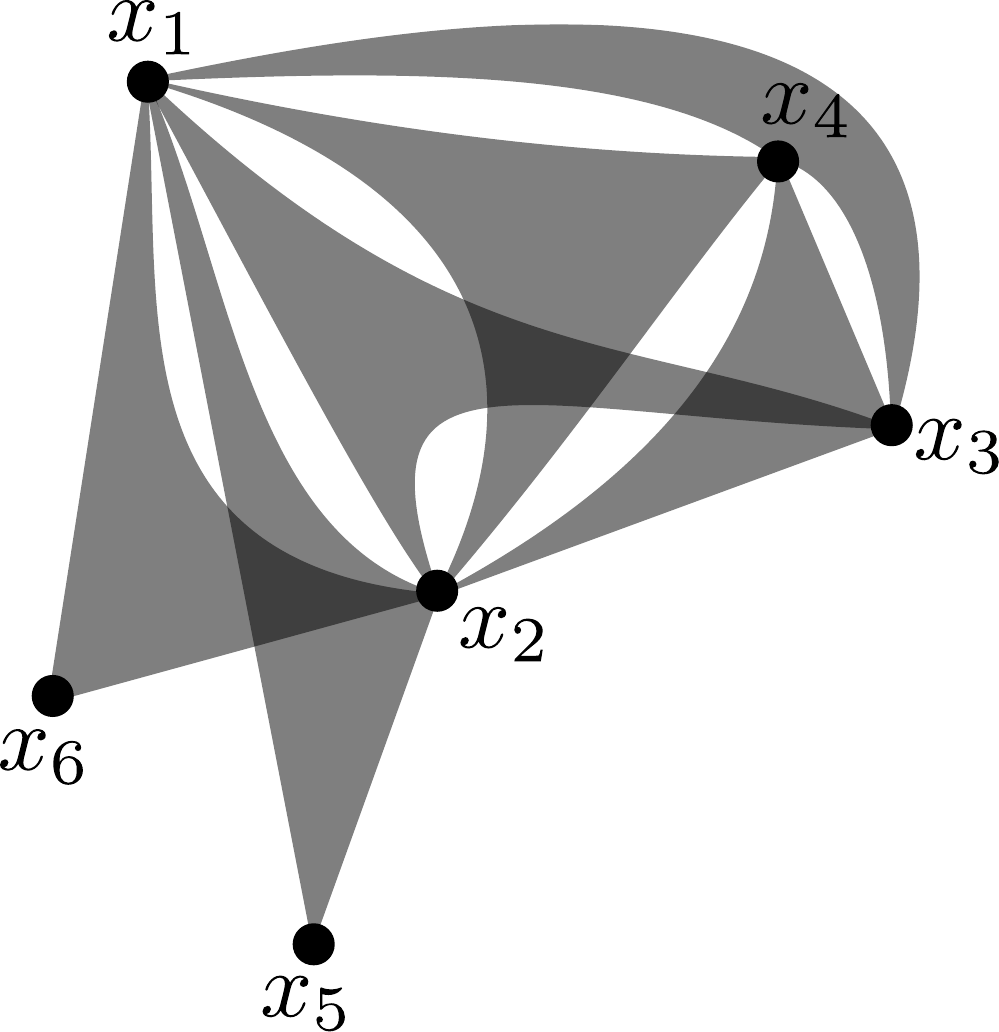}
\caption{Clutter $F_1$ \\ $x_1F'_1,x_2x_3x_4$}
\end{minipage}
\hspace{-0.5cm}
\begin{minipage}[h]{0.33 \linewidth}
\centering
\includegraphics[height=.7in]{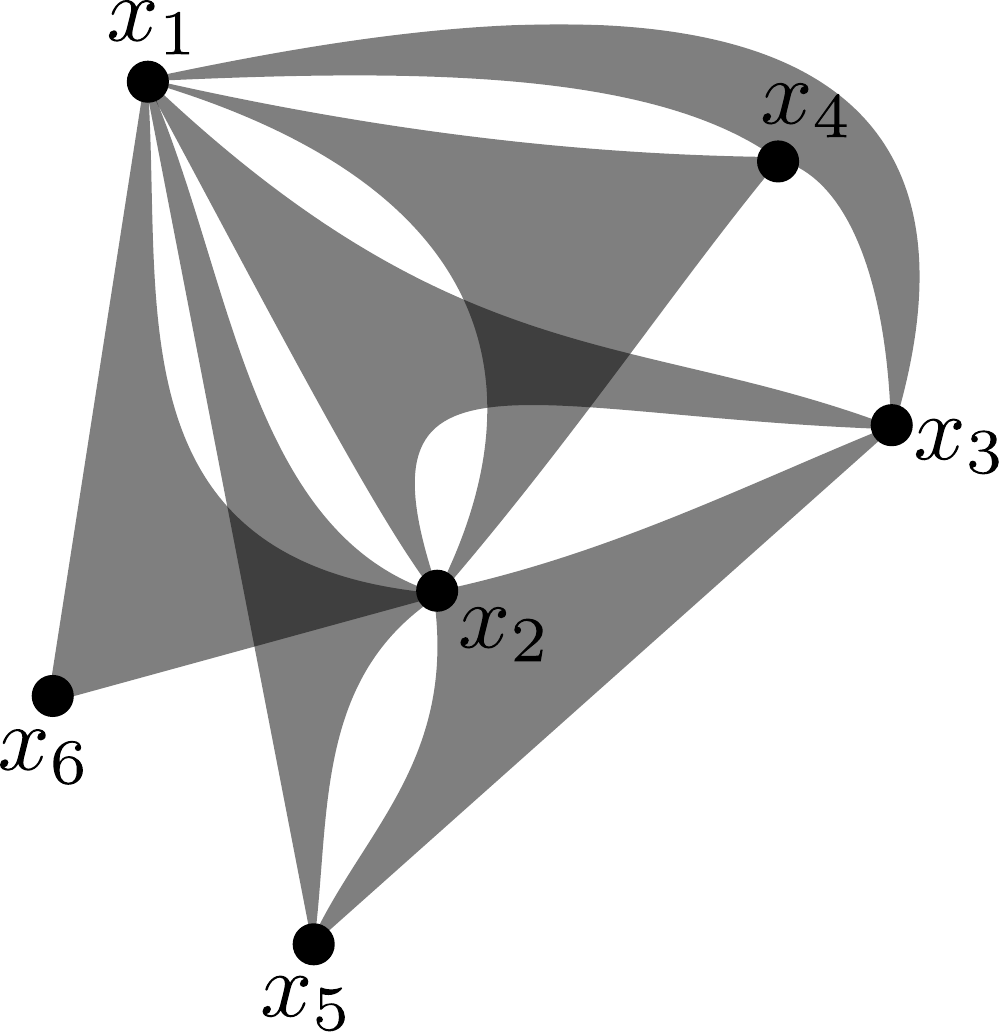}
\caption{Clutter $F_2$ \\ $x_1F'_1,x_2x_3x_5$}
\end{minipage}
\hspace{-0.5cm}
\begin{minipage}[h]{0.33 \linewidth}
\centering
\includegraphics[height=.7in]{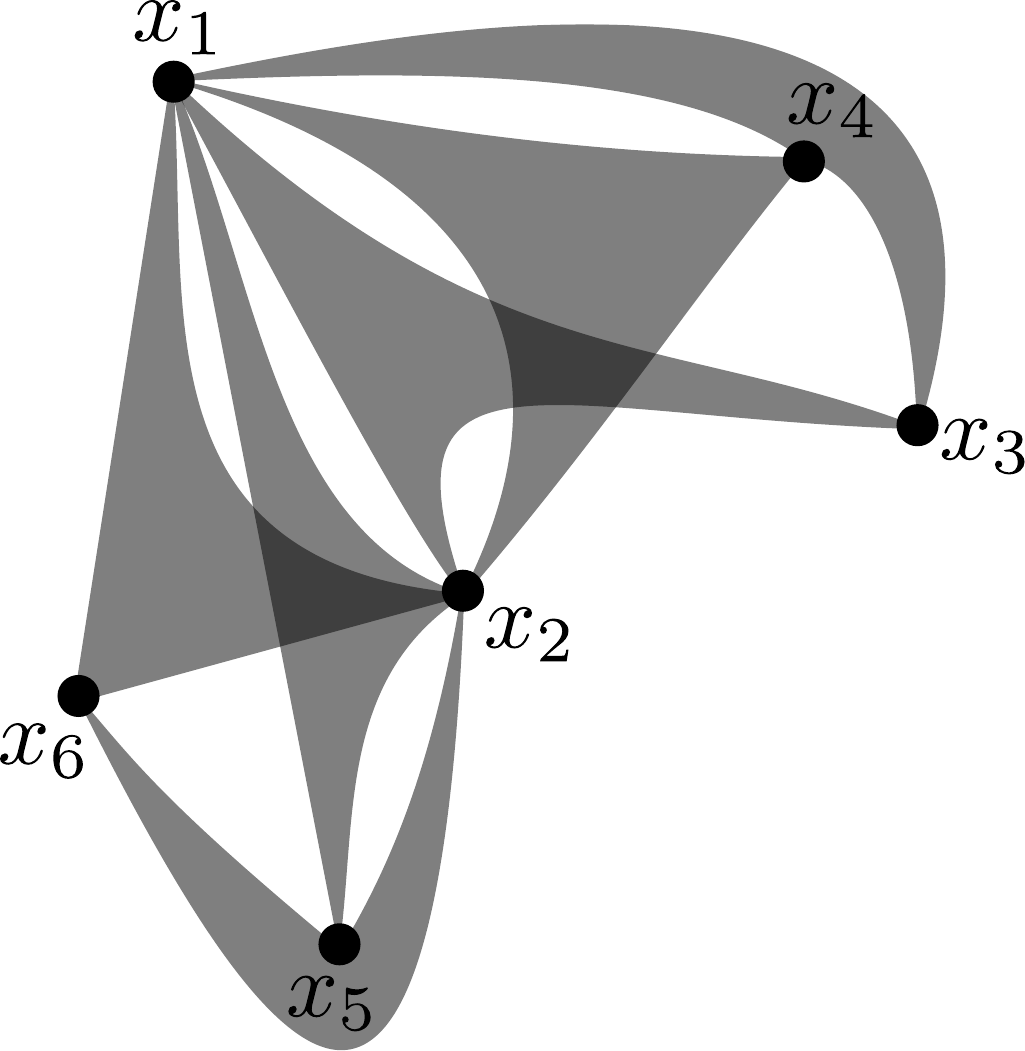}
\caption{Clutter $F_3$ \\ $x_1F'_1,x_2x_5x_6$}
\end{minipage}
\end{figure}


\begin{figure}[h!] 
\captionsetup{justification=centering}
\centering
\hspace{-0.5cm}
\begin{minipage}[h]{0.35 \linewidth}
\centering
\includegraphics[height=.7in]{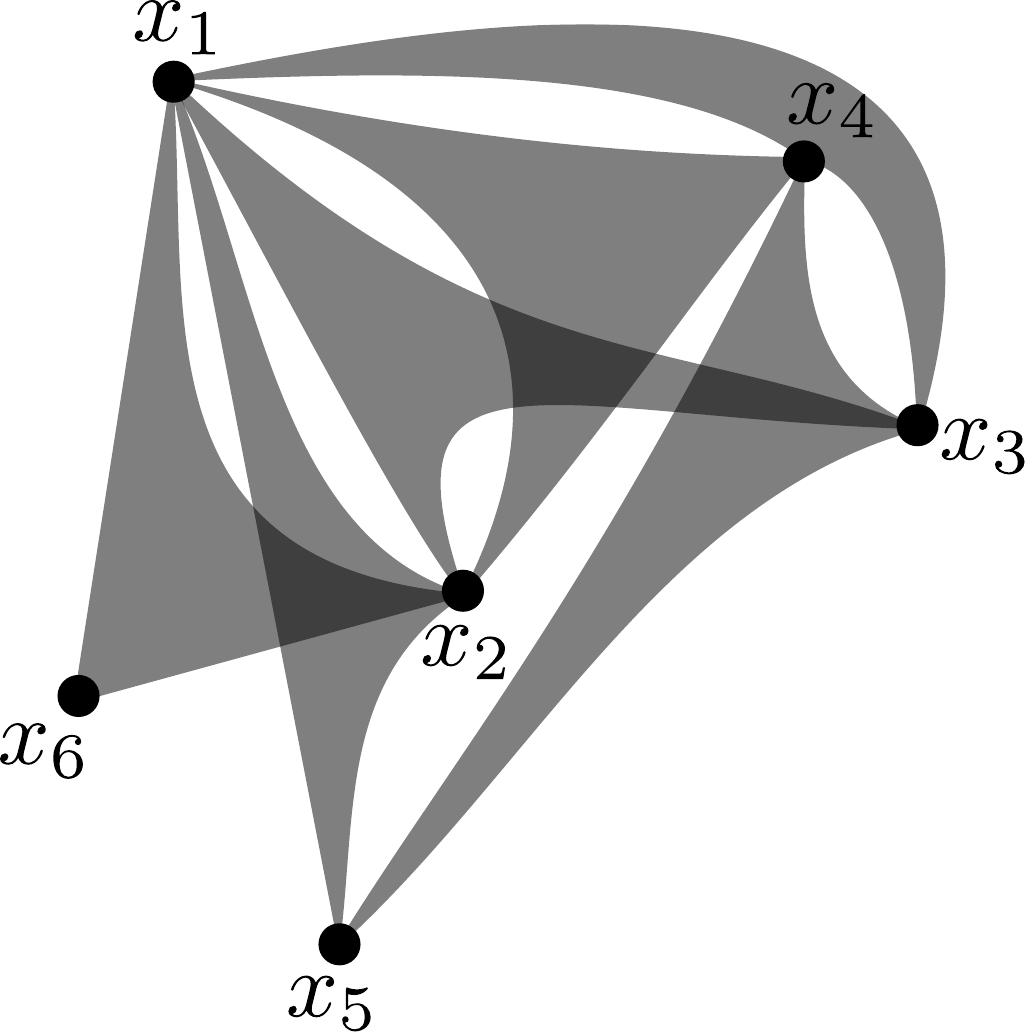}\\
\caption{Clutter $F_4$ \\ $x_1F'_1,x_3x_4x_5$}
\end{minipage}
\hspace{2cm}
\begin{minipage}[h]{0.35 \linewidth}
\centering
\includegraphics[height=.7in]{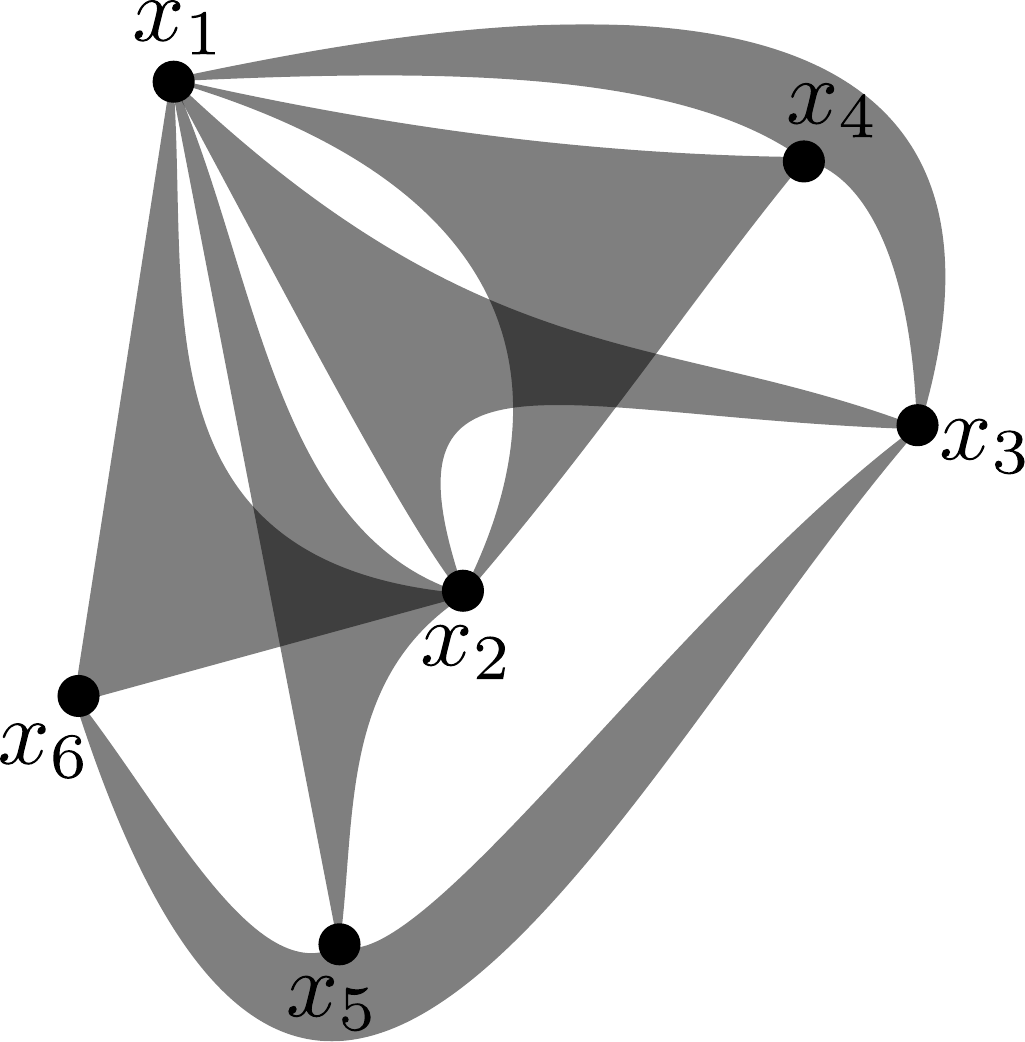}\\
\caption{Clutter $F_5$ \\ $x_1F'_1,x_3x_5x_6$}
\end{minipage}
\end{figure} 


\medskip
The stabilizer of $F'_2$ is generated by $\beta=(2,3)(5,6)$. By the Lemma \ref{isomorfas} we have
$$\mathcal{O}_{\{x_1F'_2,x_2x_3x_5\}}=\mathcal{O}_{\{x_1F'_2,\beta*(x_2x_3x_5)\}}, \, \mathcal{O}_{\{x_1F'_2,x_2x_4x_5\}}=\mathcal{O}_{\{x_1F'_2,\gamma*(x_2x_4x_5)\}},$$
$$\mathcal{O}_{\{x_1F'_2,x_2x_4x_6\}}=\mathcal{O}_{\{x_1F'_2,\beta*(x_2x_4x_6)\}} \mbox{ and } \,
\mathcal{O}_{\{x_1F'_2,x_2x_5x_6\}}=\mathcal{O}_{\{x_1F'_2,\beta*(x_2x_5x_6)\}}. $$
The possibilities for $m_2$ are $x_2x_3x_4$, $x_2x_3x_5$, $x_2x_4x_5$, $x_2x_4x_6$, $x_2x_5x_6$ and $x_4x_5x_6$. The clutters associated to each of them are:

\medskip

\begin{figure}[!htbp]
\captionsetup{justification=centering}
\centering
\hspace{-.5cm}
\begin{minipage}[h]{0.33 \linewidth}
\centering
\includegraphics[height=.7in]{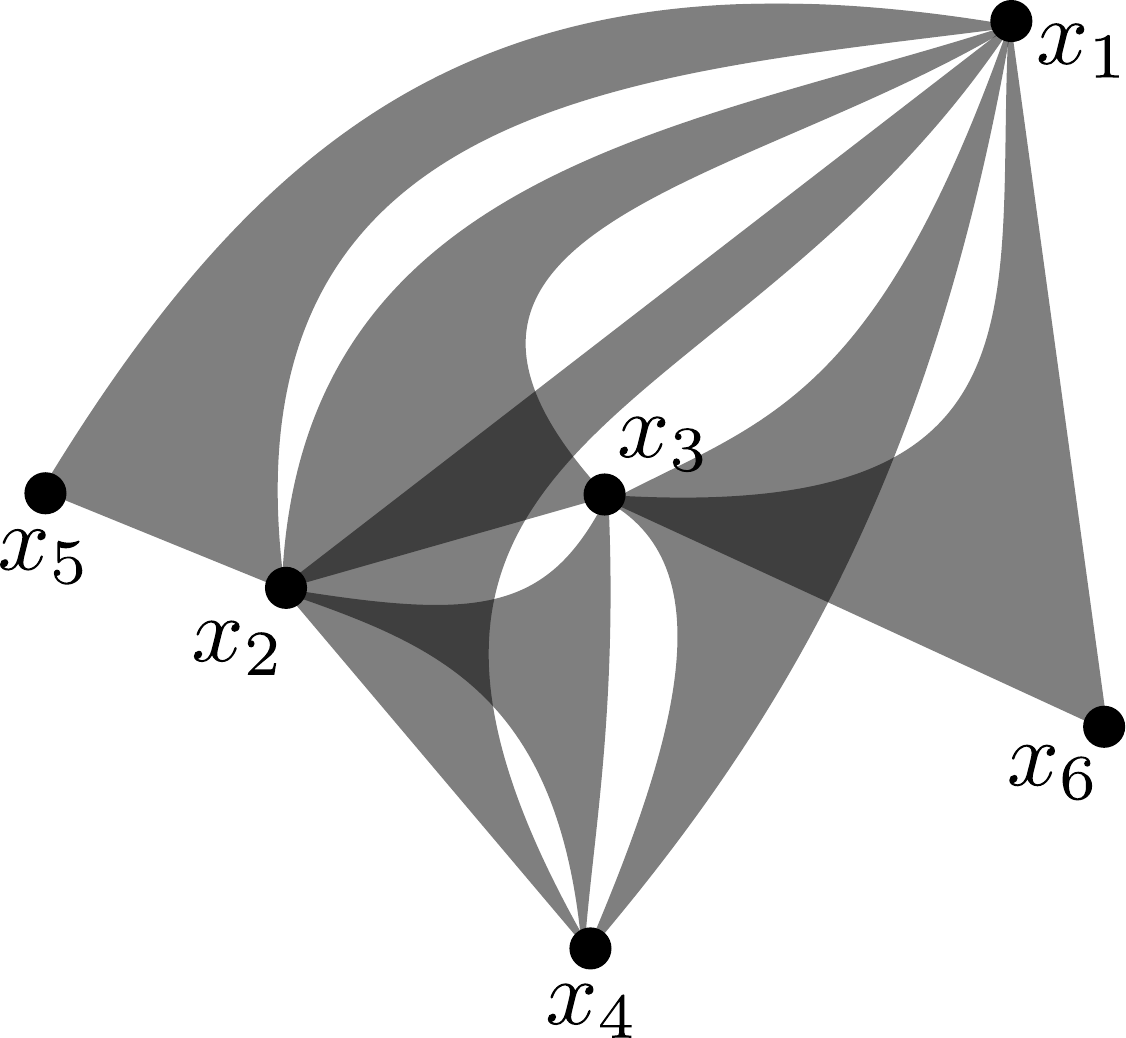}\\
\caption{Clutter $F_6$ \\ $x_1F'_2,x_2x_3x_4$}
\end{minipage}
\hspace{-0.5cm}
\begin{minipage}[h]{0.33 \linewidth}
\centering
\includegraphics[height=.7in]{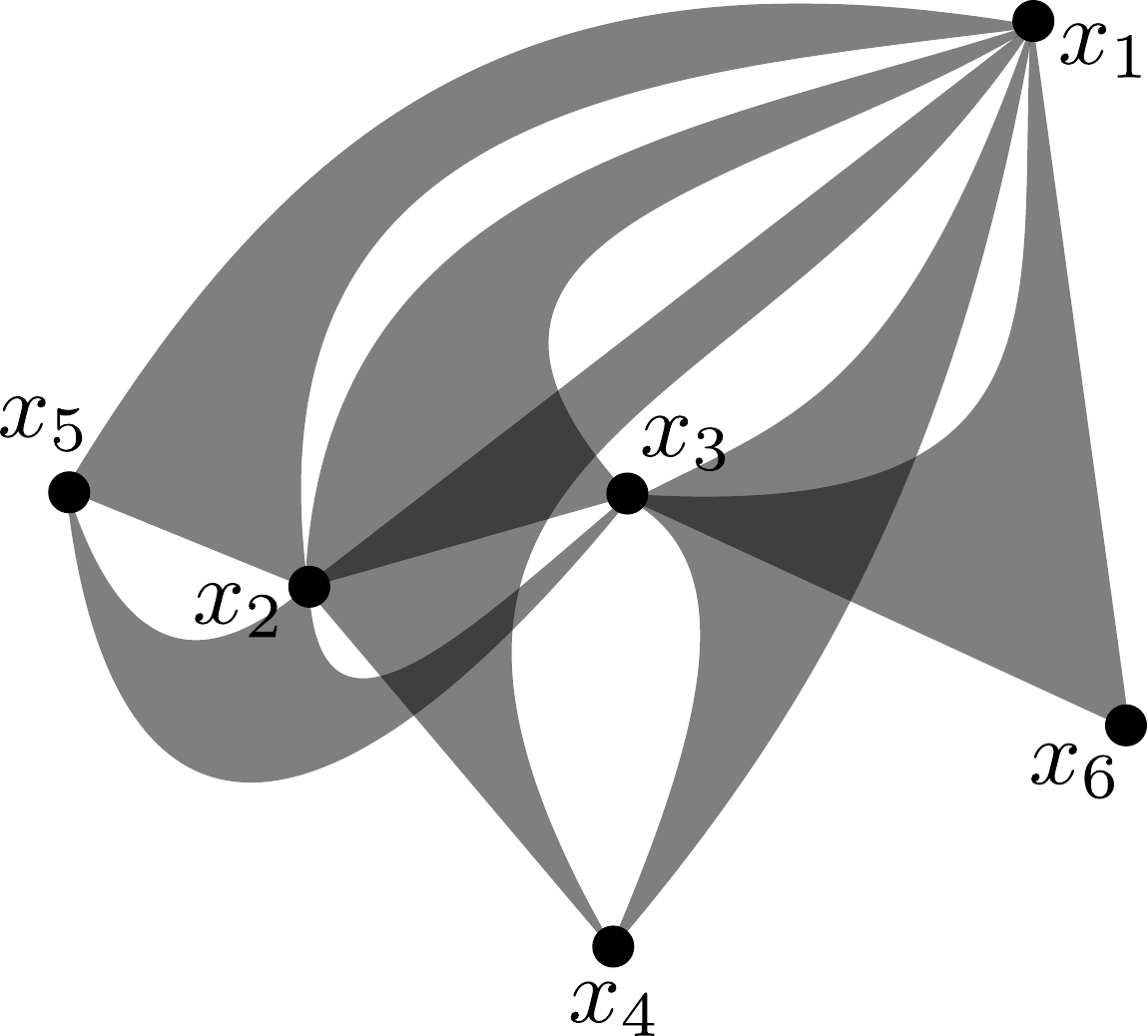}\\
\caption{Clutter $F_7$ \\ $x_1F'_2,x_2x_3x_5$}
\end{minipage}
\hspace{-0.5cm}
\begin{minipage}[h]{0.33 \linewidth}
\centering
\includegraphics[height=.7in]{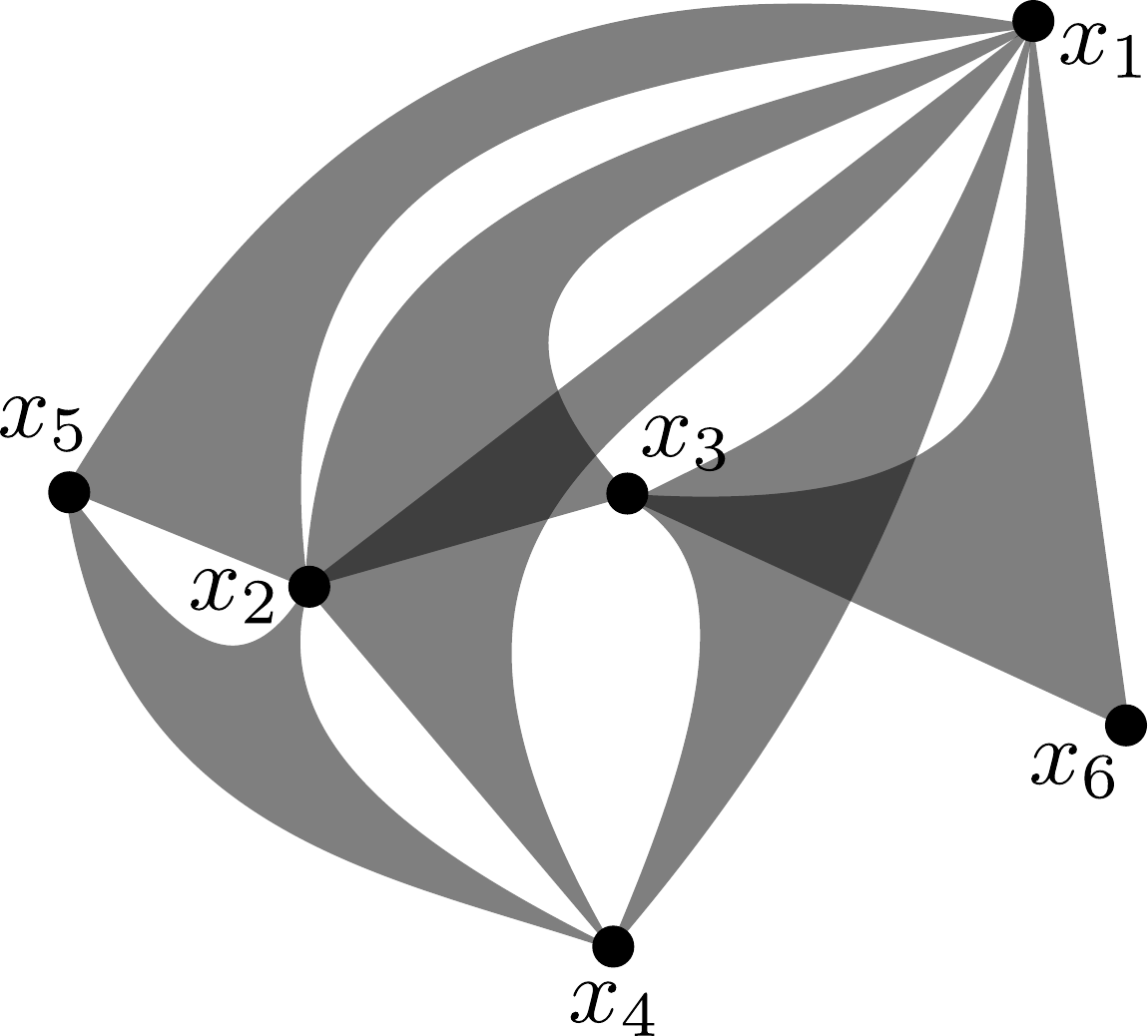}\\
\caption{Clutter $F_8$ \\ $x_1F'_2,x_2x_4x_5$}
\end{minipage}
\end{figure}

%

\medskip

\begin{figure}[h!] 

\captionsetup{justification=centering}
\centering
\hspace{-0.5cm}
\begin{minipage}[h]{0.35 \linewidth}
\centering
\includegraphics[height=.7in]{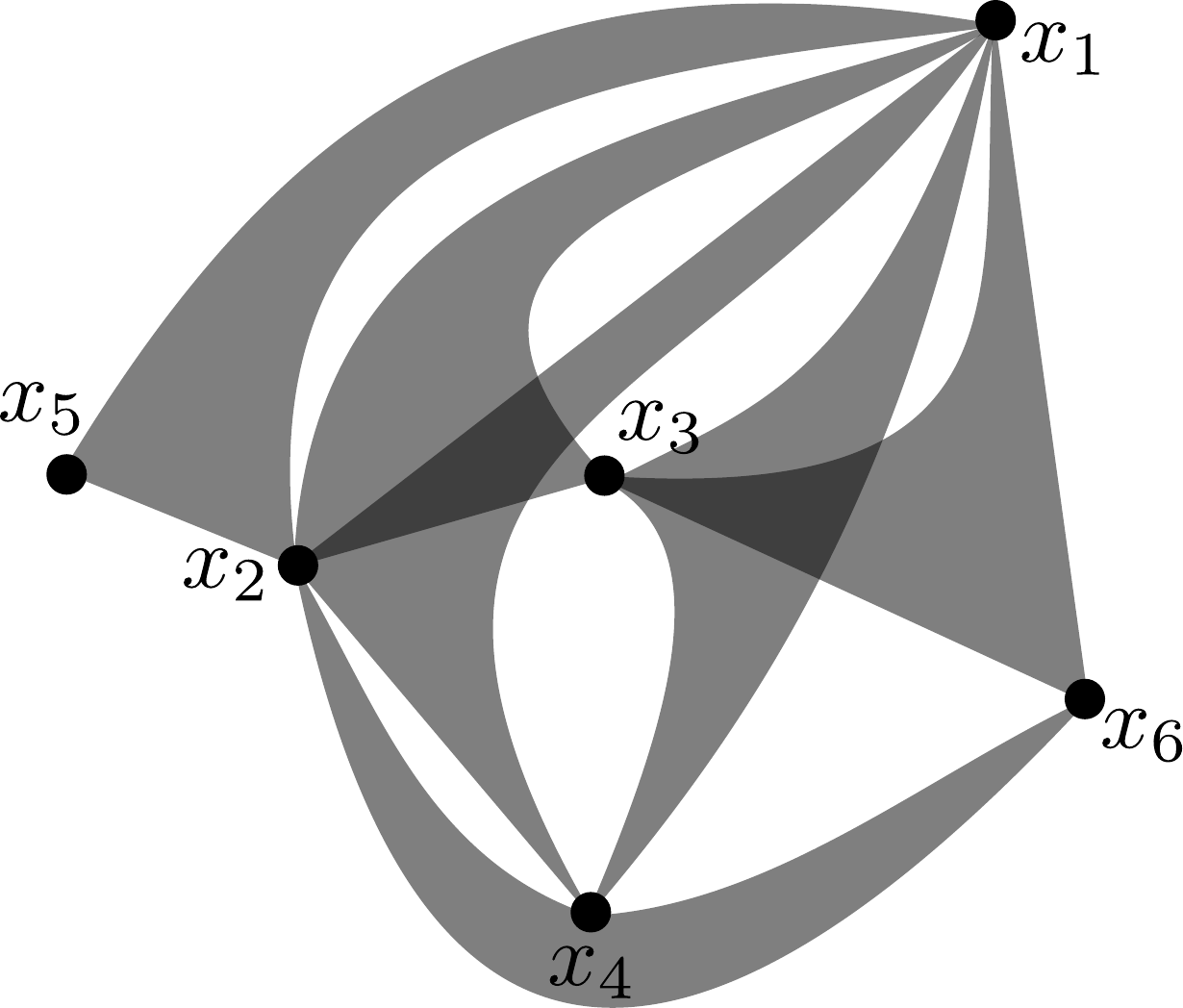}\\
\caption{Clutter $F_9$ \\ $x_1F'_2,x_2x_4x_6$}
\end{minipage}
\hspace{-0.5cm}
\begin{minipage}[h]{0.35 \linewidth}
\centering
\includegraphics[height=.7in]{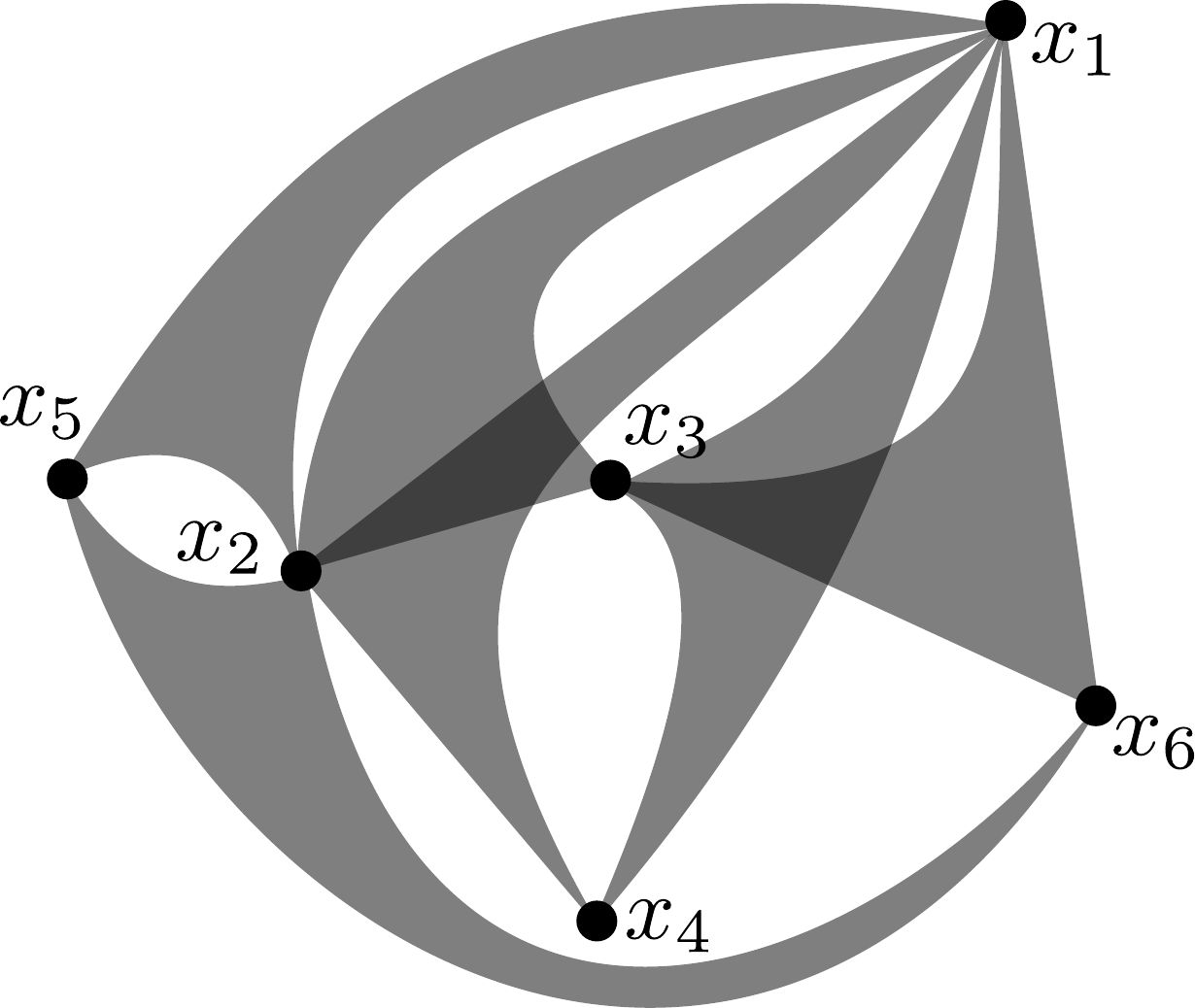}\\
\caption{Clutter $F_{10}$ \\ $x_1F'_2,x_2x_5x_6$}
\end{minipage}
\hspace{-.5cm}
\begin{minipage}[h]{0.35 \linewidth}
\centering
\includegraphics[height=.7in]{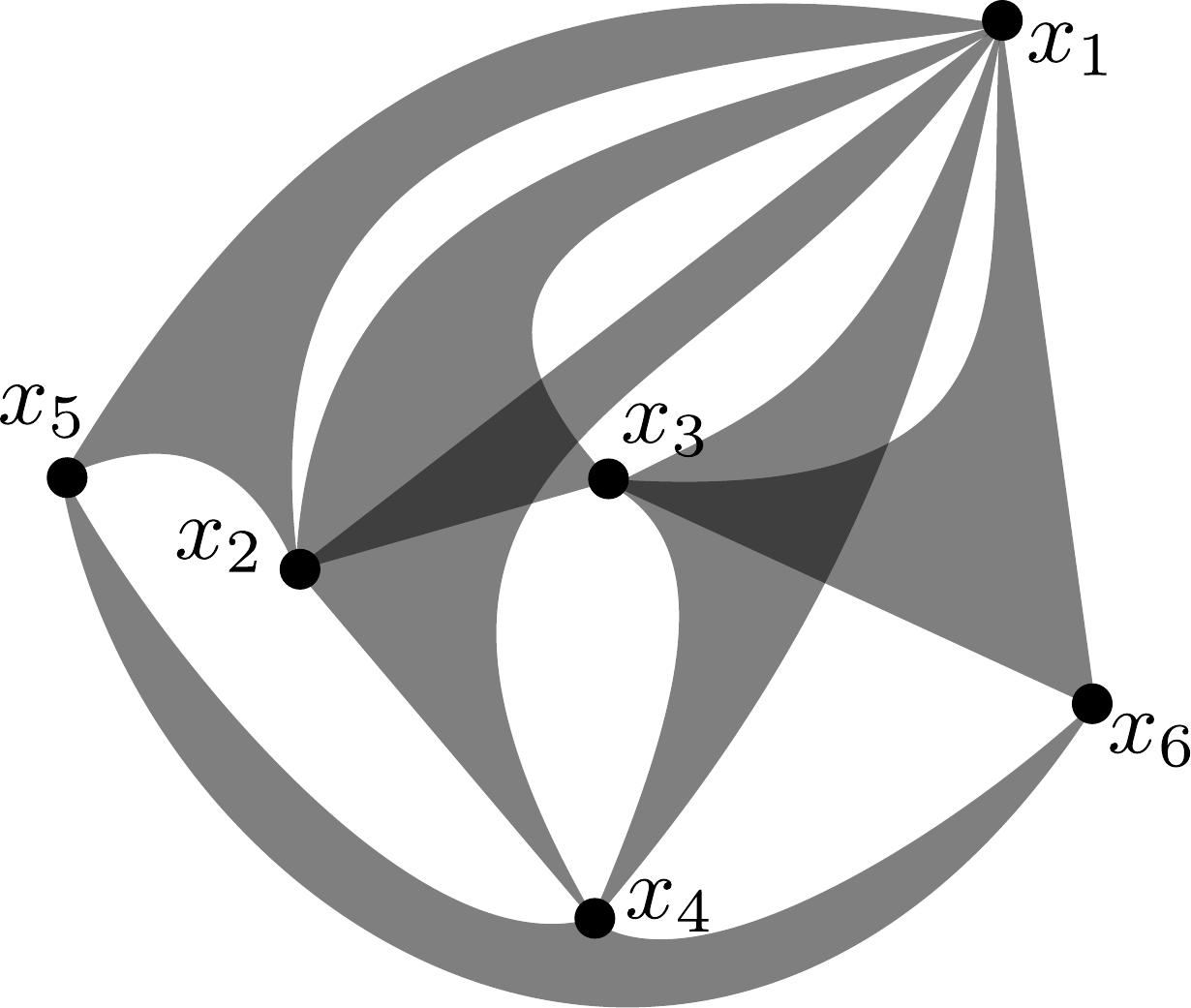}\\
\caption{Clutter $F_{11}$ \\ $x_1F'_2,x_4x_5x_6$}
\end{minipage}
\end{figure}


\medskip

The stabilizer of $F'_3$ is generated by $\beta=(3,4)$. By the Lemma \ref{isomorfas} we have
$$\mathcal{O}_{\{x_1F'_3,x_2x_3x_5\}}=\mathcal{O}_{\{x_1F'_3,\beta*(x_2x_3x_5)\}}, \, \mathcal{O}_{\{x_1F'_3,x_2x_3x_6\}}=\mathcal{O}_{\{x_1F'_3,\gamma*(x_2x_3x_6)\}} \mbox{and}$$
$$\mathcal{O}_{\{x_1F'_3,x_3x_5x_6\}}=\mathcal{O}_{\{x_1F'_3,\beta*(x_3x_5x_6)\}}. $$
The possibilities for $m_3$ are $x_2x_3x_4$, $x_2x_3x_5$, $x_2x_3x_6$, $x_2x_5x_6$, $x_3x_4x_5$, $x_3x_4x_6$ and $x_3x_5x_6$.

\begin{figure}[h!] 
\centering
\hspace{-0.3cm}
\begin{minipage}[h]{0.26 \linewidth}
\centering
\includegraphics[height=.7in]{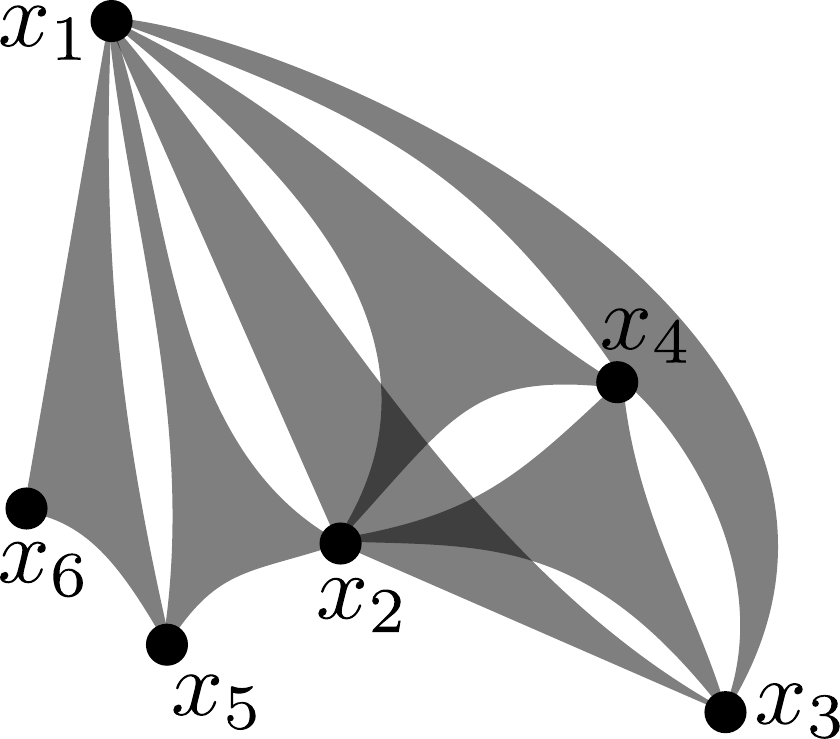}\\
\caption{\\ Clutter  $F'_1$}
\end{minipage}
\hspace{-0.5cm}
\begin{minipage}[h]{0.26 \linewidth}
\centering
\includegraphics[height=.7in]{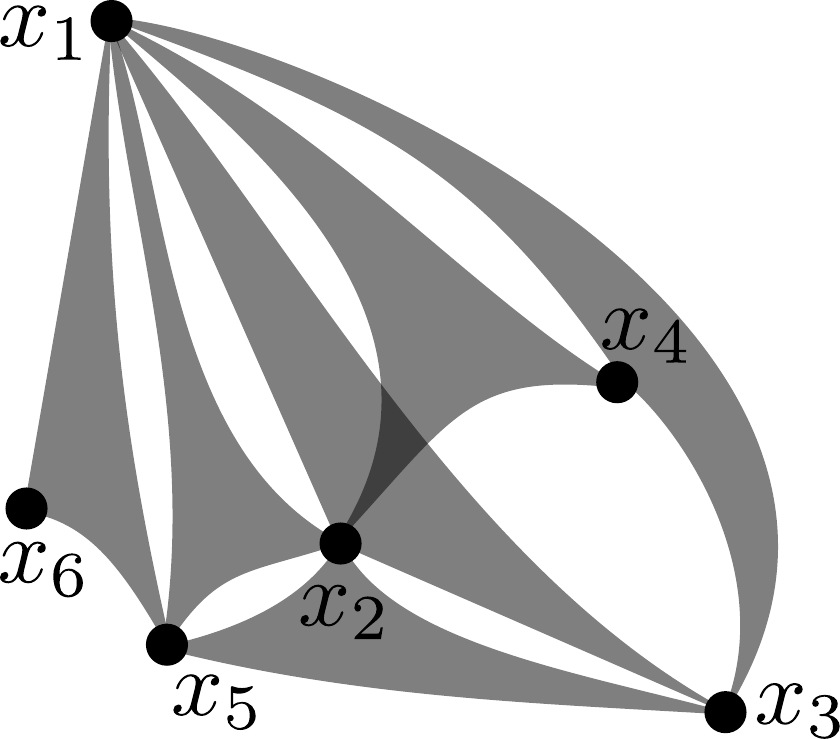}\\
\caption{\\ Clutter  $F'_2$}
\end{minipage}
\hspace{-0.4cm}
\begin{minipage}[h]{0.26\linewidth}
\centering
\includegraphics[height=.7in]{figuras2/tipo1clutter3_3m34.pdf}
\caption{\\ Clutter  $F'_3$}
\end{minipage}
\hspace{-0.5cm}
\begin{minipage}[h]{0.26 \linewidth}
\centering
\includegraphics[height=.7in]{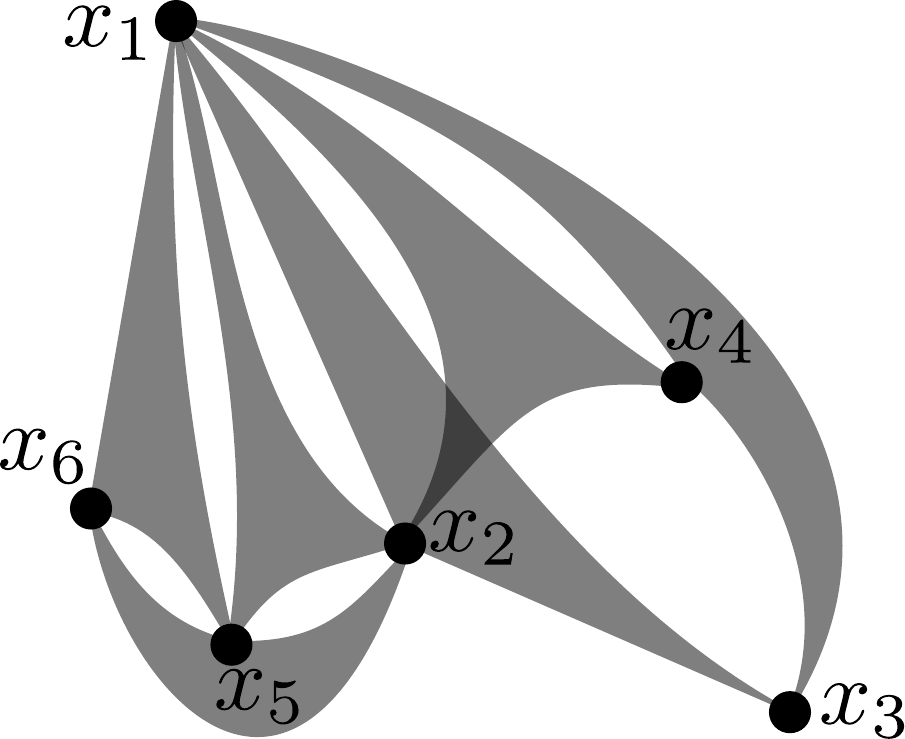}
\caption{\\ Clutter $F'_4$}
\end{minipage}
\end{figure}

\vspace{-.4in}

\medskip

\begin{figure}[h!]
\captionsetup{justification=centering} 
\centering
\hspace{-.3cm}
\begin{minipage}[h]{0.35 \linewidth}
\centering
\includegraphics[height=.7in]{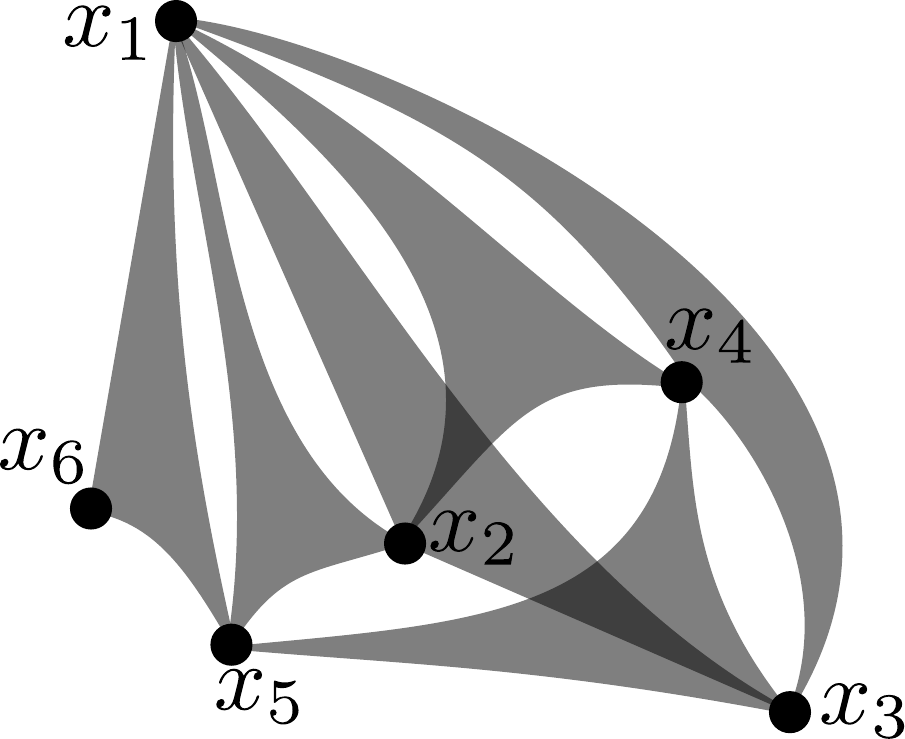}\\
\caption{Clutter $F_{16}$ \\ $x_1F'_3,x_3x_4x_5$}
\end{minipage}
\hspace{-.5cm}
\begin{minipage}[h]{0.35 \linewidth}
\centering
\includegraphics[height=.7in]{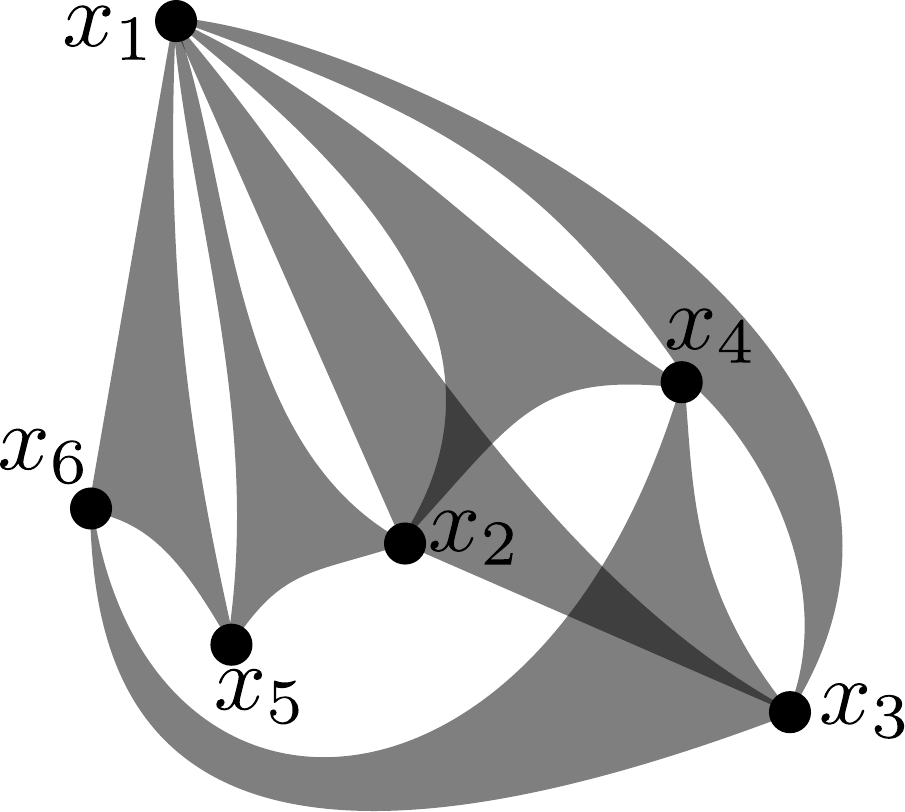}\\
\caption{Clutter $F_{17}$ \\ $x_1F'_3,x_3x_4x_6$}
\end{minipage}
\hspace{-.5cm}
\begin{minipage}[h]{0.35 \linewidth}
\centering
\includegraphics[height=.7in]{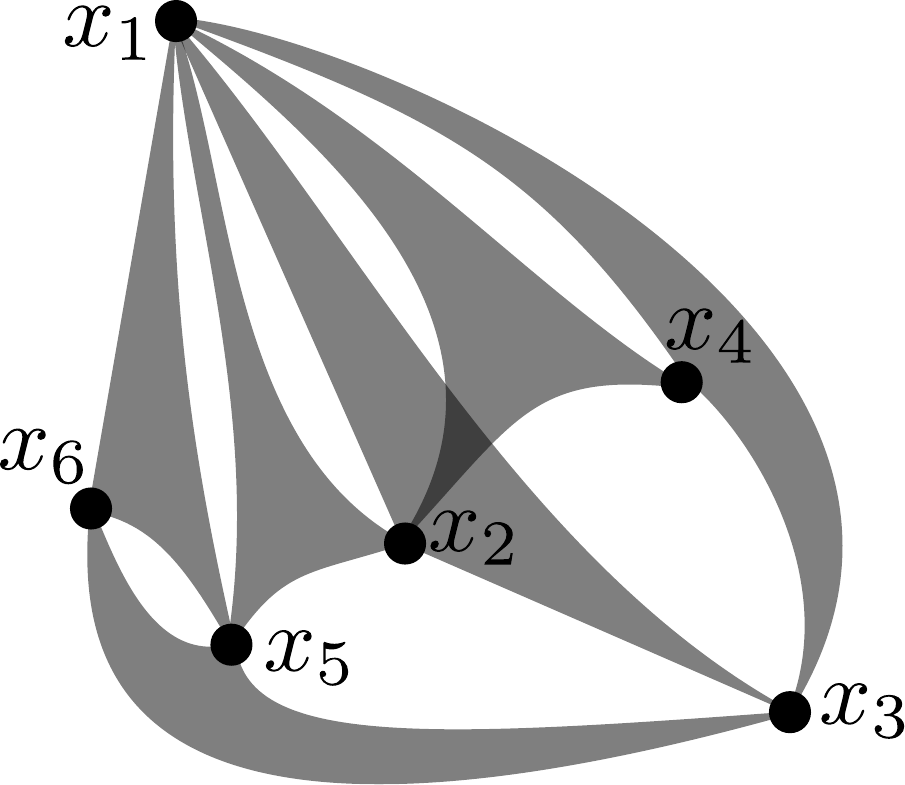}\\
\caption{Clutter $F_{18}$ \\ $x_1F'_3,x_3x_5x_6$}
\end{minipage}
\end{figure} 


\medskip
\vspace{-.5cm}
\newpage
The stabilizer for $F'_4$ is generated by $\beta=(2,3,4,5,6)$.
By the Lemma \ref{isomorfas} we have
\begin{center}
$\begin{array}{c}
\mathcal{O}_{\{x_1F'_4,x_2x_3x_4\}}=\mathcal{O}_{\{x_1F'_4,\beta*(x_2x_3x_4)\}}= \mathcal{O}_{\{x_1F'_4,\beta^2*(x_2x_3x_4)\}}= \\
=\mathcal{O}_{\{x_1F'_4,\beta^3*(x_2x_3x_4)\}}= \mathcal{O}_{\{x_1F'_4,\beta^4*(x_2x_3x_4)\}}, \mbox{ and }\\ \mathcal{O}_{\{x_1F'_4,x_2x_3x_5\}}=\mathcal{O}_{\{x_1F'_4,\beta*(x_2x_3x_5)\}}=\mathcal{O}_{\{x_1F'_4,\beta^2*(x_2x_3x_5)\}}=\\ =\mathcal{O}_{\{x_1F'_4,\beta^3*(x_2x_3x_5)\}}=\mathcal{O}_{\{x_1F'_4,\beta^4*(x_2x_3x_5)\}}.
\end{array}$
\end{center}
The possibilities for $m_4$ are $x_2x_3x_4$ and $x_2x_3x_5$. The associated clutters are:

\medskip

\begin{figure}[!hptb] 
\captionsetup{justification=centering}
\centering
\hspace{-.3cm}
\begin{minipage}[h]{0.29 \linewidth}
\centering
\includegraphics[height=.7in]{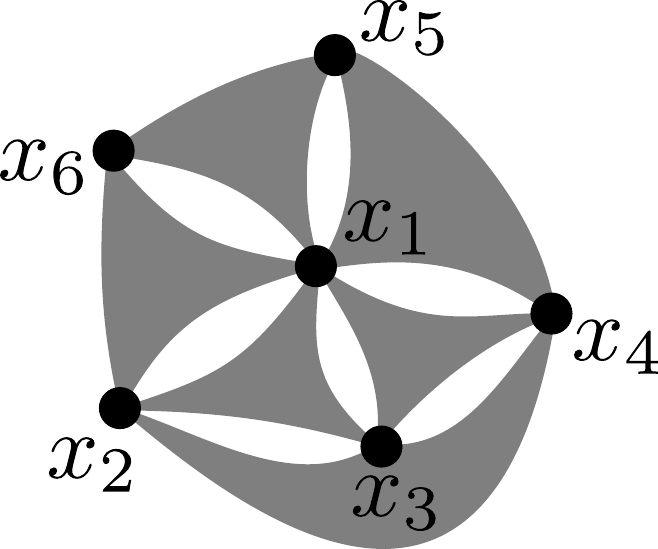}\\
\caption{Clutter $F_{19}$ \\ $x_1F'_4,x_2x_3x_4$}
\end{minipage}
\hspace{1cm}
\begin{minipage}[h]{0.29 \linewidth}
\centering
\includegraphics[height=.7in]{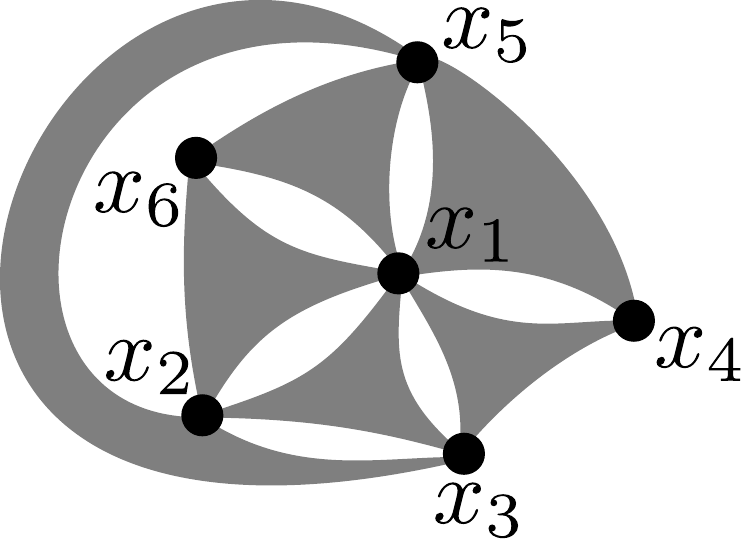}\\
\caption{Clutter $F_{20}$ \\ $x_1F'_4,x_2x_3x_5$}
\end{minipage}
\end{figure} 

\medskip

It is easy to check that these twenty sets define non-equivalent Cremona maps. Indeed, since $F_1, F_2, F_3, F_6, F_7, F_{16}$ are the only with its incidence sequence, they are non-equivalent. Moreover, they are non-equivalent to any other $F_i$ by Lemma \ref{sequencia de grau}.

By Proposition \ref{prop:d3n6t1},  there are four  non equivalent square free Cremona sets of degree $ 3$
with incidence degree $(4,4,4,3,2,1)$.  So, by the Lemma \ref{dual}, four is the number of  Cremona sets  with incidence degree $(5,4,3,2,2,2)$. Since there exists only four  Cremona sets with incidence degree $(5,4,3,2,2,2)$, namely $F_5,F_{10},F_{14}$ and $F_{15}$, they are non-equivalent.  

With the same argument we can prove that  $ F_{11}, F_{17}, F_{18}, F_{19}$ and $F_{20}$ are the only five Cremona sets with incidence degree $(5,3,3,3,2,2)$.

%
The following matrix gives the maximal cones in the remaining cases.
\begin{center}
\begin{tabular}{|c|c|c|c|}
  \hline
  MAXIMAL CONES & \,\, $x_1F'_1$ \,\,  &  \,\, $x_1F'_2$  \,\,& \,\,  $x_1F'_3$ \,\,  \\
  \hline
   $(5,4,3,3,2,1)$  & $F_4$   &  $F_8$, $F_9$   &$F_{12}, \, F_{13}$  \\
\hline
 \end{tabular}
\end{center}
The only possible equivalence could be between $F_8, F_9$ and between $F_{12}, F_{13}$, but it is not the case as one can check directly. 
For instance, its Newton dual has nonequivalent maximal cones. 

\end{proof}

\begin{prop} \label{prop:d3n6t3}
There are $10$ equivalence classes of square free monomial Cremona sets of degree $3$ in
$\K[x_1, \ldots, x_6]$ of type $3$.
\end{prop}

\begin{proof} After a possible reorder of the monomials and relabel of the variables, using Lemma \ref{mdc}, since $A_F$ does not have a row with five 
entries $1$, $F$ contains a subclutter having a maximal cone of the form $C=\{x_1x_2x_3,x_1x_2x_4,x_1g_1, x_1g_2\},$
where $x_1 \nmid g_1,g_2$ and the base of $C$ is a simple graph having $4$ edges and having at most $5$ vertices belonging to $\{x_2,\ldots,x_6\}$. 
There are six of such graphs up to isomorphism:

\begin{figure}[!h]
\captionsetup{justification=centering}
\hspace{0.6in}
\begin{minipage}[h]{0.15 \linewidth}
\begin{center}
\includegraphics[height=.55in]{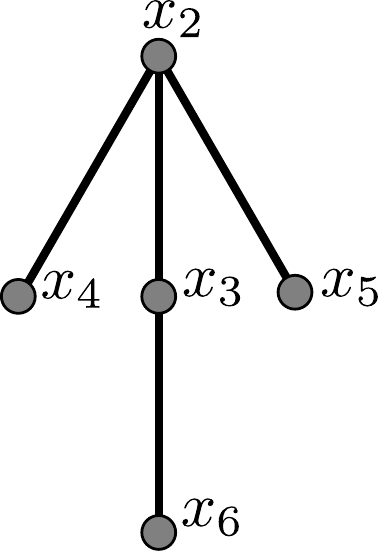}\\
\caption{\\ Base of $C_1$}
\end{center}
\end{minipage}
\hspace{0.6in}
\begin{minipage}[h]{0.18 \linewidth}
\begin{center}

\captionsetup{justification=centering}
\includegraphics[height=0.4in]{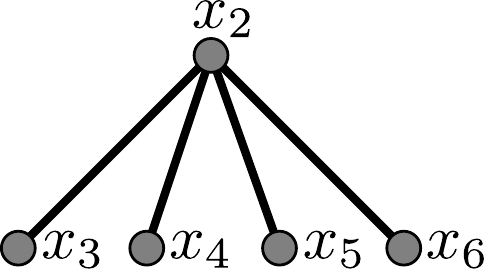}\\
\vspace{0.2in}
\caption{\\\hspace{-.1cm}Base of $C_2$}
\end{center}
\end{minipage}
\hspace{0.15in}
\begin{minipage}[h]{0.15 \linewidth}
\begin{center}
\includegraphics[height=.6in]{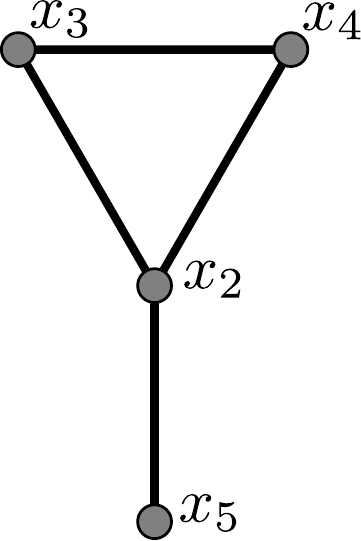}\\
\caption{\\Base of $C_3$}
\end{center}
\end{minipage}
\end{figure}

\vspace{-0.5cm}

\begin{figure}[!h]
\captionsetup{justification=centering}
\hspace{0.6in}
\begin{minipage}[h]{0.15 \linewidth}
\begin{center}
\includegraphics[height=0.55in]{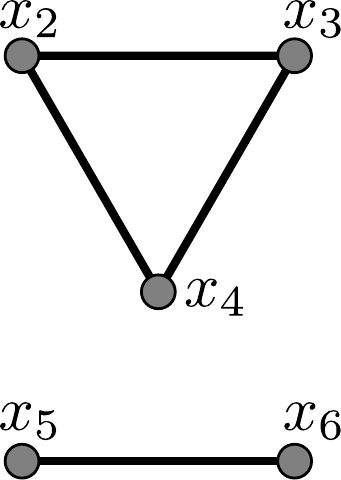}\\
\caption{\\ Base of $C_4$}
\end{center}
\end{minipage}
\hspace{0.6in}
\begin{minipage}[h]{0.15 \linewidth}
\begin{center}
\includegraphics[height=0.55in]{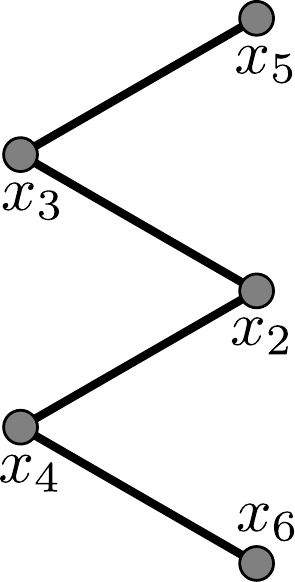}\\
\caption{\\ Base of $C_5$}
\end{center}
\end{minipage}
\hspace{0.6in}
\begin{minipage}[h]{.17 \linewidth}
\begin{center}
\includegraphics[height=0.55in]{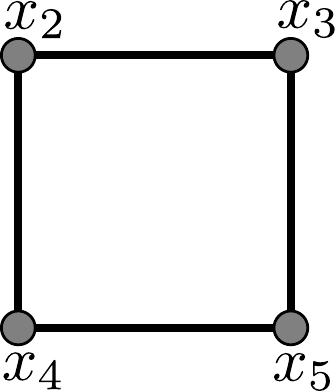}\\
\caption{\\Base of $C_6$}
\end{center}
\end{minipage}
\end{figure}
\medskip

\medskip

Consider $F_i=\{C_i, h_1,h_2\}$ for $i=1 \ldots 6$, where $h_1=x_2^{\alpha_2}x_3^{\alpha_3}x_4^{\alpha_4}x_5^{\alpha_5}x_6^{\alpha_6}$, $h_2=x_2^{\beta_2}x_3^{\beta_3}x_4^{\beta_4}x_5^{\beta_5}x_6^{\beta_6}$, with $\alpha_i,\beta_j\in\{0,1\}.$ One can check that $\det(A_{F_i})=0$ for $i=2,3,6$. 
Therefore, if $F$ is a Cremona set in the hypothesis of the proposition, then either $F=F_1$ or $F=F_4$ or $F=F_5$.

 By elementary operations on the rows of $A_{F_1}$, its possible to show that 
$$\text{det}(A_{F_1})=\pm 3 \Leftrightarrow \alpha_2+\alpha_6=1 \text{ and } \beta_2=\beta_6.$$ 
By using that the stabilizer of $C_1$ is generated by $(4,5)$ and taking account only Cremona sets of type $3$ we determine $5$ square free monomial  sets of degree $3$ having $C_1$ as maximal cone. The associated clutters are:

\begin{figure}[h!] 
\captionsetup{justification=centering}
\centering
\hspace{-.3cm}
\begin{minipage}[h]{0.34 \linewidth}
\centering
\includegraphics[height=.8in]{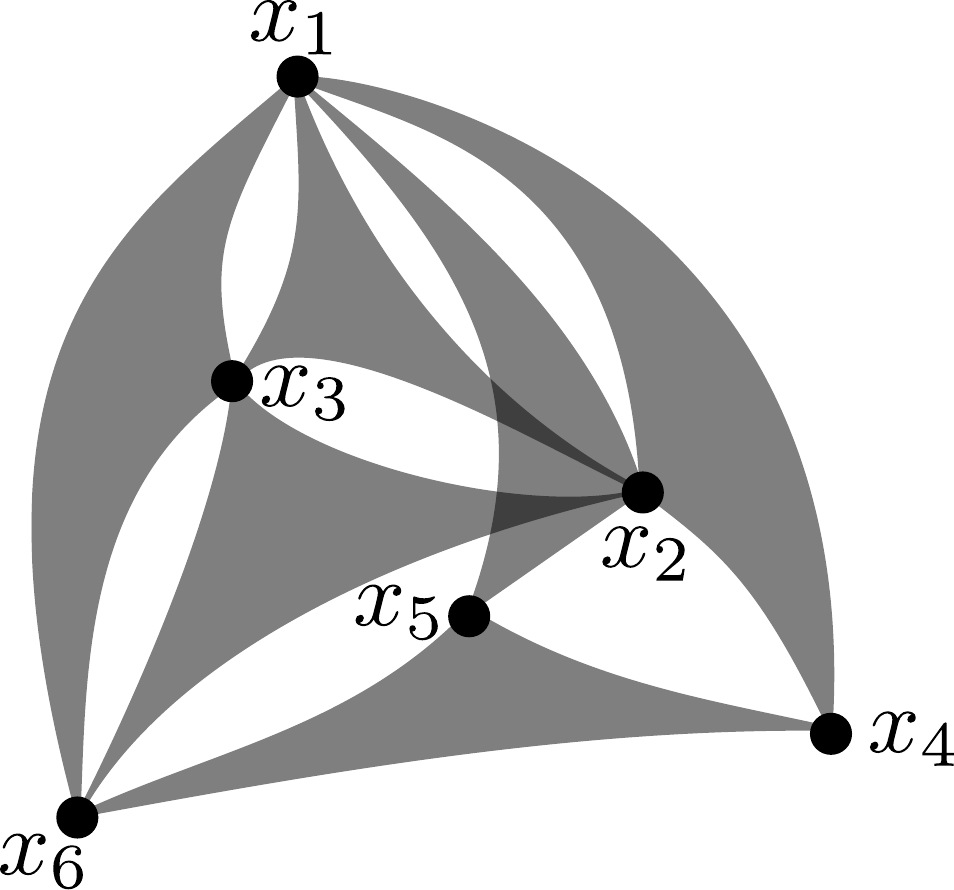}\\
\caption{Clutter $H_1$ \\ $C_1,x_2x_3x_6, x_4x_5x_6$}
\end{minipage}
\hspace{0.1cm}
\begin{minipage}[h]{0.29 \linewidth}
\centering
\includegraphics[height=.8in]{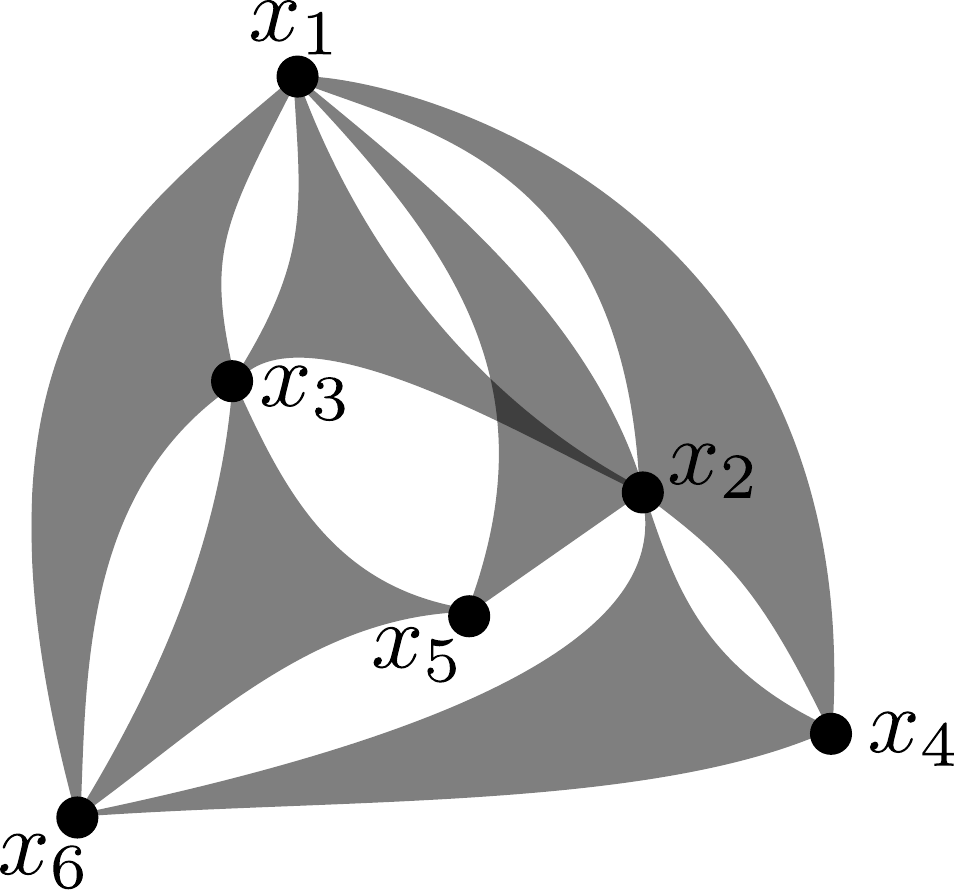}\\
\caption{Clutter $H_2$ \\ $C_1,x_2x_4x_6, x_3x_5x_6$}
\end{minipage}
\hspace{.5cm}
\begin{minipage}[h]{0.29 \linewidth}
\centering
\includegraphics[height=.8in]{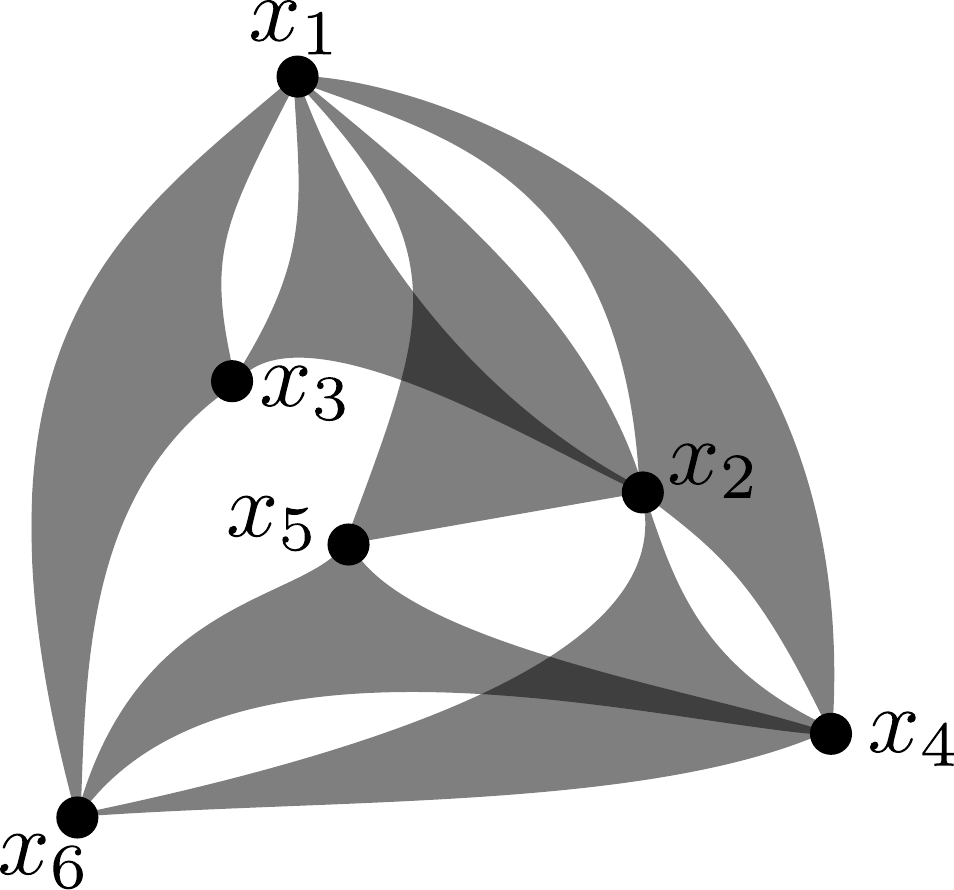}\\
\caption{Clutter $H_3$ \\ $C_1,x_2x_4x_6, x_4x_5x_6$}
\end{minipage}
\end{figure} 


\vspace{-0.5cm}

\begin{figure}[h!] 
\captionsetup{justification=centering}
\centering
\hspace{-.3cm}
\begin{minipage}[h]{0.33 \linewidth}
\centering
\includegraphics[height=.8in]{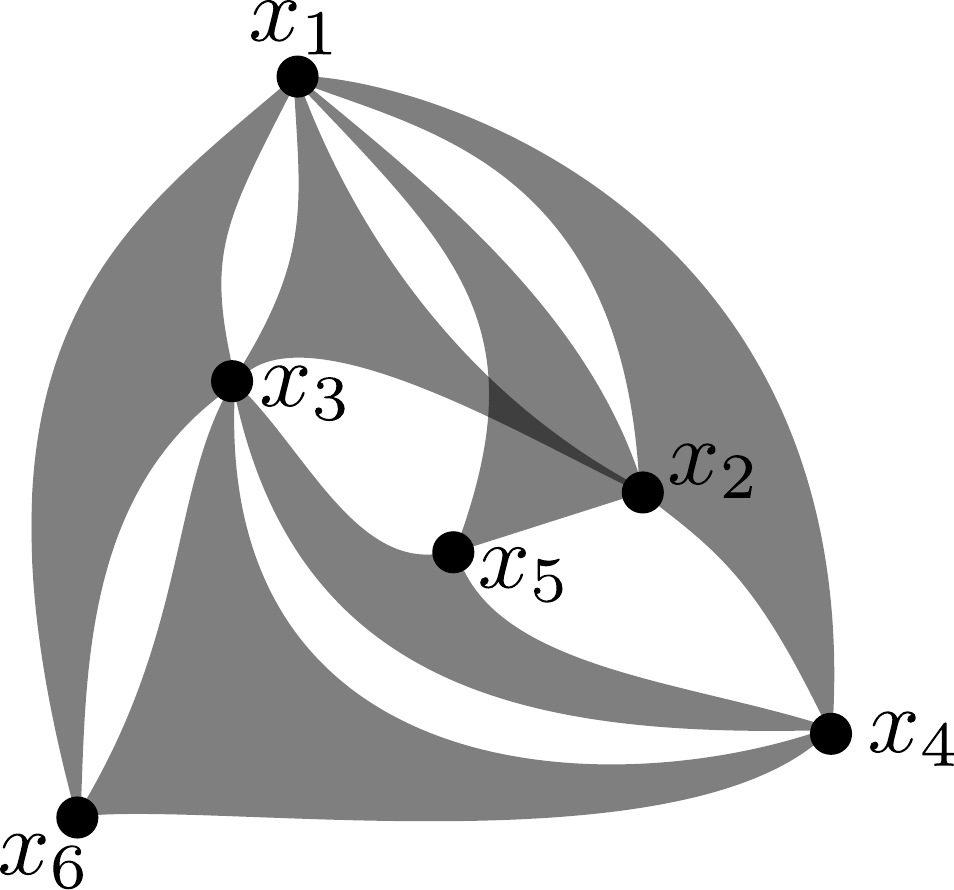}\\
\caption{Clutter $H_4$ \\  $C_1,x_3x_4x_5, x_3x_4x_6$}
\end{minipage}
\hspace{1cm}
\begin{minipage}[h]{0.29 \linewidth}
\centering
\includegraphics[height=.8in]{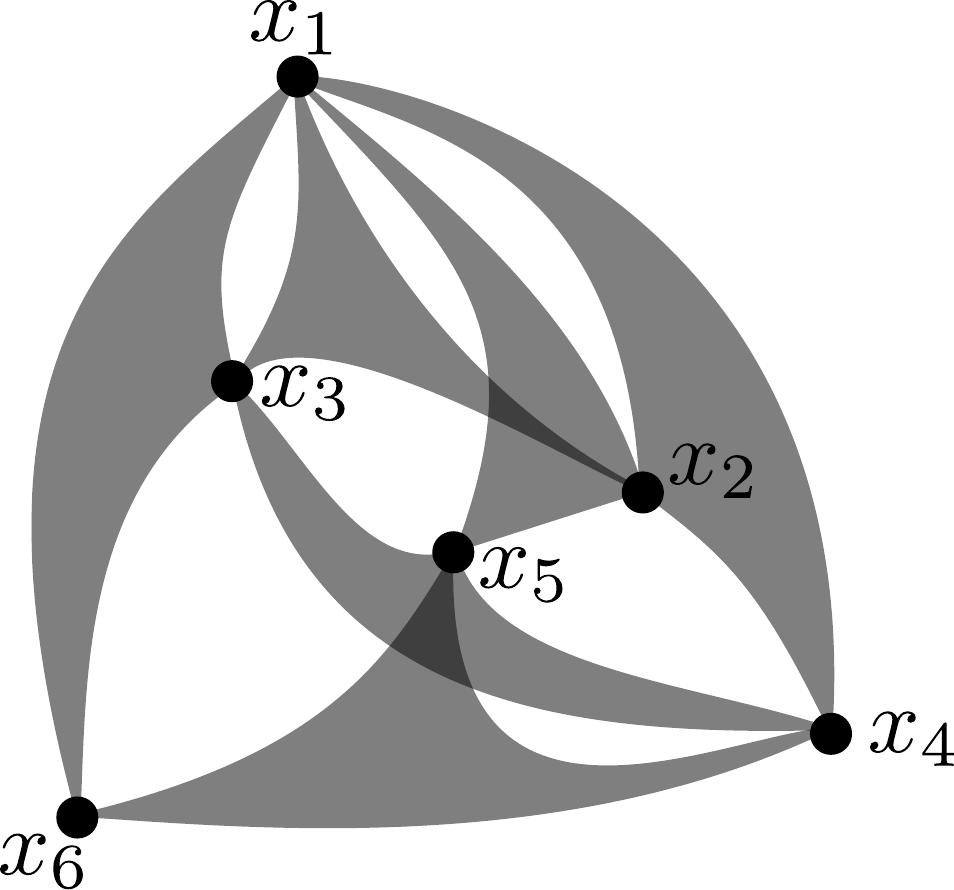}\\
\caption{Clutter $H_5$ \\ $C_1,x_3x_4x_5, x_4x_5x_6$}
\end{minipage}
\end{figure}


Analogously to the previous case, its possible to show that  
$$\text{det}(A_{F_4})=\pm 3 \Leftrightarrow \alpha_5=\alpha_6 \text{ and } \beta_5+\beta_6=1.$$ 
By using that the stabilizer of $C_4$ is generated by $\beta=(2,3)$, $\gamma=(3,4)$, $\delta=(2,4)$ and $\epsilon=(5,6)$,  we have two possible Cremona sets, whose clutters are:

\begin{figure}[h!] 
\captionsetup{justification=centering}
\centering
\hspace{-.3cm}
\begin{minipage}[h]{0.29 \linewidth}
\centering
\includegraphics[height=.7in]{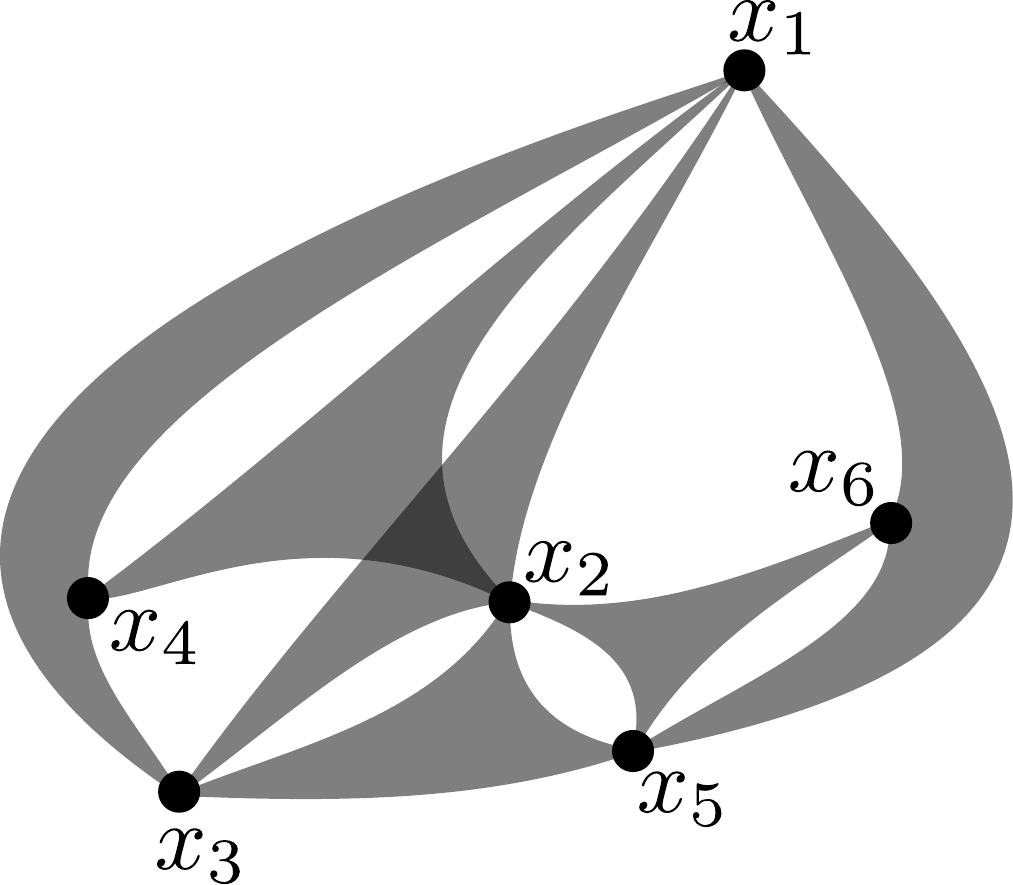}\\
\caption{Clutter $H_6$ \\ $C_4,x_2x_3x_5, x_2x_5x_6$}
\end{minipage}
\hspace{1cm}
\begin{minipage}[h]{0.29 \linewidth}
\centering
\includegraphics[height=.7in]{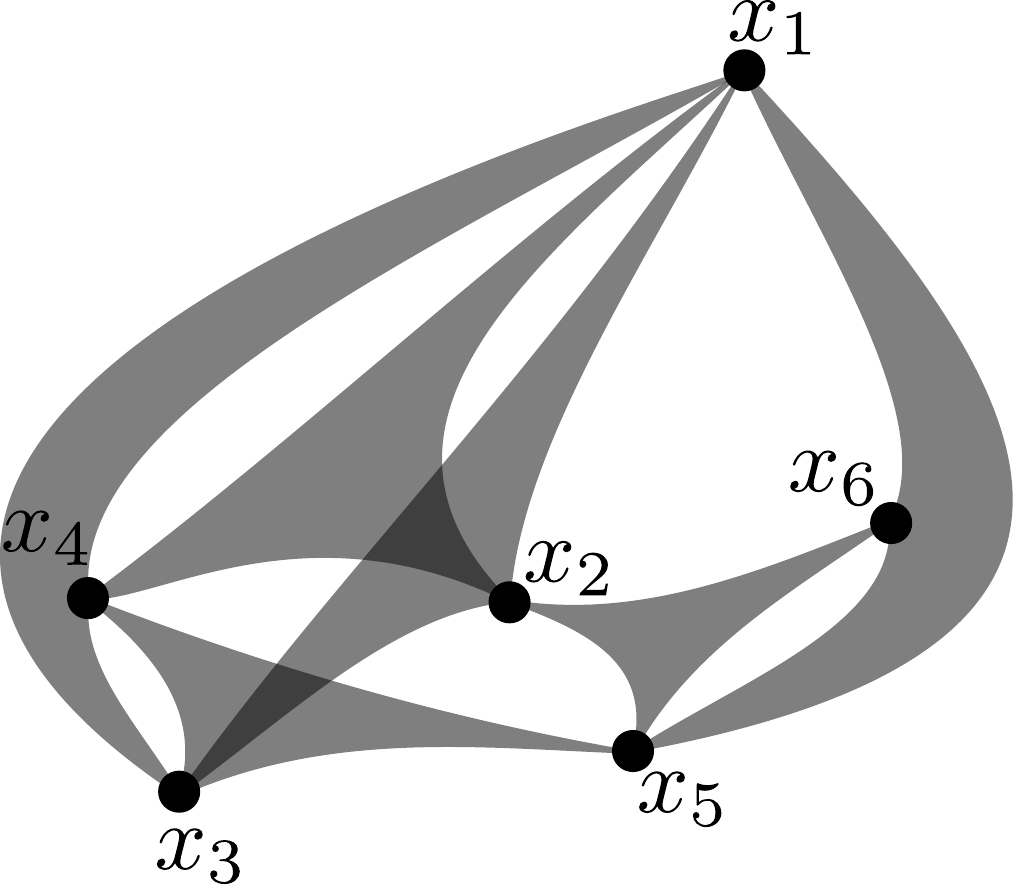}\\
\caption{Clutter $H_7$ \\ $C_4,x_2x_5x_6, x_3x_4x_5$}
\end{minipage}
\end{figure}


\medskip

For the last case, we have that
$$\text{det}(A_{F_5})=\pm 3 \Leftrightarrow \alpha_3=\alpha_4 \text{ and } \beta_3+\beta_4=1.$$ 
By using that the stabilizer of $C_5$ is generated by $\beta=(3,4)(5,6)$, we obtain $6$ Cremona sets. The associated clutters are the following ones:

\begin{figure}[h!] 
\captionsetup{justification=centering}
\centering
\hspace{0cm}
\begin{minipage}[h]{0.29 \linewidth}
\centering
\includegraphics[height=.7in]{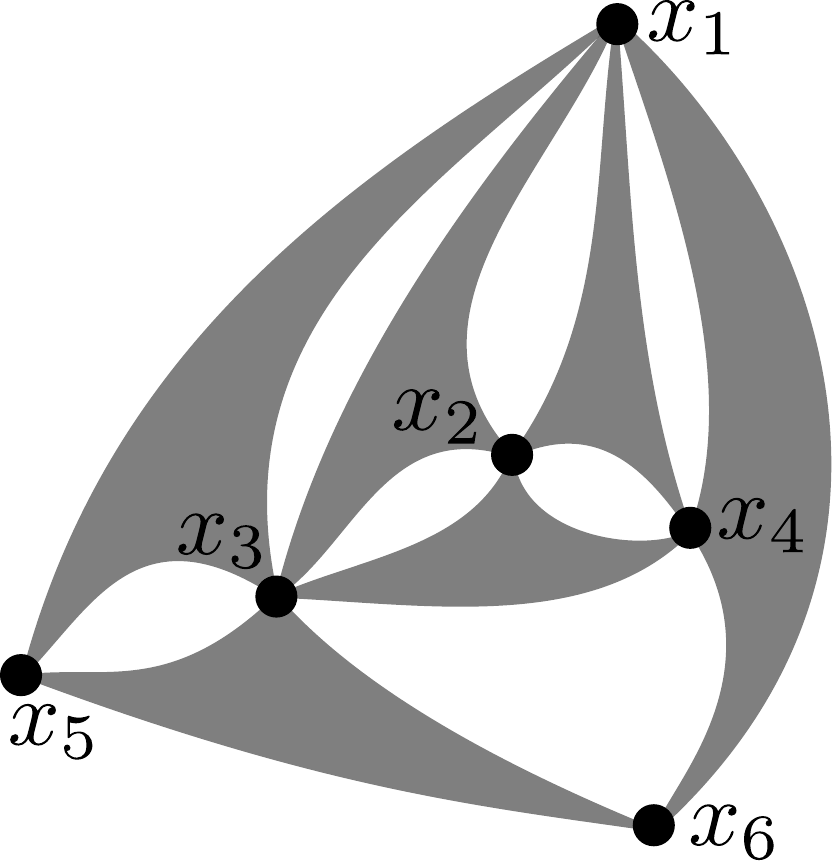}\\
\caption{Clutter $H_8$ \\ $C_5,x_2x_3x_4, x_3x_5x_6$}
\end{minipage}
\hspace{0cm}
\begin{minipage}[h]{0.29 \linewidth}
\centering
\includegraphics[height=.7in]{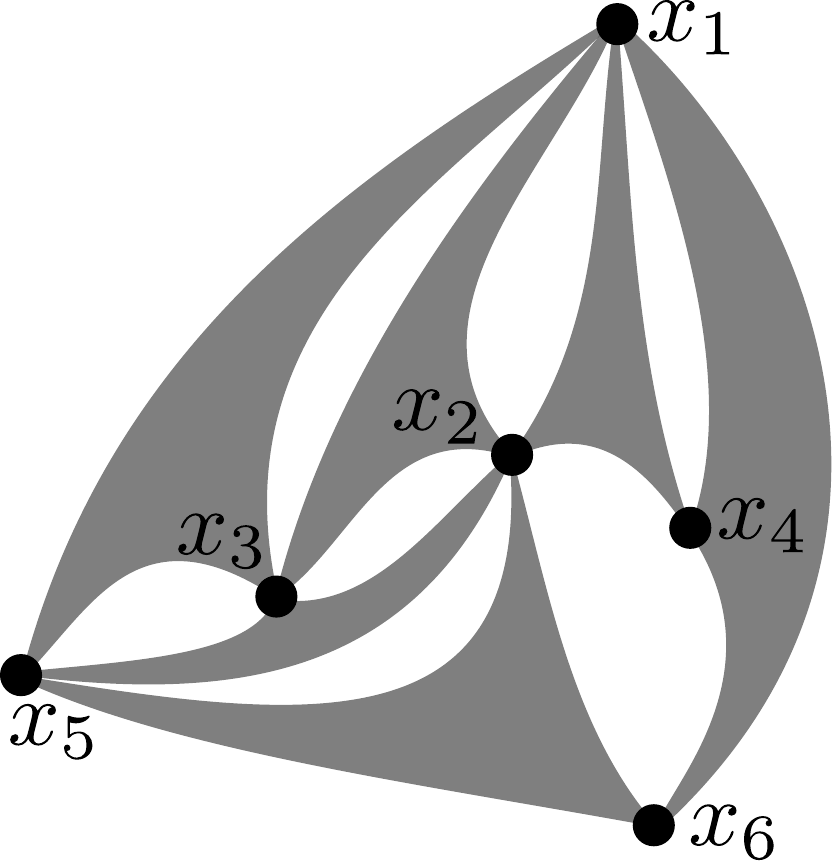}\\
\caption{Clutter $H_9$ \\ $C_5,x_2x_5x_6, x_2x_3x_5$}
\end{minipage}
\hspace{0cm}
\begin{minipage}[h]{0.29 \linewidth}
\centering
\includegraphics[height=.7in]{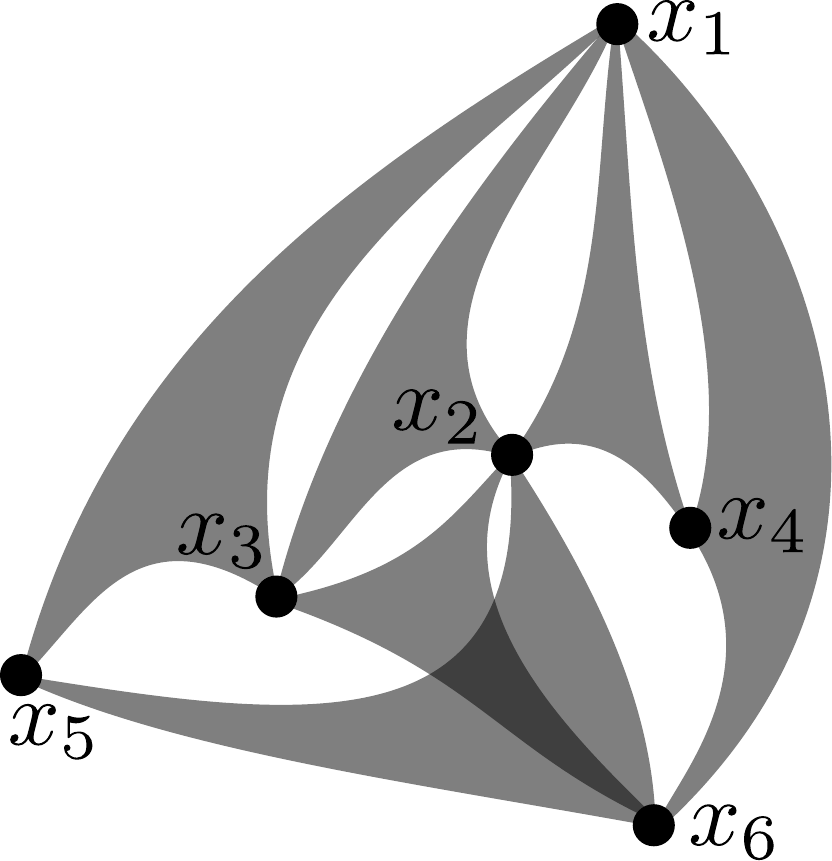}\\
\caption{Clutter $H_{10}$ \\ $C_5,x_2x_5x_6, x_2x_3x_6$}
\end{minipage}
\end{figure} 


\medskip
\begin{figure}[h!] 
\captionsetup{justification=centering}
\centering
\hspace{0cm}
\begin{minipage}[h]{0.29 \linewidth}
\centering
\includegraphics[height=.7in]{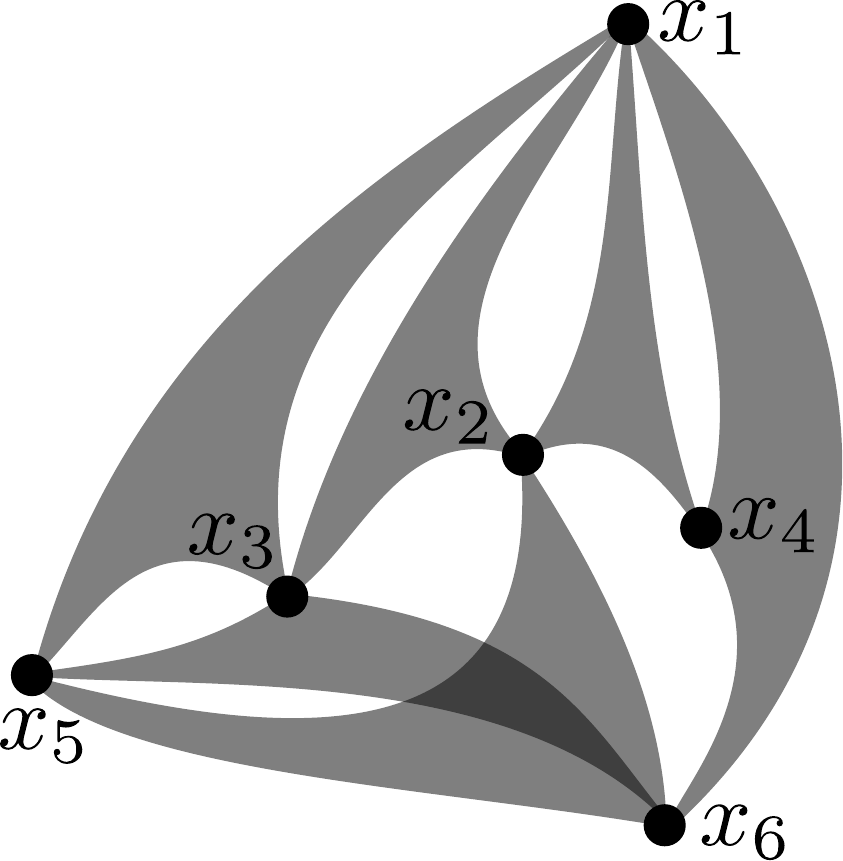}\\
\caption{Clutter $H_{11}$ \\ $C_5,x_2x_5x_6, x_3x_5x_6$}
\end{minipage}
\hspace{0cm}
\begin{minipage}[h]{0.29 \linewidth}
\centering
\includegraphics[height=.7in]{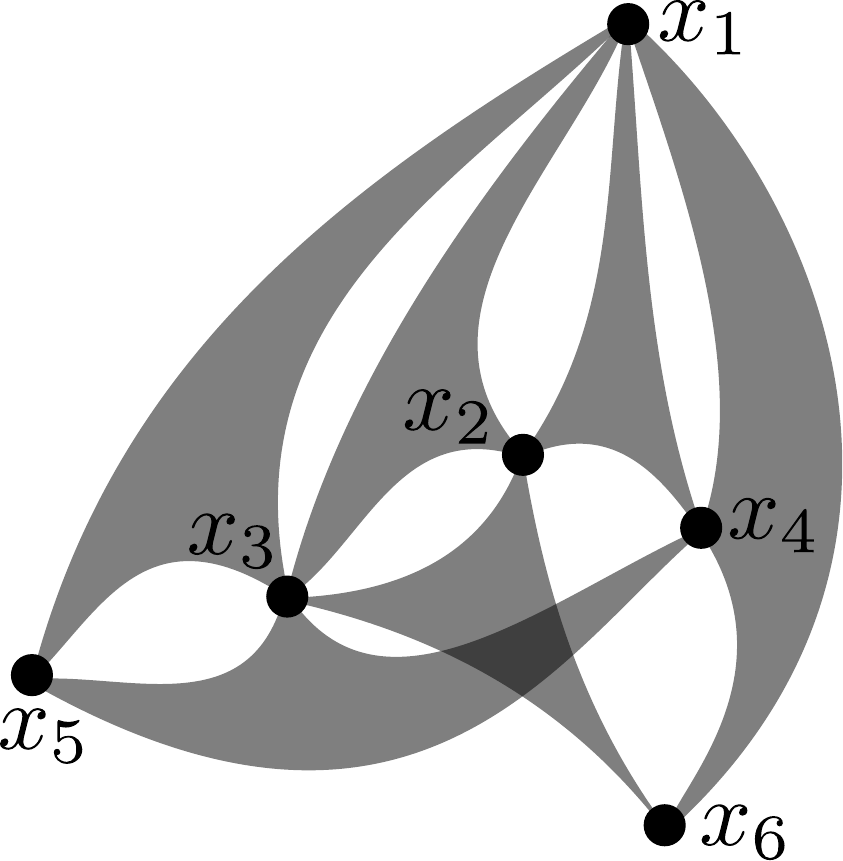}\\
\caption{Clutter $H_{12}$ \\ $C_5,x_3x_4x_5, x_2x_3x_6$}
\end{minipage}
\hspace{0cm}
\begin{minipage}[h]{0.29 \linewidth}
\centering
\includegraphics[height=.7in]{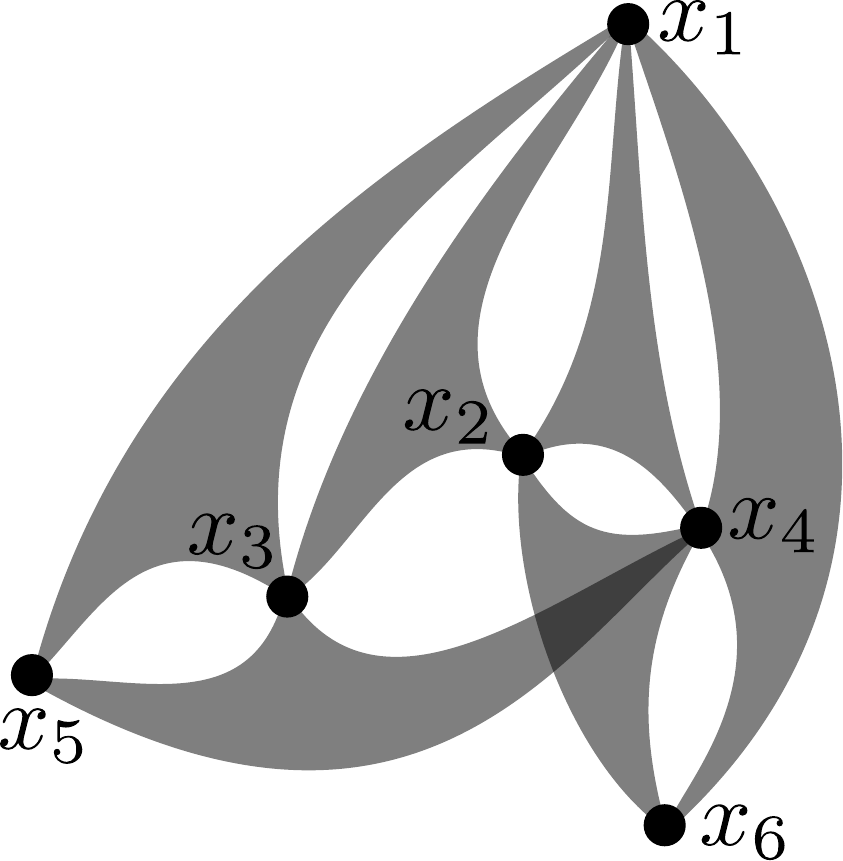}\\
\caption{Clutter $H_{13}$ \\ $C_5,x_3x_4x_5, x_2x_4x_6$}
\end{minipage}
\end{figure} 


\medskip

\begin{figure}[h!] 
\captionsetup{justification=centering}
\centering
\hspace{-.3cm}
\begin{minipage}[h]{0.29 \linewidth}
\centering
\includegraphics[height=.7in]{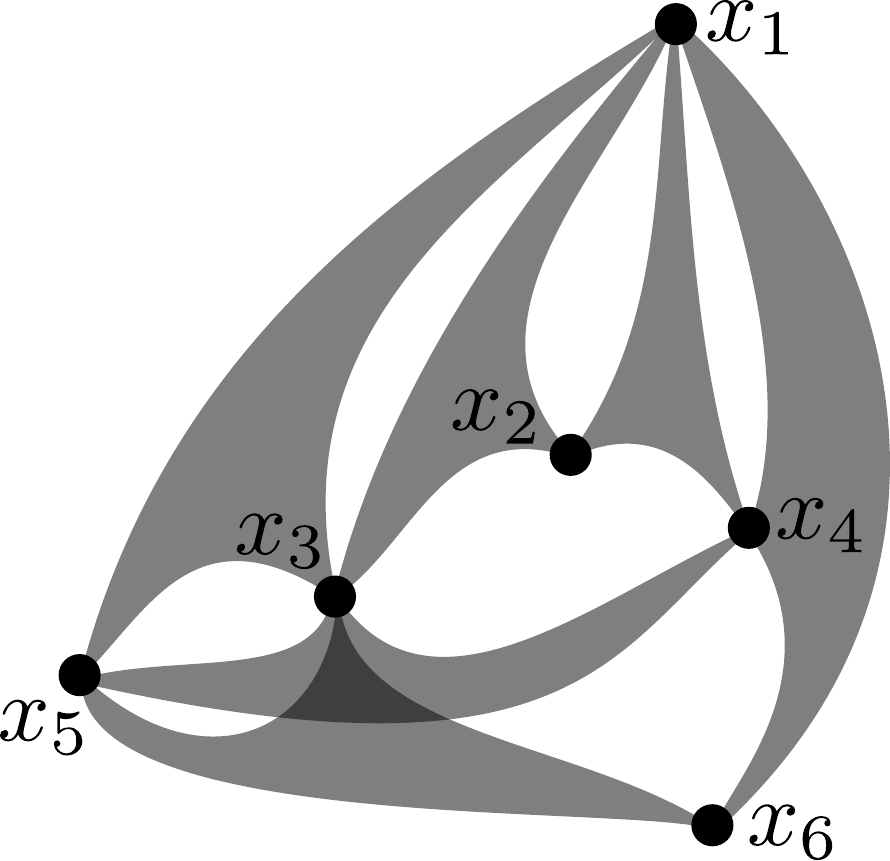}\\
\caption{Clutter $H_{14}$ \\ $C_5,x_3x_4x_5, x_3x_5x_6$}
\end{minipage}
\hspace{3cm}
\begin{minipage}[h]{0.29 \linewidth}
\centering
\includegraphics[height=.7in]{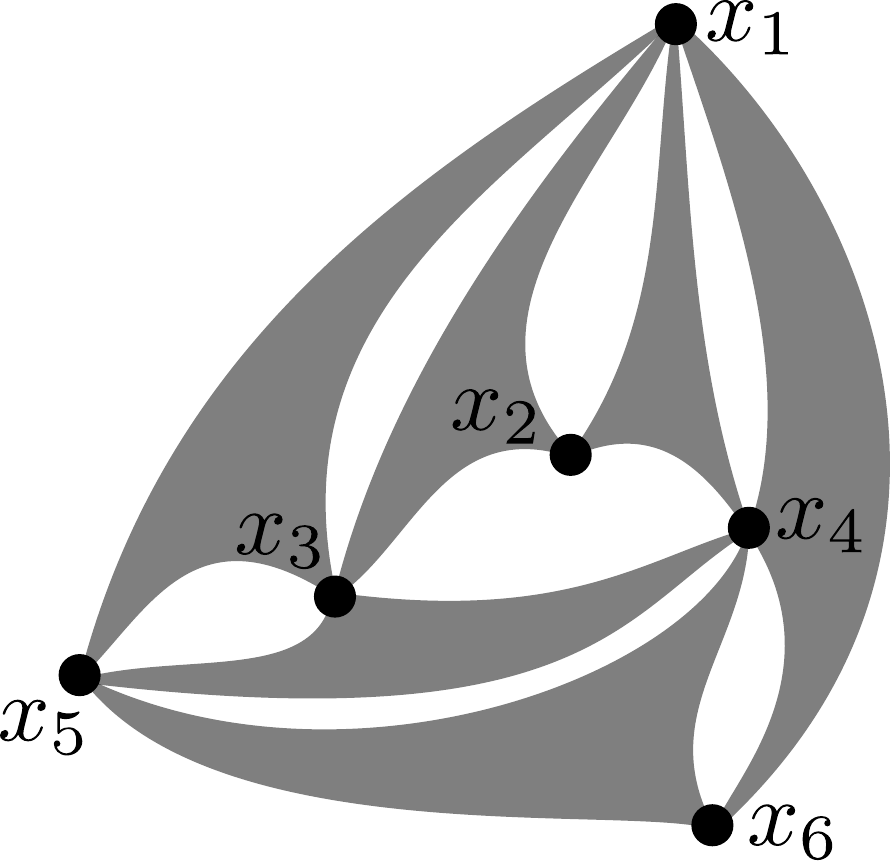}\\
\caption{Clutter $H_{15}$ \\ $C_5,x_3x_4x_5, x_4x_5x_6$}
\end{minipage}
\end{figure} 
\medskip

\vspace{0.5cm}
We have the following isomorphisms:
\begin{enumerate}
\item $(1,2)(3,4)*H_2=H_3$
\item $(1,3)(2,5,6)*H_4=H_{14}$
\item $(1,2,3)(4,6,5)*H_{10}=H_{12}$
\item $(1,2,3,5,6,4)*H_6=H_{13}$
\item $(1,3,2)(4,5)*H_8=H_9$
\end{enumerate}

Using the Lemma \ref{dual}, \ref{cone} and \ref{sequencia de grau}
one can easily check that $H_1$, $H_2$, $H_4$, $H_5$, $H_6$, $H_7$, $H_8$, $H_{10}$, $H_{11}$ and $H_{15}$ are non equivalent. The result follows. 
\end{proof}

\section{An algorithm}
In this section, we describe two pseudo-codes that allow us to determine the number of monomial square-free Cremona maps of degree $d$, $\Phi:\P^{n-1} \dashrightarrow \P^{n-1}$. We work with the exponent vectors instead of the monomials.

\subsection{The routine  build-next $\mathcal{M}^{i}_{n,d}$}
We describe the routine \verb|build_next_i_M_n_d| that construct $\mathcal{M}^{i+1}_{n,d}$ from $\mathcal{M}^{i}_{n,d}$.
\begin{itemize}
 \item Input $\verb|(n,d,i,i_M_n_d)|$
 $\verb|i_M_n_d|$ is a list with the representatives of maximal rank of the orbits of the 
 natural action $\ast:  S_n \times \mathcal{M}^i_{n,d}   \rightarrow  \mathcal{M}^i_{n,d}$ defined in \ref{def_action}.
  \item If $i=1$\\
 return $\mathcal{M}_{n,d}$. 
  \item If $i> 1$\\
  Create a list $L$ indexed by $\mathcal{M}^i_{n,d}$ such that each entry of $L$ is a copy 
 of $\mathcal{M}_{n,d}$.\\
 For each $F \in \mathcal{M}^i_{n,d}$, as soon as $L_F \neq \emptyset$, take the first 
 element $f \in L_F$. By using Lemma \ref{isomorfas} to the pair ${F,f}$, elminate all 
 $g \in L_{F}$ such that $\mathcal{O}(F,f)=\mathcal{O}(F,g)$. In this way $L$ became a list such that all orbits of $\mathcal{M}^{i+1}_{n,d}$ are 
 $\{F,g_{F}\}_{F\in \mathcal{M}^i_{n,d}, g \in L_F}$.  
  \item Use the lists $\mathcal{M}^{i}_{n,d}$ and $L$ to create the smalest list, $temp(\mathcal{M}^{i}_{n,d})$, containing all $\{F,g_{F}\}_{F\in \mathcal{M}^i_{n,d}, g \in L_F}$.
 By Lemma \ref{processo} all representastives of orbits of $\mathcal{M}^{i+1}_{n,d}$ belong to  $temp(\mathcal{M}^{i}_{n,d})$. 
  \item By using Lemma \ref{sequencia de grau}, classify these representatives according to their sequences of incidence degree, putting each class on a list which is the output.   \end{itemize}
 \subsection{ The routine refresh $temp(\mathcal{M}^{i}_{n,d})$} 
 
 We describe \verb|refresh_temp_i_M_n_d|, the routine that removes distinct representatives of the same orbit in order to 
 construct the smalest $temp(\mathcal{M}^{i}_{n,d})$. 
  \begin{itemize}
  \item Input A list of Log matrices $LOG$;
  \item While $LOG$ is not empty, take the first element $LOG[0]$ of $LOG$ compute the orbit $\mathcal{O}(LOG[0])$ and return $LOG \setminus \mathcal{O}(LOG[0])$.
 \end{itemize}

\begin{figure}[h!]
\centering
\includegraphics[width=8cm]{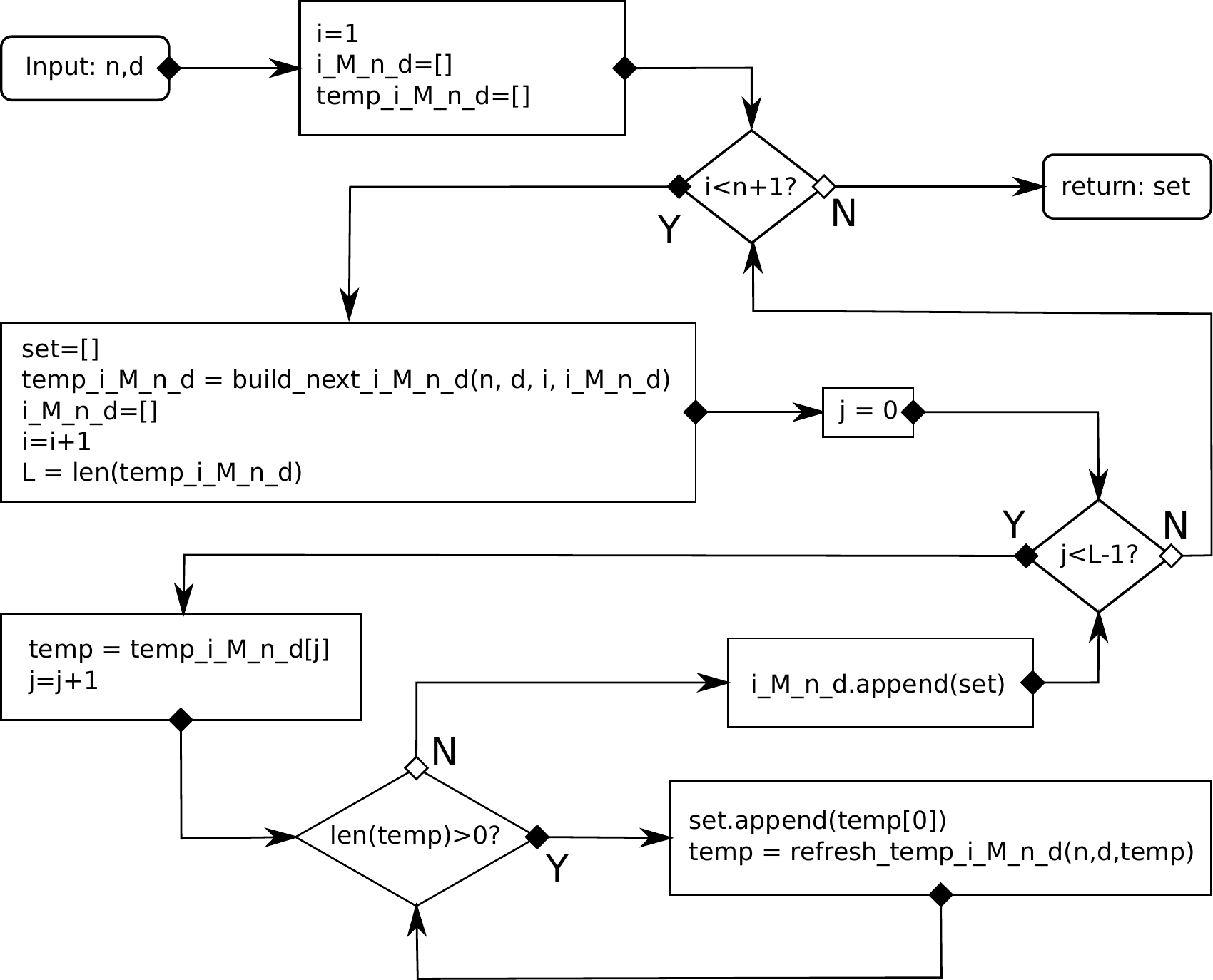}
\caption{The fluxogram}
\end{figure}

The algorithm has been implemented in the Sagemath software. The routine   is available in the website
\url{<https://github.com/ricardonmachado/Cremona>}. The algorithm counts the number of equivalence classes of square free Cremona Monomials of degree $d$ in $\K[x_1,...,x_n]$. We run it for some  choices of $n$ e $d$.  
\begin{center}
\begin{tabular}{|c|c|c|}
\hline
n&d&Number of Cremonas\\
\hline
4 &2& 1\\
\hline
5& 2 &4\\
\hline
6& 2 &8\\
\hline
6 &3 &40\\
\hline
7 &2 &23\\
\hline
7 &3 &674\\
\hline
\end{tabular}
\end{center}

To run the case $n = 7$ and $d = 3$, we used a computer with a 3.4 Ghz I7 processor and 24Gb ram, and it took about $10$ days.

\vspace{0.5cm}


\bibliography{sn-bibliography}


\end{document}